\documentclass[12pt]{article}
\usepackage{amsmath}
\usepackage{graphicx}
\usepackage{enumerate}
\usepackage{natbib}
\usepackage{url} % not crucial - just used below for the URL 
\usepackage{lscape} 

\usepackage[utf8]{inputenc}
\usepackage{amsthm}
\usepackage{amssymb}
\usepackage{color}
\usepackage{comment}
\usepackage{appendix}
\usepackage{bm}
\usepackage{titlesec}
\titlelabel{\thetitle.\quad}

\definecolor{darkRed}{rgb}{0.60,.03,.03}

\usepackage{tikz}
\usetikzlibrary{shapes,snakes, graphs}

\usetikzlibrary{arrows,%
 petri,%
 topaths}%
\usepackage{pgf}
\usepackage{subfigure,tikz}
\usetikzlibrary{arrows,automata}
\usepackage{svg}

\usepackage%[backref=page]
{hyperref}
\hypersetup{
 colorlinks=true,
 linkcolor=blue,
 filecolor=magenta,
 citecolor=darkRed,
 urlcolor=cyan,
 pdftitle={Optimal Tests of the Composite Null Hypothesis},
 }
 \usepackage{float}
 \usepackage{tabularx}

\newcommand{\blind}{1}

\newtheorem{theorem}{Theorem}

\newtheorem{corollary}{Corollary}
\newtheorem{lemma}{Lemma}
\def\ci{\mbox{\ensuremath{\perp\!\!\!\perp}}}

\newcommand{\bz}{\bar{z}}
\newcommand{\zb}{\underline{z}}
\newcommand{\1}{\textbf{1}}

\allowdisplaybreaks[1]

\begin{document}

\def\spacingset#1{\renewcommand{\baselinestretch}%
{#1}\small\normalsize} \spacingset{1}

%%%%%%%%%%%%%%%%%%%%%%%%%%%%%%%%%%%%%%%%%%%%%%%%%%%%%%%%%%%%%%%%%%%%%%%%%%%%%%

\if1\blind { \title{Optimal Tests of the Composite Null Hypothesis Arising in
 Mediation Analysis} \author{Caleb H.~Miles, Antoine Chambaz \thanks{Caleb
 H. Miles is Assistant Professor, Department of Biostatistics, Columbia
 University, New York, NY 10032, USA (email: cm3825@cumc.columbia.edu); and Antoine Chambaz is Professor of
 Applied Mathematics, Université Paris Cité, CNRS, MAP5, F-75006 Paris, France.}\hspace{.2cm}} %The authors gratefully acknowledge... This work was funded, in part, by... The contents are solely the responsibility of the authors and do not represent the official views of the funding institutions.}\hspace{.2cm}}
 \date{}
 \maketitle
 %\bigskip
} \fi

\if0\blind
{
 \bigskip
 \bigskip
 \bigskip
 \begin{center}
 {\LARGE\bf Optimal Tests of the Composite Null Hypothesis Arising in Mediation Analysis}
\end{center}
 \medskip
} \fi

\bigskip

\begin{abstract}
\noindent The indirect effect of an exposure on an outcome through an intermediate variable can be identified by a product of two regression coefficients under certain causal and regression modeling assumptions. In this context, the null hypothesis of no indirect effect is a composite null hypothesis, as the null holds if either regression coefficient is zero. A consequence is that traditional hypothesis tests are severely underpowered near the origin (i.e., when both coefficients are small with respect to standard errors). We propose hypothesis tests that (i) preserve level alpha type~1 error, (ii) meaningfully improve power when both true underlying effects are small relative to sample size, and (iii) preserve power when at least one is not. One approach gives a closed-form test that is minimax optimal with respect to local power over the alternative parameter space. Another uses sparse linear programming to produce an approximately optimal test for a Bayes risk criterion. We discuss adaptations for performing large-scale hypothesis testing as well as modifications that yield improved interpretability. We provide an R package that implements our proposed methodology.
\end{abstract}
\noindent%
 \vfill

\newpage
\spacingset{1.5} % DON'T change the spacing!
\section{\centering INTRODUCTION}
\label{sec:intro}

Mediation analysis is a widely popular discipline that seeks to understand the
mechanism by which an exposure affects an outcome by learning about
intermediate events that transmit part or all of its effect. For instance, if one is interested in the effect of an exposure $A$ on an outcome $Y$, one
might posit that at least in part, $A$ affects $Y$ by first affecting some
intermediate event $M$, which in turn affects $Y$. %The effect along such a causal pathway is known as an indirect effect. 
%Mediation analysis has a long history dating back to the 1920s with the path analysis work of Sewall Wright \citep{wright1921correlation}, and has been applied across a vast array of fields, such as economics, epidemiology, medicine, psychology, and sociology. In fact, one foundational article \citep{baron1986moderator} currently has well over 100,000 citations. In more recent decades, \citet{robins1992identifiability} and \citet{pearl2001direct} have grounded definitions of mediated effects in the causally-explicit language of counterfactuals, which has lead to an explosion of mediation research in the field of causal inference (see \citet{vanderweele2015explanation} for an overview). However, 
Despite its long
history and broad application, with an exception from recent work
\citep{van2024nearly}, existing hypothesis tests of a
single mediated effect are overly conservative and underpowered
in a certain region of the alternative hypothesis space. %, or do not preserve type~1 error uniformly over the null hypothesis space. 
In this article, we
demonstrate the reason for this suboptimal behavior, and develop new
hypothesis tests that deliver optimal power in some
decision theoretic senses while preserving type~1 error uniformly over the composite null subspace of the parameter space. We focus on the case in which the indirect effect is identified by a product
of two coefficients or functions of coefficients that can be estimated with
uniform joint convergence to a bivariate normal distribution with diagonal covariance matrix. In fact, our
approach is not limited to the realm of mediation analysis, but can be applied
to test for any product of two coefficients that are estimable in the above
sense. %We also discuss extensions to the setting in which one is testing the product of more than two coefficients being equal to
% zero in the Supplementary Materials.

Briefly, the inferential problem that arises in mediation analysis is that
when the mediated effect is identified by a product of two coefficients, even
while the joint distribution of estimators of these coefficients may converge
uniformly to a bivariate normal distribution, neither their product nor the minimum of their absolute values will converge uniformly to a Gaussian law. Thus, asymptotic approximation-based univariate test statistics that are
a function of either of these summaries of the two coefficient estimates will have poor
properties in an important region of the parameter space for any given finite
sample. %Further, another popular test based on testing each coefficient separately is underpowered near the origin of the parameter space (i.e., where both coefficients are zero) by virtue of taking the intersection of the rejection regions of these two tests. This is a consequence of the nonstandard geometry of this nonconvex composite null hypothesis space.

Previously, \citet{mackinnon2002comparison} and \citet{barfield2017testing}
presented %results from 
simulation studies comparing
hypothesis tests of a mediated effect. Both articles
observed that all existing tests are either underpowered in some scenarios or
do not preserve nominal type~1 error in all scenarios. Both articles found
the so-called joint significance test to be the best overall performing test,
while \citet{barfield2017testing} found a bootstrap-based test to have
comparable performance across some scenarios. %, with the exception of the setting with a rare binary outcome. 
\citet{huang2019genome} proved that the
joint significance test yields a smaller $p$-value than those of both
normality-based tests and normal product-based tests. %While others have conducted simulation studies comparing standard error estimators of a mediated effect,
%\citep{stone1990robustness, mackinnon1993estimating,
 %mackinnon1995simulation}. these are closely related to, yet distinct fromhypothesis testing, which is our focus in this article.
\citet{barfield2017testing} highlighted the problem of large-scale
hypothesis testing for mediated effects, which has
become a popular inferential task in genomics, where genomic measures such
as gene expression or DNA methylation are hypothesized to be potential
mediators of the effect of genes or other exposures on various outcomes. In
fact, the conservative behavior of mediation hypothesis
tests is most evident in large-scale hypothesis testing,
since one can compare the distribution of $p$-values with the uniform
distribution. \citet{huang2019genome} showed
that the joint significance and delta method tests both generate
$p$-values forming a distribution that stochastically dominates the uniform
distribution. However, while this behavior is readily
observable in large-scale hypothesis testing, the problem
is no different for a single hypothesis test; it is
merely less obvious.

Recently, \citet{huang2019genome}, \citet{dai2022multiple}, \citet{liu2022large}, and \cite{du2023methods} developed
large-scale hypothesis testing procedures for mediated
effects that account for the conservativeness of standard tests. The latter compared these methods as well as traditional methods in a large simulation study. The method
proposed by \citet{huang2019genome} is based on a distributional assumption
about the two coefficients among all distributions being tested that satisfy
the null hypothesis. %By contrast, 
The methods of
\citet{dai2022multiple}, \citet{liu2022large}, and \cite{du2023methods} involve estimating the proportions of three
components of the composite null hypothesis: when both coefficients are zero,
when one is zero and the other is not, and vice versa. \citet{dai2022multiple} and \cite{du2023methods} then relaxed the
threshold of the joint significance test and delta method-based test, respectively, based on these estimated proportions. \citet{liu2022large} used the estimated proportions to construct a new test statistic based on $p$-values attained under each of these components of the null hypothesis.
Our proposed methods resolve the conservativeness issue across the entire null
hypothesis space even for a single test, hence they require neither assumptions
about the distributions of the coefficients nor estimation of proportions of
components of the null hypothesis space. The latter is important, because in finite samples, there is no single distribution of a test statistic under the sub-case of the null hypothesis in which one parameter is null and the other is not. The distribution of a test statistic in this sub-case can vary from being well-approximated by the distribution when one parameter is zero and the other is infinite to being well-approximated by the distribution when both parameters are zero. %(indeed, this sub-case contains parameter values arbitrarily close to (0,0)). 
Instead, we can simply extend our
method to the large-scale hypothesis test setting while
adjusting for multiple testing using standard Bonferroni or
Benjamini--Hochberg corrections.

\citet{van2024nearly} is the most closely related work to this article. The authors 
%take a similar approach to ours, and 
develop a test for a mediated effect
based on defining a rejection region in $\mathbb{R}^2$. In their original version \citep{van2020almost}, they proved that a
similar test does not exist within a certain class of tests. In this article,
we show that a similar test does in fact exist in a slightly expanded class of
tests, and that it is the unique such test in this class (in the more recent version, the authors now recognize the existence of this test). They focus on
developing an almost-similar test, i.e., they approximate a
similar test %while preserving type~1 error 
by optimizing a
particular objective function involving a penalty for deviations from the
nominal type~1 error level. By contrast, we consider two approaches for
optimizing decision theoretic criteria that characterize power while
preserving type~1 error. The solution to one of these has a closed form and is
an exactly similar test. The other is an approximate Bayes risk optimal test
inspired by the test of \citet{rosenblum2014optimal}.

The remainder of the article is organized as follows. In Section~\ref{sec:prelim}, we formalize the problem and explain the shortcomings of
traditional tests. In Sections~\ref{sec:minimax} and~\ref{sec:bayes}, we present the minimax
optimal and Bayes risk optimal tests, respectively. %In Section~\ref{sec:bayes}, we present the Bayes risk optimal test. 
In Section~\ref{sec:large-scale}, we discuss adaptations for large-scale mediation hypothesis testing. 
%In Section~\ref{sec:pval}, we discuss $p$-values corresponding to the
%former tests and false discovery rate control. %In Section~\ref{sec:extensions}, we discuss extensions to tests of products of more than two coefficients. 
In Section~\ref{sec:interpretability}, we discuss interpretability challenges with the proposed tests, and introduce modifications leading to improved interpretability. 
In Section~\ref{sec:simulations}, we present results from
simulation studies. In Section~\ref{sec:data-analysis}, we apply our methodology to two data sets: one % from the U.S.~National Institute of Mental Health's Database of Cognitive Training and Remediation Studies 
to test whether cognition mediates the effect of cognitive remediation therapy on self-esteem %social functioning 
in patients with schizophrenia, and another to %perform %large-scale hypothesis 
test many mediation hypotheses of whether a host of DNA methylation CpG sites mediate the effect of smoking status on lung function. In Section~\ref{sec:discussion}, we conclude with a discussion. 

\section{\centering PRELIMINARIES}
\label{sec:prelim}

%\antoine{}{I add subsections mainly for argumentation purpose. We could move
% Section~\ref{subsec:mediation:analysis} (almost three-page long!) to the
% appendix. That would give more momentum to the paper\ldots}

%\subsection{General Setup}
%\label{subsec:set:up}

We begin with the general problem statement, which we will then connect to mediation
analysis. %in Section~\ref{subsec:mediation:analysis}. 
Suppose we have $n$
independent, identically distributed (i.i.d.)~observations from a
distribution $P_{\bm{\delta}}$ indexed by the parameter $\bm{\delta}=(\delta_x,\delta_y)^{\top}\in\mathbb{R}^2$, %which contains two scalar parameters $\delta_x$ and $\delta_y$, 
and we wish to test the null
hypothesis $H_0: ``\delta_x\delta_y=0"$ against its alternative
$H_1: ``\delta_x\delta_y\neq 0"$. Further, suppose we have a uniformly asymptotically normal and unbiased %an asymptotically normal 
estimator \citep{robins1997toward} $(\hat{\delta}_x, \hat{\delta}_y)$ for $(\delta_x, \delta_y)$, i.e., 
\begin{equation}
 \label{eq:CLT}
\sup_{\delta_x,\delta_y}\left\lvert\mathrm{Pr}_{\delta_x,\delta_y}\left[\sqrt{n}\bm{\Sigma_n}^{-1/2}\left\{(\hat{\delta}_x, \hat{\delta}_y)^{\top}
 - (\delta_x,\delta_y)^{\top}\right\}\leq[t_x,t_y]^{\top}\right] - \Phi_2([t_x,t_y]^{\top})\right\rvert \rightarrow 0
\end{equation}
as $n\rightarrow\infty$ for all $(t_x,t_y)^{\top}\in\mathbb{R}^2$,
%where the convergence
%\begin{equation}
% \label{eq:CLT}
% \sqrt{n}\bm{\Sigma_n}^{-1/2}\left\{(\hat{\delta}_x, \hat{\delta}_y)^{\top}
% - (\delta_x,\delta_y)^{\top}\right\} \rightsquigarrow \mathcal{N}\left\{(0, 0)^{\top},
% \bm{I_{2}}\right\}
%\end{equation}
% \antoine{}{to the centered Gaussian law with identity covariance
%is uniform in $(\delta_x, \delta_y)$, 
where $\bm{\Sigma_n}$ is a diagonal matrix that is a consistent estimator
of the asymptotic covariance matrix of $(\hat{\delta}_x, \hat{\delta}_y)^{\top}$, %($\bm{I_2}$ is the $2\times 2$ identity matrix), 
and $\Phi_2$ is the cumulative distribution function of the bivariate Gaussian distribution with identity covariance. We wish to test $H_0$ against $H_{1}$ with at most $\alpha$ type~1
error for each $(\delta_x,\delta_y)$ satisfying $\delta_x\delta_y=0$
and maximizing power elsewhere. %(in some sense) everywhere else. 
Clearly, $H_0$ is a
very particular type of composite null hypothesis, which in the
space spanned by $\delta_x$ and $\delta_y$ consists of the $x$ and $y$ axes.
%since $\delta_x\delta_y=0$ whenever $\delta_x=0$ or $\delta_y=0$.

%\subsection{The Composite Null Hypothesis in Mediation Analysis}
%\label{subsec:mediation:analysis}

The composite null hypothesis $H_0$ arises naturally in mediation analysis
under certain modeling assumptions. %To discuss the indirect effect of interest, we first introduce notation. 
Suppose we observe $n$ i.i.d.~copies of
$(\bm{C}^\top,A,M,Y)^\top$, where $A$ is the exposure of interest, $Y$ is the
outcome of interest, $M$ is a potential mediator, %that is temporally intermediate to $A$ and $Y$, 
and $\bm{C}$ is a vector of baseline covariates that is assumed 
%we will assume throughout to be sufficient 
to control for %various sorts of
confounding %needed for 
such that the indirect effect is %to be 
identified. The \emph{natural indirect effect} (NIE) is the mediated effect of $A$ on $Y$ through $M$. %, i.e., the effect along the causal pathway $A\rightarrow M\rightarrow Y$. 
We formally define the NIE %in terms of nested counterfactuals 
and discuss its identification in the Supplementary Materials. If the linear models with main
 effect terms given by
\begin{align}
 \label{eq:model1}
 E(M\mid A=a, \bm{C}=\bm{c}) &= \beta_0 + \beta_1a + \bm{\beta_2}^\top \bm{c},\\
 \label{eq:model2}
 E(Y\mid A=a,M=m,\bm{C}=\bm{c}) &= \theta_0 + \theta_1a + \theta_2m + \bm{\theta_3}^\top \bm{c}
\end{align}
are correctly-specified, then the identification formula for the
natural indirect effect reduces to $\beta_1\theta_2$. The \emph{product
 method} estimator is the product of estimates of $\beta_1$ and $\theta_2$ obtained by fitting the above
regression models. Under models \eqref{eq:model1} and \eqref{eq:model2} and standard regularity
conditions, the two factors of the product method estimator will jointly satisfy the
uniform convergence statement in~\eqref{eq:CLT}. In fact,
\eqref{eq:CLT} will hold for a more general class of models, though there are
important limitations to this class. For instance, consider the outcome model
with exposure-mediator interaction replacing model \eqref{eq:model2}:
\begin{align}
 \label{eq:model2bis}
 E(Y\mid A=a,M=m,\bm{C}=\bm{c}) &= \theta_0 + \theta_1a + \theta_2m + \theta_3am +
 \bm{\theta_4}^\top \bm{c}.
\end{align}
Under models \eqref{eq:model1} and \eqref{eq:model2bis}, the natural indirect
effect comparing treatment levels $a'$ and $a''$ is identified by
$(\theta_2+\theta_3a')\beta_1(a'-a'')$, which is naturally estimated by $(\hat{\theta}_2+\hat{\theta}_3a')\hat{\beta}_1(a'-a'')$, where
$\hat{\theta}_2$, $\hat{\theta}_3$, and $\hat{\beta}_1$ are estimated by
fitting the linear regression models \eqref{eq:model1} and
\eqref{eq:model2bis}. Setting $\delta_x=\theta_2+\theta_3a'$, $\delta_y=\beta_1(a'-a'')$, $\hat{\delta}_x=\hat{\theta}_2+\hat{\theta}_3a'$, and $\hat{\delta}_y=\hat{\beta}_1(a'-a'')$, the joint estimator will also satisfy \eqref{eq:CLT} under standard
regularity conditions. %, with $\delta_x=\theta_2+\theta_3a'$ and $\delta_y=\beta_1(a'-a'')$. %and with the corresponding plug-in estimators for $\hat{\delta}_x$ and $\hat{\delta}_y$.

%\subsection{Problems With Traditional Tests}
%\label{subsec:circumvent}

The problem with traditional tests of the composite null $H_0$
against its alternative $H_{1}$ stems from the fact that while
$(\hat{\delta}_x, \hat{\delta}_y)$ converges uniformly in
$\bm{\delta}$, certain scalar functions of these estimators, e.g., the product method estimator
$\hat{\delta}_x\hat{\delta}_y$, do not. This can be seen by examining the
delta method approach proposed by \citet{sobel1982asymptotic} (often referred
to as the Sobel test). %Consider the case in which $\bm{\Sigma_n}$ in \eqref{eq:CLT} is diagonal with diagonal elements $s_x^2$ and $s_y^2$. This will be the case under models \eqref{eq:model1} and \eqref{eq:model2} or \eqref{eq:model1} and \eqref{eq:model2bis} due to the required no-unobserved-confounding assumptions. 
When $\bm{\delta} \neq (0,0)^{\top}$,
the delta method yields
$\sqrt{n}(\hat{\delta}_x\hat{\delta}_y-\delta_x\delta_y)\rightsquigarrow
\mathcal{N}(0,\delta_y^2\sigma_x^2+\delta_x^2\sigma_y^2)$, where $\sigma_x^2$ and $\sigma_y^2$ are the diagonal elements of the limit in probability of $\bm{\Sigma}_n$. Thus, when
precisely one of $\delta_x$ or $\delta_y$ is zero,
$Z_n^{\text{prod}}:=n^{1/2}\hat{\delta}_x\hat{\delta}_y/(\hat{\delta}_y^2s_x^2+\hat{\delta}_x^2s_y^2)^{1/2}\rightsquigarrow \mathcal{N}(0,1)$, 
%\[Z_n^{\text{prod}}:=\frac{\sqrt{n}\hat{\delta}_x\hat{\delta}_y}{\sqrt{\hat{\delta}_y^2s_x^2+\hat{\delta}_x^2s_y^2}}=\frac{\sqrt{n}(\hat{\delta}_x\hat{\delta}_y-\delta_x\delta_y)}{\sqrt{\hat{\delta}_y^2s_x^2+\hat{\delta}_x^2s_y^2}} \rightsquigarrow \mathcal{N}(0,1),\] 
where $s^2_x$ and $s^2_y$ are
consistent estimates of $\sigma^2_x$ and $\sigma^2_y$, respectively. However,
as $\bm{\delta}\rightarrow (0,0)^{\top}$, this no longer holds, as the denominator also goes to zero. \citet{liu2022large} showed that this instead converges to a centered normal distribution with variance 1/4. 
%asymptotic variance shrinks to zero so that, at $(\delta_x, \delta_y) = (0,0)$, the asymptotic distribution is degenerate. Instead of converging at an $n^{-1/2}$ rate, $\hat{\delta}_x\hat{\delta}_y$ converges at an $n^{-1}$ rate when $(\delta_x, \delta_y) = (0,0)$. Specifically, $n\hat{\delta}_x\hat{\delta}_y/s_xs_y$ converges in law to the product-normal distribution, i.e., the law of $W_1W_2$ when $W_1$ and $W_2$ are independently drawn from the law $\mathcal{N}(0,1)$. 
Thus, the
convergence in distribution of the delta method test statistic $Z_n^{\text{prod}}$ is not
uniform, and for any given sample size, %no matter how large, 
there will be
a region in the parameter space around $(0,0)$ where the standard normal
approximation is poor. This is especially
problematic, as the region around $(0,0)$ is precisely where we expect
the truth to often lie. %, unlike many other cases in which convergence is not uniform but approximations are poor only near the boundary of the parameter space.
As $\bm{\delta}$ approaches $(0,0)$, %one moves closer to the origin in the parameter space, 
%the true distribution of 
$Z_n^{\text{prod}}$ begins to more closely resemble $N(0,1/4)$, as we show in a
Monte Carlo sampling approximation of density functions in Figure \ref{fig:density}.a in the Supplementary Materials, 
%Figure~\ref{fig:density}.a, 
where we fix $\delta_y=0$ and vary $\delta_x$ from 0
to 0.3.
\begin{comment}
\begin{figure}[h]
 \centering \includegraphics[width=.67\textwidth]{density_plots_1e5}
 \caption{Density plots estimated by Monte Carlo sampling of $Z_n$ under $\delta_y=0$ and varying $\delta_x$ with $n=100$. The standard normal density curve is displayed in black.}
 \label{fig:density}
\end{figure}
\end{comment}
\begin{comment}
\begin{figure}[h]
\centering
\begin{tabular}{ccc}
\includegraphics[width=.45\textwidth]{density_plots_1e5}
& 
&
\includegraphics[width=.45\textwidth]{density_plots_js_1e5}
\\
(a)	&	&	(b)
\end{tabular}
\caption{(a) Density plots estimated by Monte Carlo sampling of $Z_n^{\text{prod}}$ under $\delta_y=0$ and varying $\delta_x$ with $n=100$. The standard normal density curve is displayed in black. (b) Density plots estimated by Monte Carlo sampling of $Z^{\text{min}}_n$ under $\delta_y=0$ and varying $\delta_x$ with $n=100$. The folded standard normal density curve is displayed in black.}
\label{fig:density}
\end{figure}
\end{comment}
The delta method uses an approximation based on the standard normal distribution, when in reality, the test statistic's true distribution is much more concentrated about zero. %, as shown in the figure. %we can see in Figure~\ref{fig:density}.a. 
%This explains the conservative behavior of the delta method-based test near the origin.
Hypothesis tests based on inverting bootstrap confidence
intervals are also known to perform poorly near the origin of the parameter
space due to the %. Bootstrap theory requires continuity in the asymptotic distribution over the parameter space.
%\citep{barfield2017testing}. 
%As we have seen, there is a 
singularity in the pointwise
asymptotic distribution at the origin.

Another popular test of the composite null $H_0$ against its alternative $H_{1}$ is known as the joint significance test \citep{cohen2013applied}, which is %an intersection-union test 
based on the logic that both $\delta_x$ and $\delta_y$ must be nonzero for
$\delta_x\delta_y$ to be nonzero. The joint significance test with nominal level $\alpha$ amounts to testing the two null hypotheses
$H_0^x: ``\delta_x=0"$ and $H_0^y: ``\delta_y=0"$ against their
 alternatives $H_{1}^{x}:``\delta_x\neq 0"$ and
 $H_{1}^{y}:``\delta_y\neq 0"$ separately using the Wald statistics
$Z^x_n:=\sqrt{n}\hat{\delta}_x/s_x$ and
$Z^y_n := \sqrt{n}\hat{\delta}_y/s_y$, and rejecting $H_0$ for $H_{1}$ only if $H_0^x$ and $H_0^y$ are both rejected for their alternatives %by the corresponding Wald tests 
with nominal level $\alpha$.
As noted, this test has been shown to perform better in
simulations than other common tests of $H_0$ against $H_{1}$. However, it
still suffers from the lack of power near the origin of the parameter space we
see in other tests. This is again due to a lack of uniform convergence near the origin. The joint significance test implicitly relies on the approximation $Z_n^{\text{min}} := \vert Z_n^x\vert \wedge \vert Z_n^y\vert \sim \vert \mathcal{N}\vert(0, 1)$, where $\vert \mathcal{N}\vert (0, 1)$ represents the folded standard normal distribution. This approximation is poor when $\delta_x$ and $\delta_y$ are both close to zero, as is shown in Figure \ref{fig:density}.b in the Supplementary Materials. %\ref{fig:density}.b, 
%where we once again fix $\delta_y=0$ and vary $\delta_x$ from 0 to 0.3. %The true distribution of $\vert Z_n^x\vert \wedge \vert Z_n^y\vert$ is again more concentrated near zero than that of its approximation.

To avoid the problems that arise due to non-uniform convergence, we instead %focus on 
construct a rejection region in the parameter space spanned by the vector $(\delta_x,\delta_y)$ in $\mathbb{R}^2$, rather than that spanned by $\delta_x\delta_y$ in $\mathbb{R}$. In the former, convergence is uniform and there are no discontinuities in the asymptotic distribution. %over the parameter space. 
Studying the rejection regions of the traditional tests in $\mathbb{R}^2$ illustrates their conservative behavior from another perspective. Consider a sequence of local noncentrality parameters with respect to a given
direction $(\delta_x,\delta_y)^\top$,
$(\delta_{x,n},\delta_{y,n})^\top:=
(\delta_x/n^{1/2},\delta_y/n^{1/2})^\top$, and define
$\bm{\delta}^*=(\delta_x^*,\allowbreak \delta_y^*)^\top\allowbreak :=\allowbreak \mathrm{plim}_{n\rightarrow\infty}\allowbreak n^{1/2}\allowbreak \bm{\Sigma_n}^{-1/2}\allowbreak (\delta_{x,n},\allowbreak \delta_{y,n})^\top\allowbreak =\allowbreak \bm{\Sigma}^{-1/2}\allowbreak (\delta_x,\allowbreak \delta_y)^\top$
where $\mathrm{plim}$ is the limit in probability and $\bm{\Sigma} := \mathrm{plim}_{n\rightarrow\infty} \bm{\Sigma}_n$. %When $\bm{\Sigma}$ is diagonal as in the mediation case, $(Z^x_n,Z^y_n)^\top$ will simply be the pair of Wald statistics $(\sqrt{n}\hat{\delta}_x/s_x^2, \sqrt{n}\hat{\delta}_y/s_y^2)^{\top}$ for testing $H_0^x$ and $H_0^y$ against $H_{1}^{x}$ and $H_{1}^{y}$ separately as discussed above. 
Define the bivariate test statistic
$(Z^x_n,Z^y_n)^\top:=n^{1/2}\bm{\Sigma_n}^{-1/2}(\hat{\delta}_x,\hat{\delta}_y)^\top$
so that $(Z^x_n,Z^y_n)^\top$ is approximately distributed as
$\mathcal{N}\{(\delta_x^*,\allowbreak\delta_y^*)^\top,\allowbreak\bm{I_2}\}$, and let $(Z_*^x, Z_*^y)^\top\sim \mathcal{N}\{(\delta_x^*,\allowbreak\delta_y^*)^\top,\allowbreak\bm{I_2}\}$. %be a random vector following this limiting distribution. 
Based on this
asymptotic approximation, the probability of the Wald test rejecting $H_0^x$
for $H_1^x$ at nominal level $\alpha$ is $\alpha$ under all elements of
$\{\bm{\delta}^*:\delta_x^*=0\}$, and likewise for the Wald test rejecting $H_0^y$ for $H_1^y$ for all $\{\bm{\delta}^*:\delta_y^*=0\}$. 
The rejection region in $\mathbb{R}^2$ corresponding to the joint significance test is the intersection of these two rejection regions, as shown in the plot in Figure \ref{fig:js}(a). 
\begin{comment}
\begin{figure}[h]
\centering
 \includegraphics[width=.45\textwidth]{JS_plot} 
 \caption{The rejection region of the joint significance test is shown in darker gray. Dashed lines indicate the boundaries of the rejection regions of the Wald tests for $H_0^x$ and $H_0^y$.}
 \label{fig:js}
\end{figure}
\end{comment}
\begin{figure}[h]
\centering
\begin{tabular}{ccc}
\includegraphics[width=.45\textwidth]{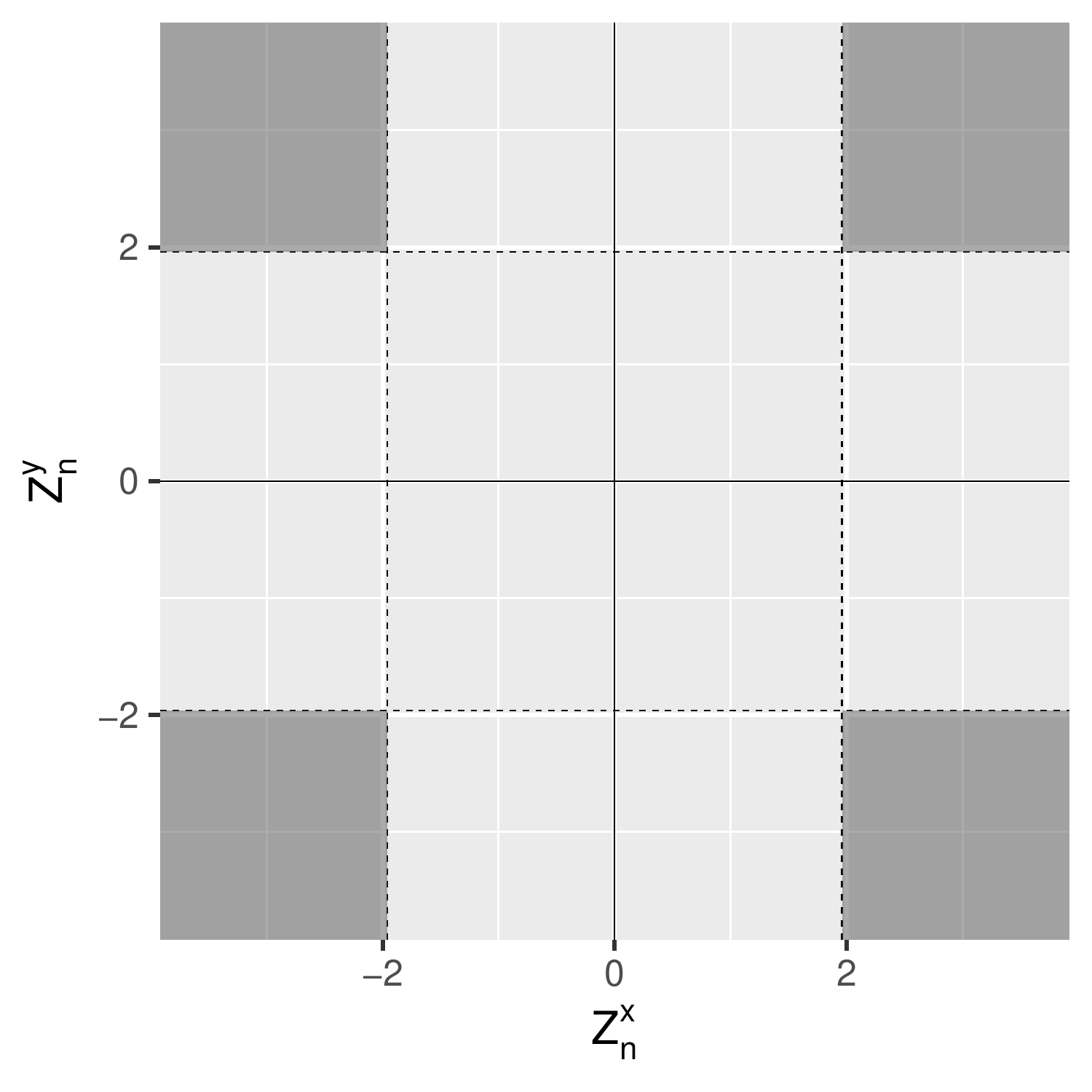}
& 
&
\includegraphics[width=.45\textwidth]{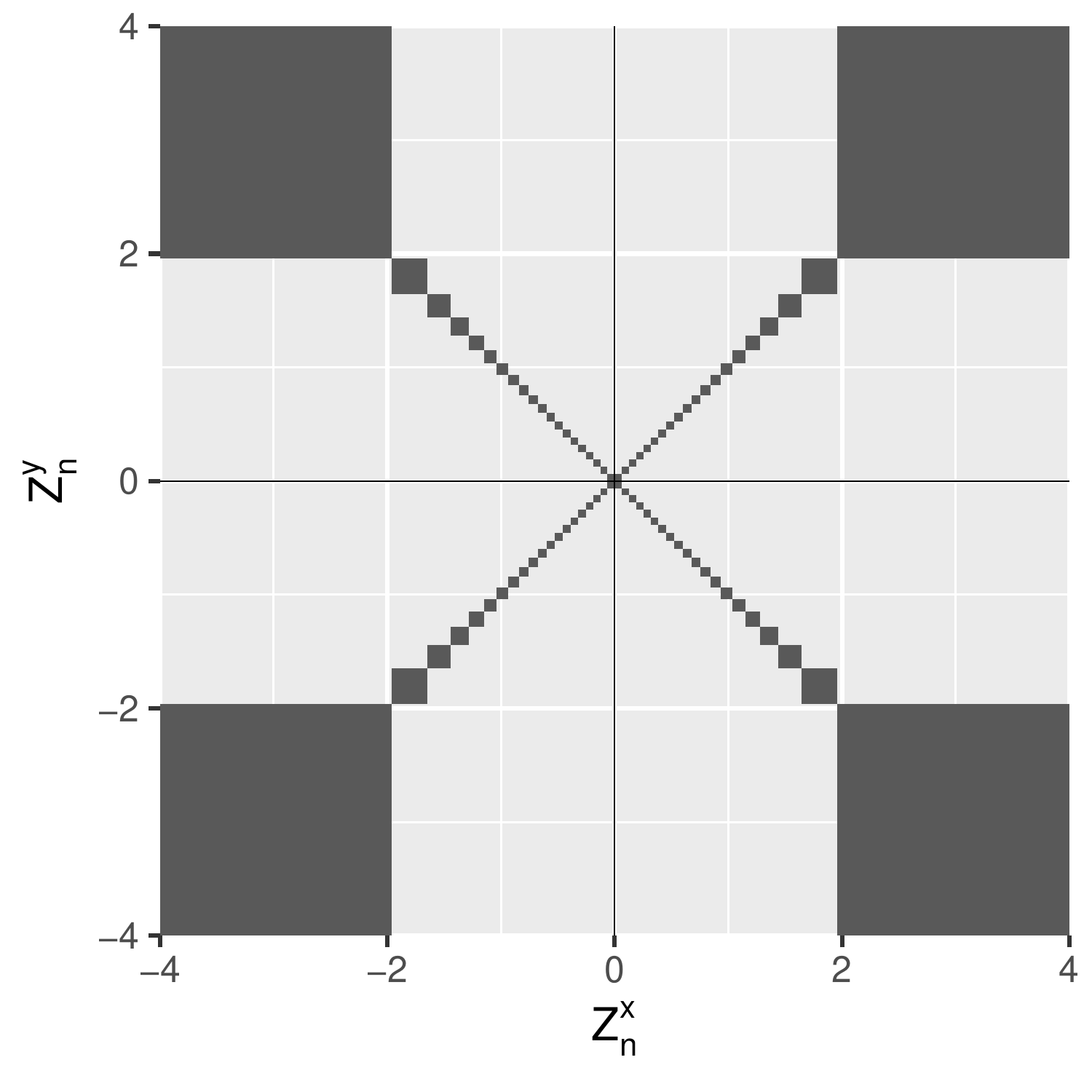}
\\
(a)	&	&	(b)
\end{tabular}
\caption{(a) The rejection region of the joint significance test is shown in darker gray. Dashed lines indicate the boundaries of the rejection regions of the Wald tests for $H_0^x$ and $H_0^y$. (b) The rejection region $R_{mm}$ for $\alpha=0.05$ is shown in dark
 gray.}
\label{fig:js}
\end{figure}
Because we take the intersection of rejection regions of two size-$\alpha$ tests, the resulting test is necessarily conservative. In fact, when $\bm{\delta}^*=(0,0)^{\top}$, the probability of being in the
joint significance test rejection region is $\alpha^2$. %, well below the desired level $\alpha$. 
By continuity of the rejection probability in
$\bm{\delta}^*$, the rejection probability is likewise well below
$\alpha$ in a region around the origin, hence the underpoweredness of the
joint significance test. That is, the worst-case type 2 error of the joint
significance test is
\begin{align*}
 \sup_{\bm{\delta}^*:\delta^*_x\delta^*_y\neq 0} \mathrm{Pr}_{\bm{\delta}^*}\{(Z^x_*, Z^y_*)\notin R\}
 %&=1-\inf_{(\delta^*_x,\delta^*_y):\delta^*_x\delta^*_y \neq 0}\mathrm{Pr}_{\delta^*_x,\delta^*_y}\{(Z^x_*, Z^y_*)\in R\}\\
 =1-\min_{\bm{\delta}^*:\delta^*_x\delta^*_y = 0}\mathrm{Pr}_{\bm{\delta}^*}\{(Z^x_*, Z^y_*)\in R\}=1-\alpha^2.%=0.9975.
\end{align*}
\cite{huang2019genome} and \cite{van2024nearly} showed that the delta method-based test's rejection region is contained in that of the joint significance test, and therefore the former is uniformly more conservative than the latter. %, hence it suffers from the same problem. 
In the following two sections, we define new tests by constructing
a rejection region that preserves type~1 error uniformly over the null
hypothesis space, but improves power in the alternative hypothesis space,
particularly near the origin of the parameter space.

\section{\centering MINIMAX OPTIMAL TEST}
\label{sec:minimax}

In this section, we seek a test defined by a rejection region
$R^*\subset\mathbb{R}^2$ satisfying
$\sup_{\bm{\delta}^*:\delta^*_x\delta^*_y=
 0}\allowbreak \mathrm{Pr}_{\bm{\delta}^*}\allowbreak \{(Z_*^x,\allowbreak Z_*^y)\allowbreak \in \allowbreak R^*\}
 \allowbreak =\allowbreak \alpha$
 \begin{comment}
\begin{align}
 \sup_{(\delta^*_x,\delta^*_y):\delta^*_x\delta^*_y=
 0}\mathrm{Pr}_{\delta^*_x,\delta^*_y}\{(Z_*^x,Z_*^y)\in R^*\}
 =\alpha
\end{align}
\end{comment}
and the minimax optimality criterion
\begin{align}
 \label{eq:minimax}
 \inf_{\bm{\delta}^*:\delta^*_x\delta^*_y \neq
 0}\mathrm{Pr}_{\bm{\delta}^*}\{(Z_*^x,Z_*^y)\in R^*\}
 =\sup_{R\subset
 \mathbb{R}^{2}}\inf_{\bm{\delta}^*:\delta^*_x\delta^*_y \neq
 0}\mathrm{Pr}_{\bm{\delta}^*}\{(Z^x_*, Z^y_*)\in R\}. 
\end{align}
The rejection region $R^{*}$ generates a test with
 type~1 error $\alpha$ achieving the minimax risk of the 0-1 loss
 function, i.e., yielding the largest worst-case power. By continuity of
the power function in $\bm{\delta}^*$, the minimax optimal power is
upper bounded by $\alpha$. %, since the rejection probability can be at most $\alpha$ at any given value of $(\delta^*_x,\delta^*_y)$ in the null hypothesis space, and can be arbitrarily close to this rejection probability in a small enough neighborhood around that point, which will contain elements of the alternative hypothesis space. 
Moreover, if \eqref{eq:minimax} is
attainable, a minimax optimal test generated by the rejection region $R$ must satisfy
\begin{equation}
 \label{eq:similarity}
 \mathrm{Pr}_{\bm{\delta}^*}\{(Z_*^x,Z_*^y)\in R\}=\alpha \quad
 \text{for all } 
 \bm{\delta}^*\in\mathbb{R}^{2} \text{ such that }
 \delta^*_x\delta^*_y=0. 
\end{equation}
A test satisfying \eqref{eq:similarity} is known as a \emph{similar test}. %\citep{lehmann2005testing}. %Technically, this only needs to hold on the boundary of the null hypothesis space; however, in our case the null hypothesis space is its own boundary.
Similarity is an important property that is closely tied to uniformly most powerful unbiased tests in the classical hypothesis testing literature. %(e.g., see Chapter 4 of \cite{lehmann2005testing}). 
In our case, it is important because a non-similar test must have type 1 error either greater than or less than $\alpha$ somewhere in the null hypothesis space. In the former case, the test will fail to preserve type 1 error. In the latter, %by continuity of the rejection probability as a function of the true parameter, 
the rejection probability will necessarily be less than $\alpha$ in some neighborhood that also contains the alternative hypothesis space. Thus, a non-similar test will necessarily be a biased test and underpowered in a region of the alternative hypothesis space.

\subsection{A Minimax Optimal and Similar Test for Unit Fraction $\alpha$}
We now define a rejection region that generates a similar test for any unit
fraction $\alpha$, i.e., any $\alpha$ such that $\alpha^{-1}$ is a positive
integer, hence a test that satisfies the minimax optimality criterion
\eqref{eq:minimax}. Let $a_k=\Phi^{-1}(k\alpha/2)$ for $k=0,\ldots,2/\alpha$,
where $\Phi$ is the standard normal distribution function. Define $R_{mm}:=\{\bigcup_{k=1}^{2/\alpha} (a_{k-1},a_k)\times(a_{k-1},a_k)\}\cup\{\bigcup_{k=1}^{2/\alpha} (a_{k-1},a_k)\times(-a_k,-a_{k-1})\}.$
\begin{comment}
\begin{align*}
 R_{mm}:=&\left\{\bigcup\limits_{k=1}^{2/\alpha} (a_{k-1},a_k)\times(a_{k-1},a_k)\right\}\cup\left\{\bigcup\limits_{k=1}^{2/\alpha} (a_{k-1},a_k)\times(-a_k,-a_{k-1})\right\}.
\end{align*}
\end{comment}
The region $R_{mm}$ for $\alpha=0.05$ is depicted in the plot in Figure
\ref{fig:js}(b). %, and consists of squares along the $y=x$ and $y=-x$ diagonals.
\begin{comment}
\begin{figure}[h]
 \centering \includegraphics[width=.45\textwidth]{rej_plot_0-05-1.pdf}
 \caption{The rejection region $R_{mm}$ for $\alpha=0.05$ is shown in dark
 gray.}%, with the squares in the corners continuing out to infinity.}
 \label{fig:minimax}
\end{figure}
\end{comment}
Clearly, $R_{mm}$ contains the rejection region of the joint significance
test, hence it yields a test that is uniformly more powerful. %A plot of its power function surface over $\mathbb{R}^2$ is available in the Supplementary Materials.
\begin{theorem}
\label{thm:similar}
The rejection region $R_{mm}$ satisfies \eqref{eq:similarity} and generates a similar test.
\end{theorem}
%All proofs are presented in the appendix. 
The essence of the proof of Theorem
\ref{thm:similar} is that when $\delta^*_y=0$ (resp.~$\delta^*_x=0$), whatever level $Z^x_*$ (resp.~$Z^y_*$) takes, the conditional probability of $Z^y_*$ (resp.~$Z^x_*$) being in $R_{mm}$ given $Z^x_*$ (resp.~$Z^y_*$) equals $\alpha$, since
the probability of being in each vertical interval is $\alpha/2$ under the
standard normal law. %distribution. %The same holds for $Z^x_*$ when $\delta^*_x=0$
%by symmetry.

%Similarity is an important property that has been closely related to uniformly most powerful unbiased tests in the classical hypothesis testing literature (e.g., see Chapter 4 of \cite{lehmann2005testing}). In our case, it is important because a non-similar test must have type 1 error either greater than or less than $\alpha$ somewhere in the null hypothesis space. In the former case, the test will fail to preserve type 1 error. In the latter, by continuity of the rejection probability as a function of the true parameter, the rejection probability will necessarily be less than $\alpha$ in some neighborhood that also contains the alternative hypothesis space. Thus, a non-similar test will necessarily be a biased test and underpowered in a region of the alternative hypothesis space.

\begin{corollary}
\label{cor:minimax}
The rejection region $R_{mm}$ satisfies the minimax optimality criterion
\eqref{eq:minimax} with
$\inf_{\bm{\delta}^*:\delta^*_x\delta^*_y \neq
 0}\mathrm{Pr}_{\bm{\delta}^*}\{(Z_*^x,Z_*^y)\in R_{mm}\}=\alpha$.
\end{corollary}

\citet{van2020almost} gave a result stating that there is no similar test of $H_0$ against $H_1$ within a particular class of tests defined by a rejection region generated by rotating and reflecting %about various axes 
a monotonically increasing function from zero to infinity. While Corollary \ref{cor:minimax} may appear to contradict this result, $R_{mm}$ does not in fact belong to the class considered by \citet{van2020almost}, but rather belongs to a class that relaxes the restriction that the function must be monotonically increasing to allow for it to be monotonically nondecreasing. %(though the more recent version \citep{van2024nearly} recognizes the existence of this function). 
In fact, the following theorem states that the test generated by $R_{mm}$ is the only similar test in this relaxed class up to a set of measure zero.
\begin{theorem}
\label{thm:unique}
Let $\mathcal{F}$ %=\{f:f(x)\rightarrow [0,x], x\in [0,\infty), f\mathrm{ is monotonically nondecreasing}\}$
be the class of all monotonically nondecreasing functions mapping all $x$ from the nonnegative real numbers to $[0,x]$, and for any $f\in\mathcal{F}$, let $R_f$ be the region generated by taking all possible negations and permutations of the region $\{(x,y):y\in(f(x),x]\}$, or equivalently reflecting this region about the $x$, $y$, $y=x$, and $y=-x$ axes, and rotating it $\pi/2$, $\pi$, and $3\pi/2$ radians about the origin. The function $x\mapsto f_{mm}(x):=\sum_{k=1/\alpha}^{2/\alpha}a_{k-1}I(a_{k-1} \leq x < a_k)$, which generates $R_{mm}$, is the unique function in $\mathcal{F}$ (up to a set of measure zero) that generates a similar test.
\end{theorem}

%In addition to minimax optimality, this test has some other desirable properties. 
This test has a simple, exact closed form, making it very fast and straightforward to implement. Additionally, it is nonrandom and symmetric with respect to negation and permutations. %We choose the squares $(a_{k-1},a_k)\times(a_{k-1},a_k)$ to be open so that the test never rejects when $Z^x_n$ or $Z^y_n$ are precisely zero. %, though it makes no real difference since the boundaries of these squares make up a set of measure zero. 
%However, it is the case that
One less desirable property is that the test will reject for $(Z^x_n, Z^y_n)$ arbitrarily close to $(0,0)$, %This is a peculiar property of $R_{mm}$, 
which we discuss further in Section \ref{sec:interpretability}. However, the following result shows this is necessary in order to attain the best worst-case power.
%from Theorem \ref{thm:unique} we infer that at least within the class of tests generated by $\mathcal{F}$, this is necessary in order to attain the best worst-case power. In fact, more generally, we have the following result. 
\begin{theorem}
\label{thm:impossibility}
There is no similar test of $H_0$ against $H_1$ with a rejection region that is bounded away from $\{\bm{\delta}^*: \delta_x^*\delta_y^*=0\}$.
\end{theorem}
Thus, one cannot have a hypothesis test of $H_0$ that is both similar/unbiased and that fails to reject for all values of $(Z_n^x,Z_n^y)$ that are arbitrarily close to the null hypothesis space. %; there is a fundamental trade-off between these two properties.

%This property of $R_{mm}$ is justified because even when $(\delta^*_x,\delta^*_y)=(0,0)$, the type~1 error based on the asymptotic approximation remains exactly $\alpha$ despite rejecting when $(Z^x_n, Z^y_n)$ is arbitrarily close to $(0,0)$. This is in large part due to the fact that the probability of being in the squares in the rejection region surrounding the origin when $(\delta^*_x,\delta^*_y)=(0,0)$ is quite small at only $\alpha^2$. %This part of the rejection region is important to improve power for alternatives near the origin.%, under which there is similar probability of being in the origin-adjacent squares to that under $(\delta^*_x,\delta^*_y)=(0,0)$.

\subsection{Dealing With Non-Unit Fraction Values of $\alpha$}
\label{sec:non-unit-fraction}
The rejection region $R_{mm}$ is only defined above for unit fractions $\alpha$. For other values of $\alpha$, the class of tests generated by $\mathcal{F}$ will not contain a similar test, though nonrandom similar tests may exist outside of this class. Tests at other values of $\alpha$ are typically of interest for two reasons: using multiple testing adjustment procedures and defining $p$-values. We discuss the latter in the Supplementary Materials. %Section \ref{sec:pval}. 
The restriction that $\alpha$ must be a unit fraction is not a limitation for the Bonferroni correction procedure, which involves dividing the familywise error rate $\alpha$ by the number of tests, provided the familywise error rate $\alpha$ itself is a unit fraction. However, it is a limitation for the Benjamini--Hochberg procedure, which involves multiplying the false discovery rate $\alpha$ by a sequence of rational numbers, only some of which are unit fractions. We discuss the latter procedure in Section \ref{sec:large-scale}.

We now define an extension of the minimax optimal test to non-unit fraction values of $\alpha\in (0,1)$. When $\alpha$ is not a unit fraction, then this extension will not yield a minimax optimal test %, and as previously noted, will not yield 
or a similar test, as similar tests do not exist in this class of tests for such values. However, they do preserve type~1 error, and have type~1 error that approximates $\alpha$ under the least-favorable distribution in the null hypothesis very closely as $\alpha$ goes to zero. Let $b_{-1}:=0$ and $b_k:=\Phi^{-1}\{1-(\lfloor\alpha^{-1}\rfloor-k)\alpha/2\}$ for $k=0,\ldots,\lfloor\alpha^{-1}\rfloor$ so that $0=b_{-1}\leq b_0 < b_1 < \ldots < b_{\lfloor \alpha^{-1} \rfloor}=+\infty$. We define the rejection region for any $\alpha\in(0,1)$ to be $R_{mm}':=\{\bigcup_{k=0}^{\lfloor\alpha^{-1}\rfloor} (b_{k-1},b_k)\times(b_{k-1},b_k)\}\cup\{\bigcup_{k=0}^{\lfloor\alpha^{-1}\rfloor} (b_{k-1},b_k)\times(-b_k,-b_{k-1})\}
\cup\{\bigcup_{k=0}^{\lfloor\alpha^{-1}\rfloor} (-b_k,-b_{k-1})\times(b_{k-1},b_k)\}\cup\{\bigcup_{k=0}^{\lfloor\alpha^{-1}\rfloor} (-b_k,-b_{k-1})\times(-b_k,-b_{k-1})\}$.
\begin{comment}
\begin{align*}
R_{mm}':=&\left\{\bigcup\limits_{k=0}^{\lfloor\alpha^{-1}\rfloor} (b_{k-1},b_k)\times(b_{k-1},b_k)\right\}\cup\left\{\bigcup\limits_{k=0}^{\lfloor\alpha^{-1}\rfloor} (b_{k-1},b_k)\times(-b_k,-b_{k-1})\right\}\\
\cup&\left\{\bigcup\limits_{k=0}^{\lfloor\alpha^{-1}\rfloor} (-b_k,-b_{k-1})\times(b_{k-1},b_k)\right\}\cup\left\{\bigcup\limits_{k=0}^{\lfloor\alpha^{-1}\rfloor} (-b_k,-b_{k-1})\times(-b_k,-b_{k-1})\right\}.
\end{align*}
\end{comment}
When $\alpha$ is a unit fraction, $R'_{mm}$ coincides with $R_{mm}$, hence $R'_{mm}$ generates the minimax optimal test discussed earlier. Otherwise, the test generated by $R'_{mm}$ corresponds closely to the minimax optimal test in that if $\delta_y^*=0$, then given $Z^x_*=z$ with $z\notin[-b_0,b_0]$, the probability of being in the rejection region is exactly $\alpha$. %That is, $\mathrm{Pr}_{\delta^*_y=0}\allowbreak \left\{(Z^x_*,\allowbreak Z^y_*)\allowbreak \in \allowbreak R^*\mid Z^x_*\right\}\allowbreak =\allowbreak \alpha$ for all values of $Z^x_*$ except for in the four squares adjacent to the origin. 
Precisely:
\begin{theorem}
    \label{thm:non-unit:fraction}
    The rejection region $R_{mm}'$ is such that
    \begin{align}
      \label{eq:type1:inf}
      \lfloor\alpha^{-1}\rfloor\alpha^2+(1-\lfloor\alpha^{-1}\rfloor\alpha)^2
      &= \inf_{\bm{\delta}^*:\delta^*_x\delta^*_y=
        0}\mathrm{Pr}_{\bm{\delta}^*}\{(Z_*^x,Z_*^y)\in R_{mm}'\}\\
      \notag
      &\leq \sup_{\bm{\delta}^*:\delta^*_x\delta^*_y=
        0}\mathrm{Pr}_{\bm{\delta}^*}\{(Z_*^x,Z_*^y)\in R_{mm}'\} = \alpha.
    \end{align}
\end{theorem}
Thus, $R'_{mm}$ generates a size-$\alpha$ test. For all $\alpha$ between any two consecutive unit fractions, say $\alpha\in ((K+1)^{-1},K^{-1}]$, the difference between the lower bound in \eqref{eq:type1:inf} and the nominal type~1 error $\alpha$ is maximized at the midpoint, where it equals $\{4K(K+1)\}^{-1}=O(K^{-2})$. 
%As a consequence, type~1 error is minimized at $(\delta_x^*,\delta_y^*)=(0,0)$, where it will be $\lfloor\alpha^{-1}\rfloor\alpha^2+(1-\lfloor\alpha^{-1}\rfloor\alpha)^2$. Between any two consecutive unit fractions, say $K^{-1}$ and $(K+1)^{-1}$, the difference between this actual type~1 error and the nominal type~1 error $\alpha$ will be maximized at the midpoint, i.e., $\{K^{-1}+(K+1)^{-1}\}/2$, and will be $\{4K(K+1)\}^{-1}=O(K^{-2})$. 
Thus, the difference between the nominal and actual worst-case type~1 error is maximized at $\alpha=3/4$, where the actual worst-case type~1 error is 5/8, 1/8 less than the nominal type~1 error. However, as noted, this maximal difference shrinks at a quadratic rate in $\lfloor\alpha^{-1}\rfloor$, and is at most $1/1680\approx 0.0006$ for all $\alpha<0.05$.

\subsection{Asymptotic approximation considerations}
The results we have given so far have been in terms of the limit distribution, i.e., that of $(Z^x_*,Z^y_*)$. However, if the joint convergence of $(\hat{\delta}_x,\hat{\delta}_y)$ is uniform in $\bm{\delta}$ as in \eqref{eq:CLT}, then %the test statistic $(Z^x_n,Z^y_n)$ will inherit this uniform convergence such that 
the power $\beta_n(\delta_n):=\mathrm{Pr}_{\hat{\delta}_x,\hat{\delta}_y}\{(Z^x_n,Z^y_n)\in R_{mm}\}$ for the observed data distribution will converge uniformly to $\beta_*(\delta^*):=\mathrm{Pr}_{\bm{\delta}^*}\{(Z^x_*,Z^y_*)\in R_{mm}\}$, i.e., that of the limiting distribution. Thus, for any arbitrarily small tolerance $\epsilon>0$, there is an $N$ such that $\beta_n(\delta_n)$ will be within $\epsilon$ of $\beta_*(\delta^*)$ for all $\bm{\delta}$ and $n\geq N$. In particular, this means violations of type~1 error can be made arbitrarily small uniformly over $H_0$ given a sufficiently large sample size. The asymptotic approximation can be improved for a given finite sample $n$ by replacing all instances of the normal distribution with a $t$ distribution with $(n-1)$ degrees of freedom.
\begin{comment}
\begin{figure}[h]
\centering
 \includegraphics[width=.85\textwidth]{power-surface.pdf} 
 \caption{The power function surface for the minimax optimal test over $(\delta^*_x,\delta^*_y)$. The power function takes the value 0.05 everywhere on the $\delta^*_x$ and $\delta^*_y$ axes, and is strictly greater than 0.05 everywhere else, while going to one as $\delta^*_x$ and $\delta^*_y$ both go to infinity.}
 \label{fig:power}
\end{figure}
\end{comment}

\section{\centering BAYES RISK OPTIMAL TEST}
\label{sec:bayes}
We now consider a Bayes risk optimality criterion. We draw inspiration from and modify the testing procedure of \citet{rosenblum2014optimal}, which gives approximately Bayes risk optimal tests for the distinct purpose of simultaneously testing for treatment effects both in the overall population and in two subpopulations. Their approach was to discretize the test statistic space into a fine grid of cells, then cast the problem as a constrained optimization problem, where the unknown parameters %to optimize over 
are the rejection probabilities in each cell, resulting in a rejection region that optimizes their particular Bayes risk. %they defined specific to their problem. %This approach turns out to be easily adapted to our problem of testing the composite null hypothesis $H_0$.

%Let $M=M(Z^x_*,Z^y_*,U)$ be the randomized test mapping into $\{0,1\}$, where 1 corresponds to reject, 0 corresponds to not reject, and $U$ is a uniform random variable with support $[0,1]$. 
We let $M : \mathbb{R}^2 \times [0, 1] \rightarrow \{0, 1\}$ denote a generic testing function of $H_0$ against $H_1$ that characterizes the randomized test consisting of rejecting $H_0$ for $H_1$ if and only if $M(Z^x_*,Z^y_*,U) = 1$, where $U$ is a uniform random variable on $[0, 1]$. As in \citet{rosenblum2014optimal}, it is purely for computational reasons that we consider a randomized test, as doing so yields a linear programming problem rather than an integer programming problem. %, the latter being far more computationally burdensome than the former. %, especially for our fairly high-dimensional optimization problem. 
%the choice of a randomized test is purely for computational reasons, as it will yield a linear programming problem rather than an integer programming problem, the latter of which is far more computationally burdensome, especially for our optimization problem, which will be fairly high dimensional. 
However, our solution is almost entirely deterministic, with very few cells having non-degenerate rejection probabilities. Let %$L\left\{M(Z^x_*,Z^y_*,U);\delta_x^*,\delta_y^*\right\}$ 
$L:\{0,1\}\times \mathbb{R}^2\rightarrow\mathbb{R}$ be a bounded loss function, and $\Lambda$ denote a prior distribution on $\bm{\delta}^*$. We consider a 0-1 loss function given by
\[L\{M(Z^x_*,Z^y_*,U),\bm{\delta}^*\}=M(Z^x_*,Z^y_*,U)I(\delta_x^*\delta_y^*=0)+\{1-M(Z^x_*,Z^y_*,U)\}I(\delta_x^*\delta_y^*\neq
 0),\] 
and a bounded quadratic loss function given by
\[L\{M(Z^x_*,Z^y_*,U),\bm{\delta}^*\}=\{1-M(Z^x_*,Z^y_*,U)\}\{(\delta_x^*\delta_y^*)^2\wedge 4\Phi^{-1}(1-\alpha/2)^2\}.\] 
We consider $\mathcal{N}\{0,2\Phi^{-1}(1-\alpha/2)\bm{I_2}\}$ for the prior $\Lambda$. 
The constrained Bayes optimization problem is the following: For given $\alpha\in (0,1)$, $L$, and $\Lambda$, find the function $M$ minimizing %the Bayes risk
\begin{equation}
 \label{eq:bayesrisk}
\int E_{\bm{\delta}^*}\left[L\left\{M(Z^x_*,Z^y_*,U),\bm{\delta}^*\right\}\right]d\Lambda(\bm{\delta}^*)
\end{equation}
subject to the type~1 error constraint
\begin{equation}
  \label{eq:type:1:constraint}
  \mathrm{Pr}_{\bm{\delta}^*}\{M(Z^x_*,Z^y_*,U)=1\}\leq  \alpha \text{
    for all }
  \bm{\delta}^*\in \mathbb{R}^{2} \text{ such that }
  \delta_x^*\delta_y^* = 0.
\end{equation}
%$\mathrm{Pr}_{\delta_x^*,\delta_y^*}\{M(Z^x_*,Z^y_*,U)=1\}\leq \alpha$ for all $(\delta_x^*,\delta_y^*)\in H_0$.

%To formulate an approximation of this problem, we begin by defining the rejection region outside of $B:= [-b,b]\times[-b,b]$ to be the same as the joint significance test, where we select $b=2\Phi(1-\alpha/2)$. Specifically, define the rejection region outside of this square to be $R_{\bar{B}}:=\{(Z^x_*,Z^y_*):\lvert Z^x_*\rvert\geq \Phi(1-\alpha/2),\lvert Z^y_*\rvert\geq \Phi(1-\alpha/2)\}\setminus B$. 
To    formulate    an
  approximation of this problem, we begin  by making the test coincide outside
  of  $B:= [-b,b]\times[-b,b]$  with  the joint  significance  test, where  we
  select $b=2\Phi^{-1}(1-\alpha/2)$.  
  \begin{comment}
  In other words, we impose $M(Z^x_*,Z^y_*,U) I \{(Z^x_*,Z^y_*)\in \bar{B}\}
      = I \{(Z^x_*,Z^y_*)\in R_{\bar{B}}\}$, where $R_{\bar{B}}
      :=\{(\delta_x^*,\delta_y^*)   \in   \mathbb{R}^{2}   :   \lvert
      \delta_x^*\rvert\geq  \Phi^{-1}(1-\alpha/2),  \lvert  \delta_y^*\rvert\geq
      \Phi^{-1}(1-\alpha/2)\}\setminus B$ 
%  \begin{align*}
%    &M(Z^x_*,Z^y_*,U) I \{(Z^x_*,Z^y_*)\in \bar{B}\}
%      = I \{(Z^x_*,Z^y_*)\in R_{\bar{B}}\}, \quad \text{where}\\
%    &R_{\bar{B}}
%      :=\{(\delta_x^*,\delta_y^*)   \in   \mathbb{R}^{2}   :   \lvert
%      \delta_x^*\rvert\geq  \Phi^{-1}(1-\alpha/2),  \lvert  \delta_y^*\rvert\geq
%      \Phi^{-1}(1-\alpha/2)\}\setminus B.
%  \end{align*}
%The reason for this is that outside of $B$, points $(Z^x_*,Z^y_*)$ are very likely to be drawn from distributions far enough away from the origin that the joint significance test will only be negligibly conservative. 
and $\bar{B}:=\mathbb{R}^2\setminus B$. 
\end{comment}
This     is    because     the    event
  $(Z^x_*,Z^y_*) \not\in  B$ is only likely  under $\bm{\delta}^*$
  far enough  away from the origin  and very unlikely otherwise, so  that the joint
  significance test will only be negligibly conservative. We then discretize the null hypothesis space into a grid of $8K+1$ points and $B$ into a $(2K+1)\times (2K+1)$ grid consisting of squares $R_{k,k'}$, where $K$ is a suitably large integer. %Let $\mathcal{R}$ denote the collection of squares in the latter grid. 
  Within $B$, we define a class %$\mathcal{M}_\mathcal{R}$ 
  of random rejection regions whose realizations are constant within each square. That is, if $(z^x,z^y)\in R_{k,k'}$, then $M(z^x,z^y,u)=M(z^{x'},z^{y'},u)$ for all $(z^{x'},z^{y'})\in R_{k,k'}$ and all $u\in[0,1]$. The rejection probabilities corresponding to each square %$r\in\mathcal{R}$ by $m_r$. These 
  are the unknown arguments defining the rejection region that are to be optimized over. A discretized version of the constrained optimization problem is then defined corresponding to these grids. Collectively, these form a linear optimization problem that approximates the Bayes risk optimization problem. %in $(m_{r})_{r\in\mathcal{R}}$. %, where all other terms are known. 
  Full details %of the discretization and approximation to the optimization problem 
  are available in the Supplementary Materials. 

The plots in
Figure \ref{fig:bayes} show the solutions to the two linear approximation problems
with $\alpha=0.05$ and $K=64$.
\begin{figure}[h]
\centering
\begin{tabular}{ccc}
\includegraphics[width=.47\textwidth]{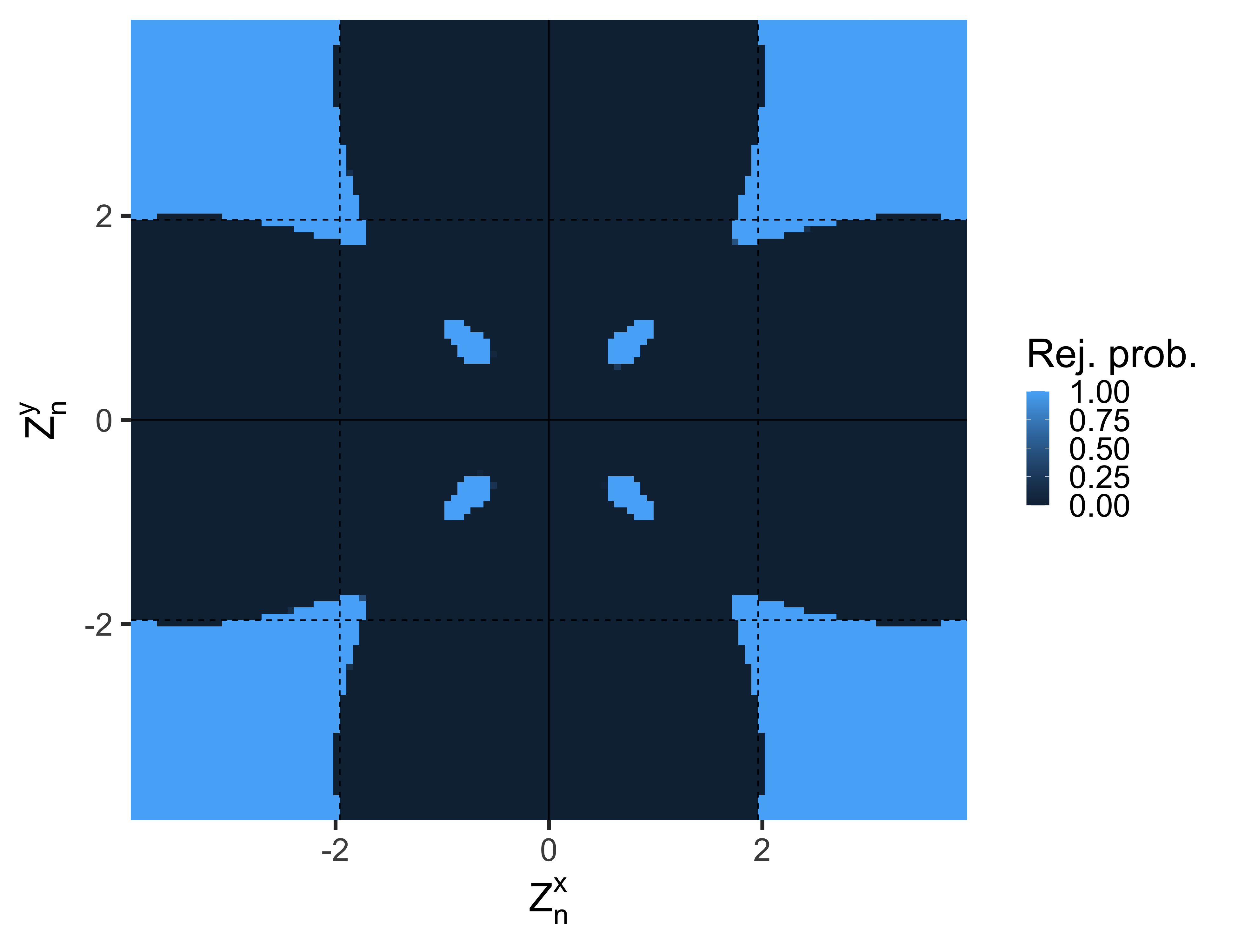}
& 
&
\includegraphics[width=.47\textwidth]{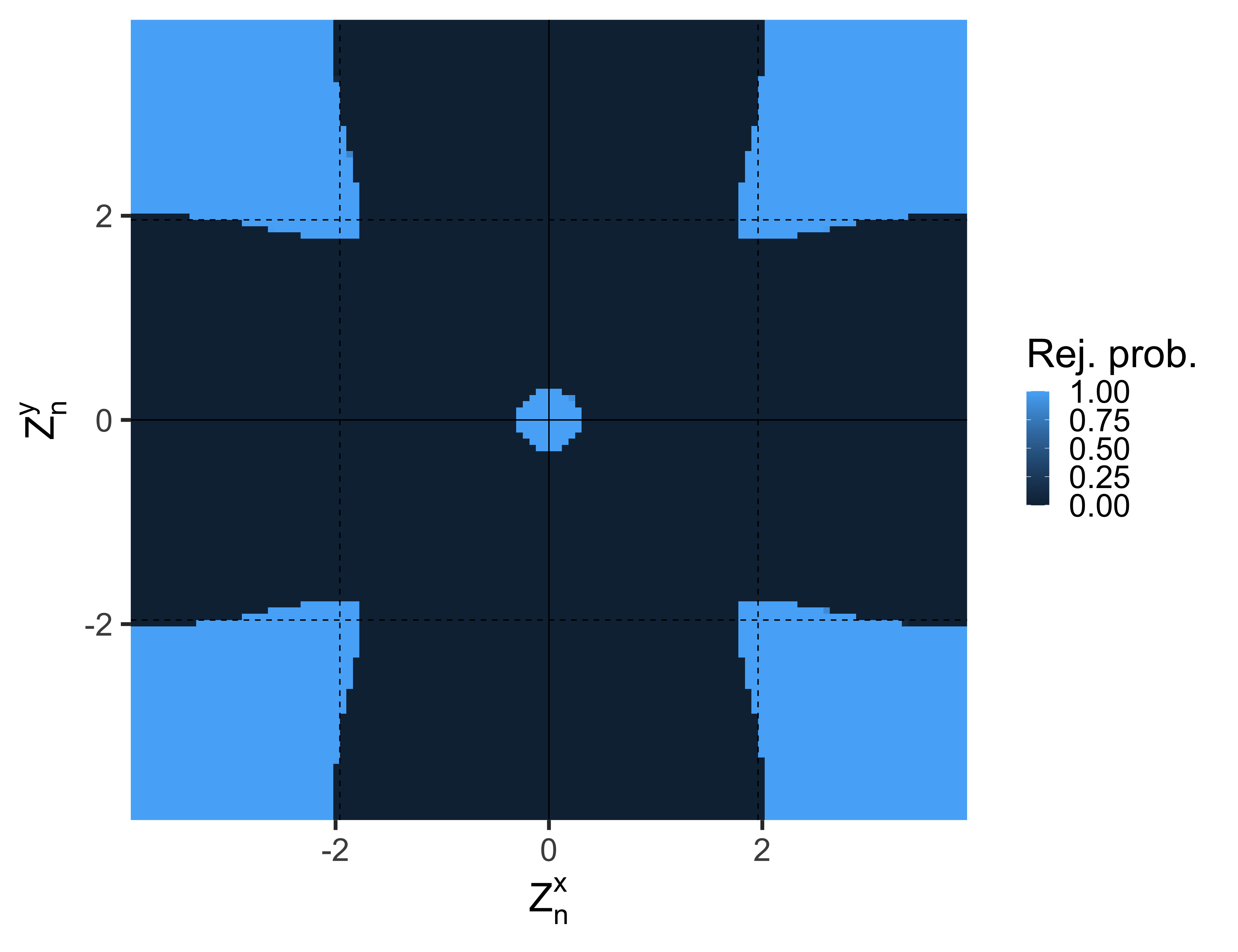}
\\
(a)	&	&	(b)
\end{tabular}
\caption{(a) The rejection probabilities of the approximate Bayes risk optimal test with respect to the (a) 0-1 loss function, and (b) bounded quadratic loss function. The rejection region for each is defined to be equivalent to that of the joint significance test outside of the displayed figure, which is outlined by the continuation of the dashed lines. The dashed lines indicate the boundaries of the rejection regions of the Wald tests for $H_0^x$ and $H_0^y$.}
\label{fig:bayes}
\end{figure}
For both tests, the rejection probabilities are nearly all zero or one. If one prefers a nonrandom test, one can obtain a slightly more conservative test by setting all nondegenerate probabilities to zero. %Although we did not enforce symmetry in this particular implementation of the optimization problem, the plots are nonetheless almost completely symmetrical. 
Both rejection regions almost contain the entire rejection region of the joint significance test. Moreover, both rejection regions take interesting, somewhat unexpected shapes, with four roughly-symmetric disconnected regions inside of the acceptance region in the case of the 0-1 loss, and a single region inside the acceptance region centered around (0,0) in the case of the quadratic loss. As with the minimax optimal test, it seems that it is important to have part of the rejection region along the diagonals at least somewhat close to the origin, so that smaller alternatives in both $\delta_x^*$ and $\delta_y^*$ have larger rejection probabilities to offset the conservativeness induced by lying in the intersection of two acceptance region bands. As with the minimax optimal test, one might reasonably take issue with rejecting test statistics that are arbitrarily close to the null hypothesis space. We discuss this further in Section \ref{sec:interpretability}.

While solving the linear programming problem to produce these tests takes a nontrivial amount of time to run, this need only be done once, after which the object is stored and can be quickly loaded on demand. The approximate Bayes risk optimal tests can be solved for any size $\alpha\in(0,1)$. However, the linear optimization problem must be solved separately for each value of $\alpha$, and so cannot be obtained immediately. %, unless it is a value for which the problem has been previously solved and the solution stored.

As with the minimax optimal test, the approximate optimality of the tests defined in this section is in terms of the limit distribution, i.e., of $(Z_*^x,Z_*^y)$. Assuming uniform convergence of $(\hat{\delta}_x,\hat{\delta}_y)$ as in \eqref{eq:CLT}, the (sample) approximations to the (computational) approximate Bayes risk optimality objective functions and type~1 error constraints can be made arbitrarily small uniformly over the parameter space given a sufficiently large sample size. As with the minimax optimal test, the asymptotic approximation can be improved by replacing all instances of the normal distribution with the $t$ distribution with $(n-1)$ degrees of freedom.

\section{\centering LARGE-SCALE HYPOTHESIS TESTING}
\label{sec:large-scale}
As mentioned above, applying the Bonferroni correction to the proposed methods in order to control for familywise error rate is straightforward even without the use of $p$-values. For $J$ hypotheses, one can simply run any of the proposed tests with level $\alpha/J$ instead of comparing the $p$-value to a threshold of $\alpha/J$. Adapting the Benjamini--Hochberg procedure to control false discovery rate requires a bit more explanation if $p$-values are not available (though the $p$-value proposed in the Supplementary Materials %of Section \ref{sec:pval} 
can be used, which will be conservative). A modified Benjamini--Hochberg procedure works as follows: (i) find the largest value $j$ such that at least $j$ of the $J$ tests reject at level $\alpha j/J$, (ii) reject all hypotheses that the test rejects at level $\alpha j/J$. It can be shown that this procedure controls the false discovery rate at $\alpha$ even for tests that are not monotone in $\alpha$. A (potentially) conservative version of this procedure can avoid an exhaustive search for the largest $j$ in step (i). This procedure starts with $j=1$ and increases $j$ until fewer than $\alpha j/J$ hypotheses are rejected for $\kappa$ consecutive iterations for some pre-specified $\kappa$, after which the largest $j$ in this sequence at which at least $\alpha j/J$ hypotheses are rejected is chosen for part (i).

%first ordering the hypotheses by $p$-value, and then only considering candidate $j$'s if the test rejects at level $\alpha j/J$ for the $j$th hypothesis.

%One important %and often overlooked 
%point 
When testing for the presence of multiple mediated effects, %one either needs to consider distinct exposures as well as mediators across the different tests, or 
    one needs to %be willing to 
    assume that the true mediators do not affect one another, which would result in some mediators being exposure-induced confounders, yielding the NIE non-identifiable. %\citep{avin2005identifiability}. %and that none of the candidate mediators confounds the effect of any of the others. 
    %It suffices to assume that no candidate mediator affects any other candidate mediator. 
    This may not be realistic in many scenarios; however, this problem is beyond the scope of this article.

\section{\centering INTERPRETATION AND MODIFICATIONS}
\label{sec:interpretability}
As we have pointed out, the rejection regions for the minimax and Bayes risk optimal %(which we will henceforth refer to as Bayes risk optimal for brevity) 
tests have some peculiar features that some may regard as undesirable. %or unacceptable. 
Tests with such properties are known to arise in multi-parameter hypothesis testing and have received criticism in \citet{perlman1999emperor} as being flawed, though the philosophical debate about what criteria should be used to compare among statistical tests is perhaps not entirely settled %as can be seen in the comments by 
\citep{berger1999comment}. %,mcdermott1999emperor}. %in response to \citet{perlman1999emperor}. 
%We find both sides of this debate compelling, and 
%Rather than choosing a side on this debate, %in this article, we %instead 
We take a pluralistic stance and provide a menu of options, among which we hope anyone across this ideological spectrum can find a test that suits their preferences. Furthermore, we discuss differences in interpretation of different tests. %In the extreme, we allow that some may not find any of the tests we propose to be acceptable, in which case we would recommend use of the joint significance test. %(which we will see can be recovered as a special case of a class of tests we will propose in this section).

%%% BRILLIANT!!!

At least two of the criticisms of \citet{perlman1999emperor} apply to our tests. First, the minimax optimal test can reject when the point estimate is arbitrarily close to the null. How then is one to interpret a test that rejects when both test statistics are small? One perspective %that one might take 
is that one need not rely entirely on the result of a hypothesis test to determine whether an association or an effect is %meaningful or 
``practically significant'', so long as the hypothesis test has correct type 1 error. One could always pre-specify a minimally meaningful effect size or proportion mediated, %(the natural indirect effect divided by the total effect), 
and never reject if the point estimate does not exceed it. 
%\citet{perlman1999emperor} argue that the relevant criterion for hypothesis testing is the following: ``based on the observed data, does the family of distributions represented by the alternative hypothesis $H_1$ fit (or support, or explain) the observed data significantly better than the family represented by the null hypothesis $H_0$''. One issue with this criterion, as pointed out by \citet{berger1999comment} is that this lacks a precise definition. 

Second, the minimax optimal test is not monotonic in $\alpha$, nor is the Bayes risk optimal test guaranteed to be. Thus, in the large-scale hypothesis testing case, it is possible for the Bonferroni- or BH-adjusted minimax optimal test to reject based on a particular test statistic when there are, say, $J$ hypotheses to test, but not when there are $J-1$. While this is true for a fixed test statistic, the expected number of rejected nulls will always be lower with a smaller number of hypotheses, and FWER or FDR are preserved for any number of hypotheses. In the context of large-scale hypothesis testing, one might care less about the interpretation of the rejection of a particular hypothesis, and more about minimizing false positives, as rejected nulls may be viewed as being flagged for further scrutiny in subsequent analyses. However, these are not the only options. In the Supplementary Materials, %Section \ref{sec:pval}, 
we introduce a $p$-value corresponding to the minimax optimal test, which %does not correspond one-to-one with the minimax optimal test. In fact, the $p$-value 
can itself be used to define a test that rejects at level $\alpha$ whenever the $p$-value is less than $\alpha$. Such a test will be monotone in $\alpha$ by construction. %However, without further modification, the $p$-value based test can reject test statistics that are arbitrarily small.
A constrained variant of the Bayes risk optimal test can be devised for a given decreasing sequence of $\alpha$'s by setting the rejection value to be no less %(greater) 
than the largest %(smallest) 
rejection probability among all tests with larger %(smaller) 
$\alpha$ values preceding it in the sequence. %The same realization of the randomization variable  $U$ ought to be used across tests to preserve monotonicity. 
Therefore, if one is concerned about coherence, then one might opt to use this $p$-value based test at the cost of power, but which is still more powerful than traditional tests while preserving FWER or FDR, or alternatively this constrained version of the Bayes risk optimal test.%Such a procedure will likely be sensitive to the sequence of $\alpha$ values. %, but will be 

To circumvent the first issue, we propose modifications of the minimax optimal and Bayes risk optimal tests. For the minimax optimal test, one can simply truncate the rejection region according to any user-specified distance, say $d$, from the null hypothesis space. If one wishes to preserve monotonicity in $\alpha$, then one can do the same for the $p$-value-based test we have just discussed. For a test with rejection region $R$, the rejection region of the corresponding truncated version will simply be $R \setminus\{(Z_n^x, Z_n^y)^{\top}: \lvert Z_n^x\rvert \wedge \lvert Z_n^y\rvert \leq d\}$. The rejection region of the truncated minimax optimal test with $d=0.1$ is shown in panel (a) of Figure \ref{fig:trunc}, %Of course, any cutoff can be chosen; we simply illustrate with 0.1 for
where we chose 0.1 for the sake of comparison with %because this is the distance of the rejection region of 
\citet{van2024nearly}. %from the null hypothesis space.
\begin{figure}[h]
\centering
\begin{tabular}{ccc}
\includegraphics[width=.45\textwidth]{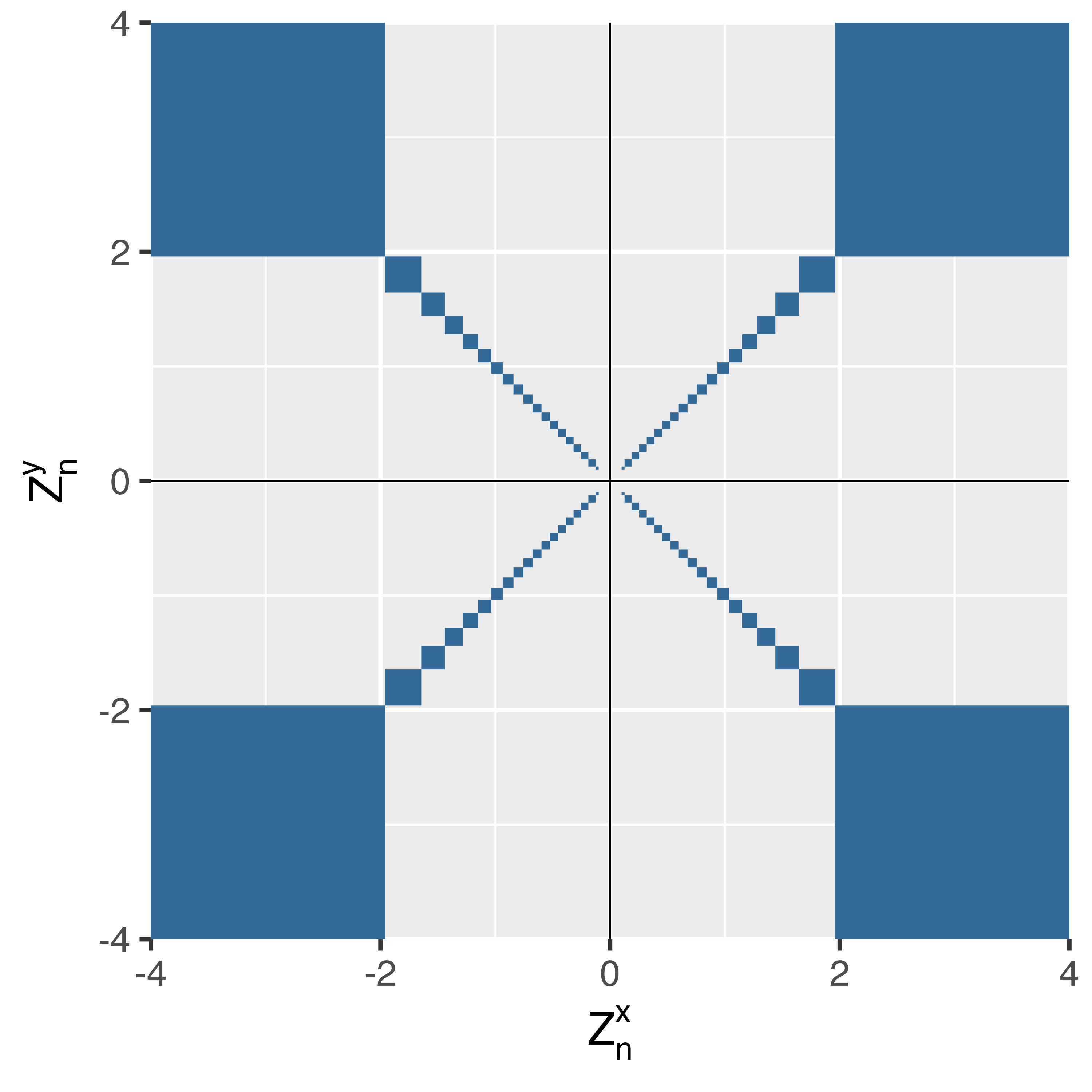}
& 
&
\includegraphics[width=.45\textwidth]{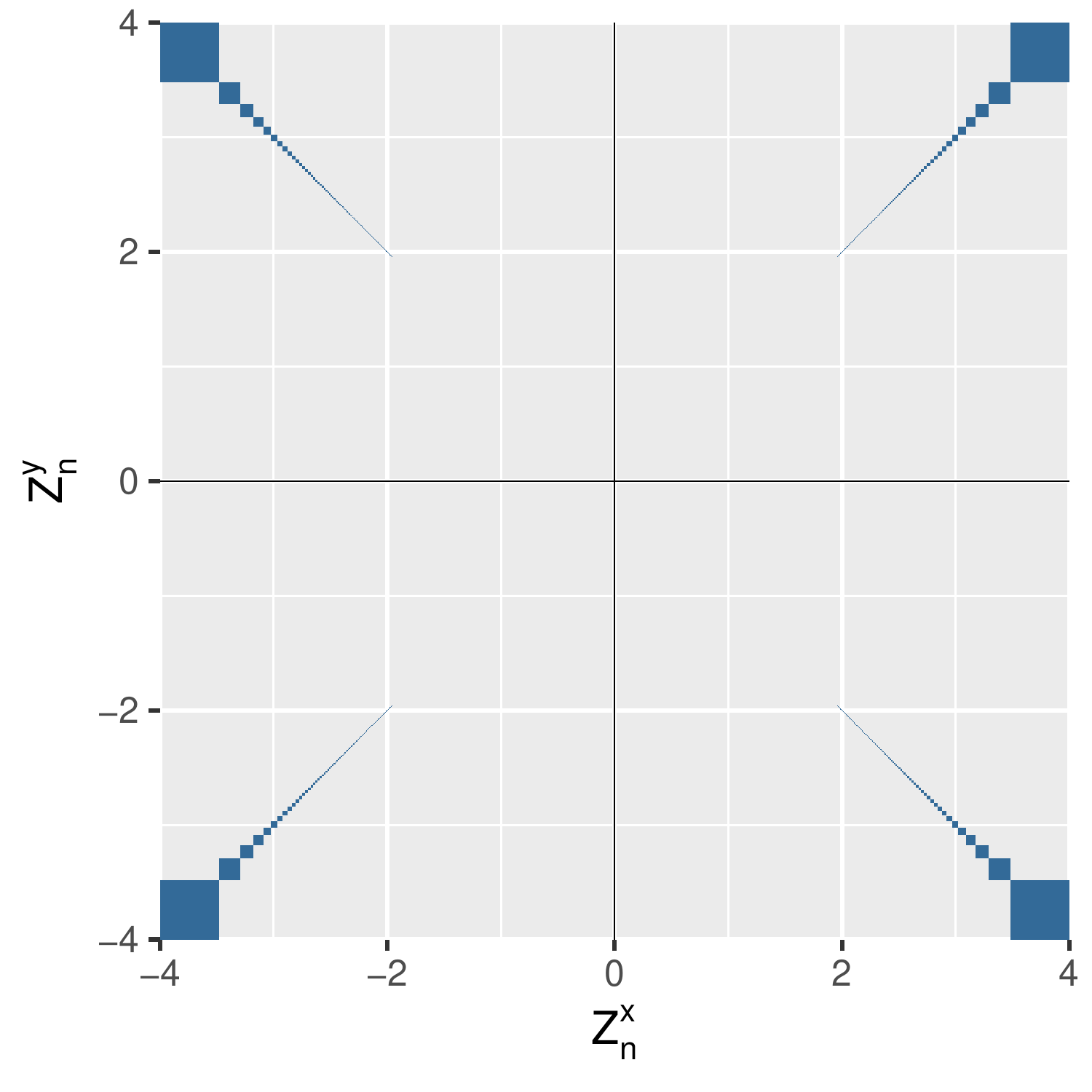}
\\
(a)	&	&	(b)
\end{tabular}
\caption{(a) The rejection region of the truncated minimax optimal test with $d=0.1$ and $\alpha=0.05$. (b) The rejection region of the truncated minimax optimal test with $d=\Phi^{-1}(0.975)$ and $\alpha=0.05/100$, as one might use for a Bonferroni correction when testing 100 mediation hypotheses.}
%Large-scale hypothesis test rejection region with Bonferroni correction for 100 tests with a joint significance test filter, i.e., truncating the rejection region by 1.96.}
\label{fig:trunc}
\end{figure}
Such truncated tests will clearly be dominated by their corresponding non-truncated versions in terms of power. %A more powerful approach would be to instead optimize a constrained version of the minimax optimality criterion, where the constraint is that the test cannot reject for test statistics such that $\lvert Z_n^x\rvert \wedge \lvert Z_n^y\rvert \leq d$. This is a challenging problem that we do not address here. %, but that could be a fruitful direction for future research. 
One rebuttal to the claim of \cite{perlman1999emperor} that the rejection region ought to be bounded away from the null hypothesis space is that this distance is not specified, nor is there a clearly-defined criterion specifying what this distance should be \citep{berger1999comment}. If one believes that this distance should be determined according to the likelihood ratio test, then one would set $d$ to be $\Phi^{-1}(1-\alpha/2)$, in which case the joint significance test (i.e., the likelihood ratio test) %\citep{liu2022large,van2024nearly}, 
would be recovered in the single hypothesis test case. %, and our test offers nothing novel. 
However, if one is conducting large-scale hypothesis testing and wishes to control for familywise error rate (for instance), and insists on a $d$ of $\Phi^{-1}(1-\alpha/2)$, then this does generate a novel test with greatly-improved power over the Bonferroni-corrected joint significance test. For example, for a Bonferroni correction for 100 tests at level $\alpha=0.05$, one can use the truncated minimax optimal test with $\alpha$ level $0.05/100$ and with $d=\Phi^{-1}(1-\alpha/2)$, which yields the rejection region in panel (b) of Figure \ref{fig:trunc}. That is, one can apply the multiple testing-adjusted minimax optimal test with level $0.0005$, and then use the joint significance test with $\alpha$ level $0.05$ to screen out effects deemed too small. %In this case, the Bonferroni-corrected minimax optimal 
The resulting test will have 100 times the worst-case power per test of that of the Bonferroni-corrected joint significance test. %, while still preserving familywise error rate and respecting the specified bound. %away from the null hypothesis space.

For the Bayes risk optimal tests, the same truncation could be easily applied. However, %unlike the minimax optimal test, 
we also formulate a constrained version of the Bayes risk optimization problem in this article. For $d=0.1$, the Bayes risk optimal test with 0-1 loss function will not change, since there is no part of the rejection region within 0.1 of the null hypothesis space. On the other hand, the Bayes risk optimal test with quadratic loss will change, since its rejection region contains a region centered on the origin %The rejection region of the latter with %constrained Bayes risk optimal test with quadratic loss and 
%$d=0.1$ is shown in 
(see the supplementary materials). %Figure \ref{fig:quad_constr}. %Interestingly, there remains a rough circle centered on the origin; however, now there is a cross cut out of the center, respecting our constraint. 
%Once again, one may prefer to choose other values of $d$.%---0.1 is merely chosen for illustration.
\begin{comment}
\begin{figure}[h]
\centering
 \includegraphics[width=.45\textwidth]{rej_region_quad_constr.png} 
 \caption{The rejection region of the constrained Bayes risk optimal test with quadratic loss function and $d=0.1$.}
 \label{fig:quad_constr}
\end{figure}
\end{comment}

Another important challenge has to do with the traditional interpretation of a hypothesis test rejection decision: if the null hypothesis were true, then the probability of observing an event as or more extreme than the observed event is at most $\alpha$. %In multi-parameter hypothesis testing settings, this interpretation can be ambiguous, as there is not a total ordering in the parameter space. Nevertheless, 
In the product of coefficients case, such an interpretation can be applied to a partial ordering, where one pair of test statistics is ``as or more extreme'' than another if both elements in the first pair are at least as large in magnitude as those in the second. \cite{van2022optimality} refer to a test that rejects for all pairs of test statistics that are more extreme %(in this sense) 
than another pair for which it also rejects as ``information coherent''. Clearly none of the tests we have proposed nor the test of \cite{van2024nearly} satisfy this property, hence these tests lack this traditional interpretation. %of a rejection decision. 
Instead, they admit a weaker interpretation: if the null hypothesis were true, then the observed event would have been unlikely to occur in the sense that its rejection probability would be at most $\alpha$. Alternatively, if the joint significance test and one of our proposed tests both reject, then the stronger interpretation can be applied. In a case where they disagree and one is applying one of our tests, then one must be prepared to accept the possibility that one pair of test statistics that are smaller in magnitude may reject while another does not. %this is entirely consistent with this weaker interpretation. 

\cite{van2022optimality} showed that the joint significance test is the most powerful test in the class of information coherent tests. Thus, there is a fundamental trade-off between power and preserving the former stronger interpretation over the latter weaker interpretation. %If one wishes to improve on the power of the joint significance test, then one has no choice but to construct a test lacking the property of information coherence. %There is a fundamental trade-off between power and this particular notion of interpretability. 
If one deems the underpoweredness of the joint significance test to be an important problem, then one has no choice but to sacrifice a degree of interpretability. In proposing more powerful tests, %we do not intend to force a decision upon anyone about how to navigate this trade-off. Rather, 
we %merely 
wish to present options to gain power at the cost of some interpretability for those who would deem it a worthwhile trade-off. %Ultimately, it is up to the user to make this decision.
%One may reasonably feel uncomfortable making such a trade-off when scientific conclusions or policy decisions are to be based on the result. However, 
One may feel more comfortable with this trade-off in large-scale hypothesis testing scenarios where the focus is on making multiple discoveries while controlling the number of errors. Thus, our proposed methodology is perhaps best suited for this latter setting.

Lastly, we reiterate that the validity of each of our proposed tests is predicated on satisfaction of identification assumptions. As is typical when comparing tests, the most pronounced improvements in power occur when effect sizes are small relative to standard errors. However, analyses in such cases are also qualitatively more vulnerable to identification failures, such as residual confounding. Thus, it is essential to prioritize rigorous study design, comprehensive covariate measurement, and domain-informed confounder selection before concerning oneself with maximizing power. %, as enhanced power to test the wrong identification formula will merely make one more likely to arrive at the wrong qualitative conclusion.

\section{\centering SIMULATION STUDIES}
\label{sec:simulations}
\subsection{Single mediation hypothesis testing}
To compare the finite sample performances of our proposed tests %the proposed minimax optimal and Bayes risk optimal tests 
with those of the others we have discussed, %the delta method, joint significance, and \cite{van2024nearly} (henceforth vGvG) tests, 
we performed a simulation study in which we sampled from independent $t$ distributions with %identity covariance matrix and 
noncentrality parameters $\bm{\delta}$. We considered four scenarios: (a) We varied $\delta_x$ and $\delta_y$ jointly from 0 to 0.4. (b) We fixed $\delta_y$ to 0.2 and varied $\delta_x$ from 0 to 0.4. (c) Using a normal approximation, we fixed $\delta_y$ to 0 and varied $\delta_x$ from 0 to 0.4, thereby exploring empirical type 1 error across a range of parameter values within the composite null hypothesis space. (d) We used the same setting as in (c), but using a $t$-distribution with $(n-1)$ degrees of freedom approximation instead of the normal approximation. We drew 100,000 Monte Carlo samples with %of sample size 
$n=50$ for each value of $\delta_x$, and applied each test of $H_0:``\delta_x\delta_y=0"$ against $H_1:``\delta_x\delta_y\neq 0"$ to each sample. In particular, we used three versions of the Bayes risk optimal test: the tests with either the 0-1 loss or quadratic loss function, and the constrained test with quadratic loss function. Likewise, we used three version of the minimax optimal test: the standard version, the test based on the corresponding $p$-value defined in the Supplementary Materials, %Section \ref{sec:pval}, 
and a truncated minimax optimal test. In both the constrained and truncated tests, we used $d=0.1$ in order to make a fair comparison with the \cite{van2024nearly} (henceforth vGvG) test. %, whose rejection region is 0.1 away from the null hypothesis space. %in Euclidean distance. 
%In all cases, we used the normal approximation, and in the null case in which $\delta_y$ is fixed at zero, we also used tests based on the $t$-distribution approximation. 
%For the vGvG test, we used a quantile-quantile transformation of their boundary function in order to produce the $t$-distribution-based version. 
The Monte Carlo power estimates are displayed in the plots in Figure \ref{fig:sims}.
\begin{comment}
\label{sec:simulations}
\begin{figure}[h]
\centering
 \includegraphics[width=.8\textwidth]{sims_power.pdf} 
 \caption{Monte Carlo rejection probabilities of the minimax optimal (MMO), Bayes risk optimal (BRO), and joint significance (JS) tests with $\delta=\delta_x=\delta_y$ varying from 0 to 0.4. The nominal $\alpha$ level 0.05 is indicated by the dashed horizontal line.}
 \label{fig:sims}
\end{figure}
\end{comment}
\begin{figure}%[H]
\centering
\begin{tabular}{ccc}
\includegraphics[width=.47\textwidth]{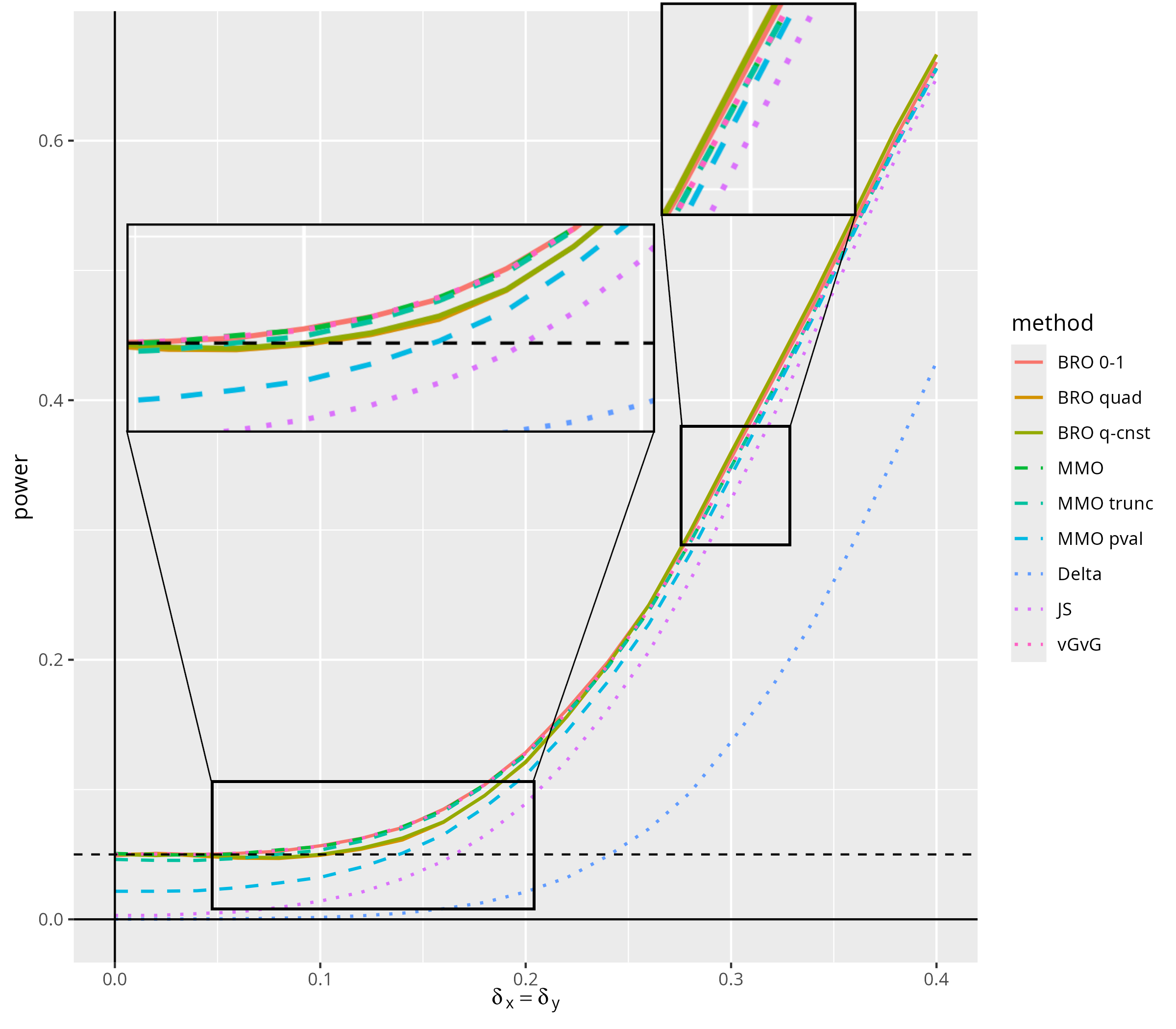}%{sims_power_comparisons_2.png}
& 
&
\includegraphics[width=.47\textwidth]{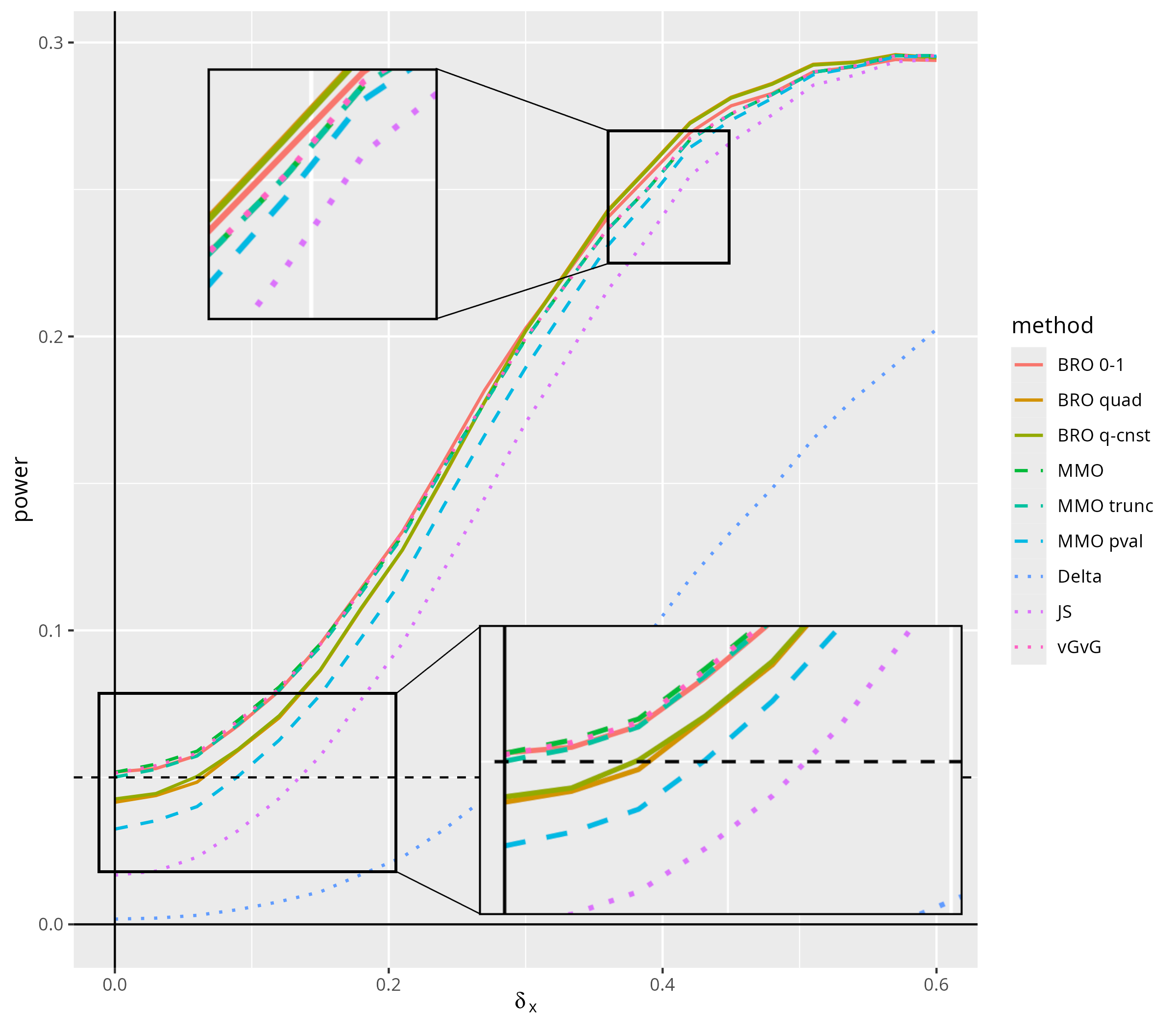}%{sims_power_comparisons_horiz_2.png}
\\
(a)	&	&	(b) \\
&& \\
\includegraphics[width=.47\textwidth]{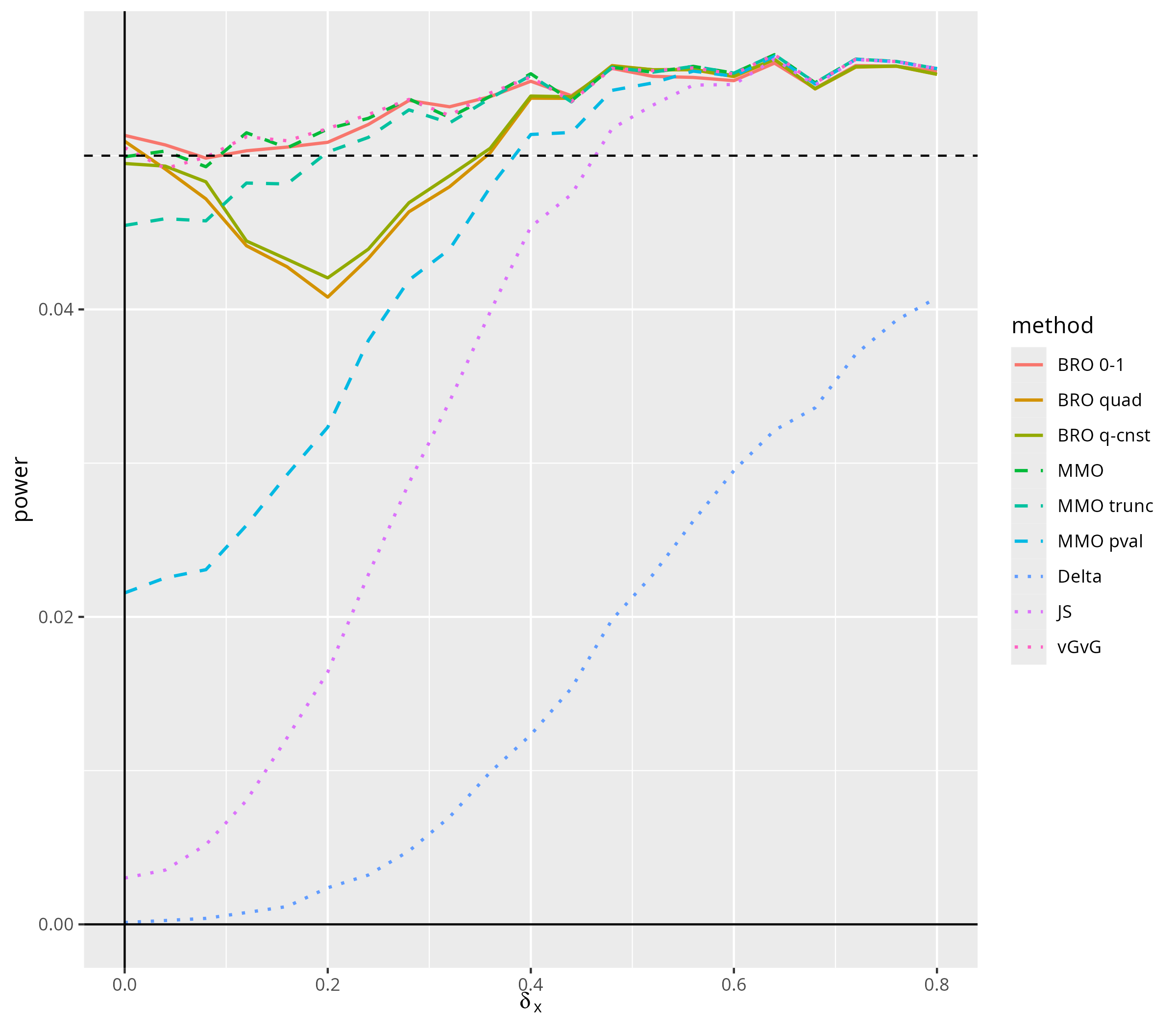}%{sims_null.pdf}
& 
&
\includegraphics[width=.47\textwidth]{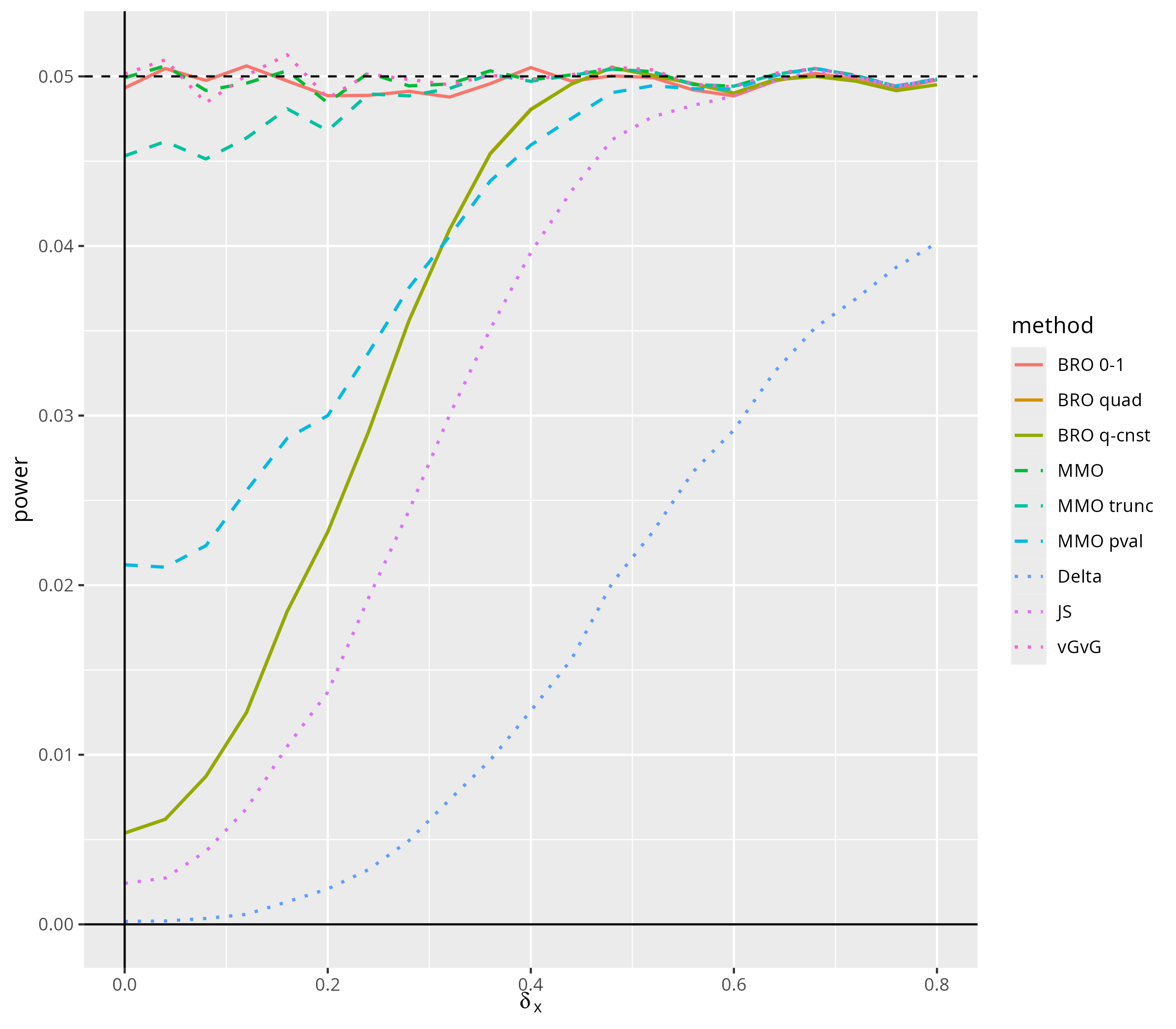}%{sims_null_tdist.pdf}
\\
(c)	&	&	(d) \\
\end{tabular}
\caption{Monte Carlo rejection probabilities of the following: Bayes risk optimal with 0-1 loss (BRO 0-1), Bayes risk optimal with quadratic loss (BRO quad), constrained Bayes risk optimal with quadratic loss (BRO q-cnst), delta method (Delta), joint significance (JS), minimax optimal (MMO), $p$-value-based test of the minimax optimal test (MMO pval), truncated minimax optimal (MMO trunc), and vGvG. %The nominal $\alpha$ level 0.05 is indicated by the dashed horizontal line. 
(a) $\delta_x=\delta_y$ varying from 0 to 0.4. (b) $\delta_y$ fixed at 0.2, $\delta_x$ from 0 to 0.6. (c) $\delta_y$ fixed at 0, $\delta_x$ from 0 to 0.8 with normality-based tests. (d) $\delta_y$ fixed at 0, $\delta_x$ from 0 to 0.8 with $t$-distribution-based tests. In panel d, the BRO quad and BRO q-cnst lines are identical.} 
\label{fig:sims}
\end{figure}
In scenario (a), apart from the minimax optimal $p$-value-based test, all of the minimax optimal and Bayes risk optimal tests and the vGvG test have very close to nominal type~1 error at $\bm{\delta}=(0,0)^{\top}$. %in all cases. 
All of these tests greatly outperform the delta method and joint significance tests in terms of power for smaller values of $\delta_x=\delta_y$, with the minimax optimal $p$-value-based test having power somewhere in between, but converging more quickly to the other tests. %than the joint significance test. 
The delta method test remains very conservative over the entire range of parameter values. %All of the tests except for the delta method test begin to converge in power closer to 0.4. 
The minimax optimal (apart from the $p$-value based test), Bayes risk optimal, and vGvG tests perform very similarly over the range of parameter values, with the quadratic loss versions of the Bayes risk optimal test trading off some power loss in the 0.1--0.2 range for a slight improvement in power in the 0.3--0.4 range, demonstrating the greater emphasis the quadratic loss places on larger alternatives. The truncated minimax optimal test begins slightly conservative relative to 0.05, but catches up with the most powerful tests quickly. 
%, with the former having a very slight advantage (at most about 0.008) for values of $\delta_x=\delta_y$ greater than about 0.25, and the latter having an even slighter advantage (at most 0.002) for most smaller values of $\delta_x=\delta_y$. 
Despite not being similar tests, the Bayes risk optimal and vGvG tests suffer very little in terms of power near the least favorable distribution at $\bm{\delta}=(0,0)^{\top}$.

In scenario (b), the trends are largely the same as in scenario (a). The main difference is that the quadratic loss versions of the Bayes risk optimal test are slightly conservative %with respect to 0.05 
under the null, %i.e., when $(\delta_x,\delta_y)=(0,0.2)$, 
albeit less conservative than the minimax optimal $p$-value-based test. These tests trade off some power loss for smaller values of $\delta_x$ ($\lesssim$0.25) for slightly improved power for larger values of $\delta_x$ ($\approx$0.35--0.55). Since $\delta_y$ is fixed at 0.2, power for all tests appears to plateau a little under 0.3. In scenarios (c) and (d), we see that all tests approximately preserve type 1 error; however, there is some anti-conservative behavior in scenario (c) due to the normal approximation being a bit too concentrated for a sample size of 50. This is corrected with the use of the $t$-distribution approximation, which yields type 1 error much closer to 0.05 across the %entire range of the
null hypothesis parameter space. %we consider. 
In both cases, we see the performances of all tests except the delta method to converge by around $\delta_x=0.55$. %, due to the identical shape of the rejection region for large values of $Z_n^x$.

\subsection{Large-scale mediation hypothesis testing}
\cite{du2023methods} ran a large simulation study comparing a number of methods for performing large-scale hypothesis testing. We reproduce one of their simulation settings and apply our proposed methodology to compare performance with the methods they compared. In particular, we performed 2,000 replications under their setting (a) with $J=100,000$ mediated effect null hypotheses, $n=200$ (sample size), $\tau=0.1$, which controls how dispersed the true parameters $\delta_x$ and $\delta_y$ are under their respective alternatives, and $H_0^x$ and $H_0^y$ hold for 99.8\% of tests, $H_0^x$ and $H_a^y$ hold for 0.1\% of tests, and $H_a^x$ and $H_0^y$ hold for 0.1\% of tests. See \cite{du2023methods} for more details about the simulation set-up. We compare the results of our proposed minimax optimal test with those of the tests compared in \cite{du2023methods}. The ratios of the actual false positive rates (FPRs) to the true FPRs at different cut-off values are displayed in Table \ref{table:sims} (except for the Sobel/delta method test, which performs worse than the joint significance test). The minimax optimal test yields FPR ratios closer to one and with smaller standard deviations than any other test apart from the joint significance test, which was excessively conservative.

\begin{table}[h]
%\begin{center}
{\centering
\caption{Mean (standard deviation) of the ratio of the actual FPR to the nominal FPR using the joint significance, JT-comp \citep{huang2019genome}, HDMT \citep{dai2022multiple}, Sobel-comp \citep{du2023methods}, DACT \citep{liu2022large}, and minimax optimal tests.}
\label{table:sims}
        \begin{tabular}{l r@{.}l @{ } r@{.}l r@{.}l @{ } r@{.}l r@{.}l @{ } r@{.}l r@{.}l @{ } r@{.}l r@{.}l @{ } r@{.}l r@{.}l @{ } r@{.}l}% r@{.}l @{ } r@{.}l }
        \\
        \hline\hline
        Cutoff   & %\multicolumn{4}{c}{Sobel} & 
        \multicolumn{4}{c}{Joint sig.} & \multicolumn{4}{c}{JT-comp}& \multicolumn{4}{c}{HDMT} & \multicolumn{4}{c}{Sobel-comp}           & \multicolumn{4}{c}{DACT} & \multicolumn{4}{c}{MM-opt}   \\
         \hline
        $1$e$-3$     & %0&00 & (0&00) & 
        0&00 & (0&01) & 1&11 & (0&10) & 1&02 & (0&16) & 0&89 & (0&34) & 1&47 & (1&81) & 1&00 & (0&10) \\
        $1$e$-4$     & %0&00 & (0&00) & 
        0&00 & (0&01) & 1&48 & (0&37) & 1&02 & (0&46) & 0&81 & (0&66) & 2&90 & (4&19) & 1&00 & (0&32) \\
        $1$e$-5$     & %0&00 & (0&00) & 
        0&00 & (0&04) & 3&02 & (1&70) & 1&17 & (1&30) & 1&04 & (1&37) & 6&54 & (10&61) & 0&98 & (0&99) \\
        $1$e$-6$     & %0&00 & (0&00) & 
        0&01 & (0&32) & 9&99 & (10&09) & 1&81 & (4&46) & 1&81 & (4&63) & 16&50 & (31&56) & 1&01 & (3&21) \\
        %$5 \cdot 10^{-7}$
        $5$e$-7$ & %0&00 & (0&00) & 
        0&00 & (0&00) & 15&19 & (17&36) & 2&11 & (6&62) & 2&32 & (7&17) & 22&65 & (47&46) & 0&91 & (4&31) \\
        \hline
        %\\
\end{tabular}}
%\end{center}
\end{table}

\section{\centering DATA EXAMPLES}
\label{sec:data-analysis}
\subsection{DCTRS Data Analysis}
We applied our methodology to data from the Database of Cognitive Training and Remediation Studies (DCTRS), which consists of data from several randomized trials testing the efficacy of cognitive remediation therapy (CRT) in patients with schizophrenia. CRT targets patients' cognitive outcomes with the long-term goal of cognitive gains translating into improvements in %more distal outcomes like 
functioning and quality of life. %In particular, 
We used data from three trials described in \citet{wykes1999effects}, 
\citet{wykes2007cognitive-a}, and \citet{wykes2007cognitive-b} consisting of 128 %136 
patients with complete treatment, mediator, and outcome data. In each study, patients were randomized to a CRT arm or a control arm. %The purpose of this analysis is merely to illustrate the proposed methods and compare them with other methods. 
There are a number of issues (e.g., potential unmeasured confounding, variable follow-up time, and the fact that the data come from studies with different protocols) that we do not entirely account for in this analysis, and so our results should be viewed as purely illustrative of the proposed methodology.

Our exposure, $A$, was an indicator of whether the patient was assigned to the treatment arm; our outcome, $Y$, was the patient's Rosenberg self-esteem scale score %social behavior scale score---a functioning outcome measure---
at follow-up, and the mediator, $M$, was the %Trailmaking Test Part A (TMTA) time to completion---a measure of processing speed---%
patient's Wechsler Adult Intelligence Scale (WAIS) working memory digit span test %Wisconsin Card Sorting Test (WCST) non-perseverative errors---a measure of executive dysfunction---
at their end of treatment, which measures both working memory and attention. 
%The mean social behavior scale score was 10.7, and the sample range was from 0 to 46. 
The mean self-esteem scale score was 33.6 with sample range from 17 to 50. Since the studies were all randomized trials, we only need to assume that variables included in $\bm C$ and $A$ are sufficient to control for confounding of the effect of $M$ on $Y$. Here, we let $\bm C$ consist of the patients' sex, race, education, which study they participated in, and their self-esteem scale and working memory digit span test scores %social behavior scale and WCST non-perseverative errors 
measured at baseline. We provide more detailed discussion about the variable choices and modeling assumptions as well as unadjusted and more comprehensively adjusted analyses in the Supplementary Materials. However, we remind the reader that the assumption of no unmeasured mediator-outcome confounding is inherently untestable, and covariate selection should ideally be based primarily on causal understanding and domain expertise. We prioritize the results in Table \ref{table:dctrs}, as these are based on the most relevant subject-matter based covariate selection rather than data-driven selection, as the latter may also introduce post-selection inference problems. All tests agree in the additional analyses based on the more comprehensive covariate adjustment presented in the Supplementary Materials. The relationship between covariate adjustment and agreement of the various tests is complex. Adjusting for more covariates can influence the underlying test statistics in either direction, and doing so does not consistently make the various tests more or less likely to agree. 
Previously, \cite{wykes2007cognitive-b} found cognitive remediation to improve working memory digit span test scores; 
\cite{wykes1999effects} found cognitive remediation to improve Rosenberg self-esteem scale scores. 
%Furthermore, \cite{wykes2007cognitive-a} found suggestive evidence of an association between WCST categories achieved and social behavior scale score. 
We are testing whether there is an effect of cognitive remediation on patients' self-esteem that is mediated by its effect on working memory and attention. 

%Statistically significant as well as suggestive evidence has been found of a protective effect of cognitive remediation therapy on cognition as measured by various components of the Wisconsin Card Sorting Test (WCST) \citep{delahunty1993specific,wykes1999effects,wykes2007cognitive-a,fiszdon2016cognitive,reeder2017new}. %, and by WCST perseverative errors \citep{reeder2017new,fiszdon2016cognitive}. 
%Furthermore, \cite{wykes2007cognitive-a} found suggestive evidence of an association between WCST categories achieved and social behavior scale score. We are interested in testing whether there is an effect of cognitive remediation therapy on patients' social behavior scale score that is mediated by their WCST non-perseverative errors. 

We fit models \eqref{eq:model1} and \eqref{eq:model2bis} using OLS, and estimated the NIE by the corresponding estimator given in Section \ref{sec:prelim}. %To account for between-study heterogeneity, we tested for and included certain interactions with study ID in these models. %Further details are provided in the supplementary materials. 
We then applied the three versions of the minimax optimal test and the three versions of the Bayes risk optimal test from Section \ref{sec:simulations}, and compared these with the joint significance, delta method, and vGvG tests, all at level $\alpha=0.05$, using the t-distribution approximation. The estimated total effect is 0.70, which is interpreted as CRT modestly improving the self-esteem scale score by an average of 0.70 (0.10 standard deviations). The estimated natural indirect effect is 0.29, %0.51 %-0.80, 
which is interpreted as CRT modestly improving self-esteem by an average of 0.29 (0.04 standard deviations) through its effect on working memory and attention as measured by the %WAIS working memory 
digit span test. 
The estimated proportion mediated is 41.0\%. The test statistic is plotted on the rejection regions of the minimax and Bayes risk optimal tests in Figures \ref{fig:scatter_dctrs_bro01}--\ref{fig:scatter_dctrs_mmo_trunc} of the Supplementary Materials. 
%increasing (i.e., worsening) the social behavioral scale score by an average of 0.51 (or 0.06 standard deviations of the social behavior scale score) through its effect on cognition as measured by the WCST non-perseverative errors. Hypothesis test results and corresponding $p$-values are given in Table \ref{table:dctrs}. 
\begin{table}[h]
{\centering
%\begin{center}
\caption{Hypothesis test results for the NIE}
\label{table:dctrs}
\begin{tabular}{lc r@{.}l}
\\
\hline\hline
Hypothesis test & Reject &\multicolumn{2}{c}{$p$-value}\\
\hline
Bayes risk optimal, 0-1 loss & No & \multicolumn{2}{c}{Undefined} \\
Bayes risk optimal, quadratic loss & No & \multicolumn{2}{c}{Undefined} \\
Bayes risk optimal, quadratic loss, constrained & No & \multicolumn{2}{c}{Undefined} \\
Minimax optimal & Yes & 0&039 \\
Minimax optimal, truncated & Yes & 0&039 \\
Minimax optimal, $p$-value & Yes & 0&039 \\
Delta method & No & 0&256 \\
Joint significance & No & 0&115 \\
van Garderen and van Giersbergen    &   Yes  &   \multicolumn{2}{c}{Undefined} \\
\hline
%\\
\end{tabular}\par
}
%\end{center}
\end{table}

Hypothesis test results are presented in Table \ref{table:dctrs}, where we observe the improved power of the different versions of the minimax optimal test to reject the null hypothesis that the NIE is zero in favor of the alternative that it is nonzero. The interpretation %of the result 
for any of these tests is that if the true NIE were in fact zero, then the event we observed would have been unlikely to have occurred in the sense that the test would have rejected with at most 5\% probability. Since these tests disagree with the joint significance test, they lack the stronger interpretation that the probability of observing an event as or more extreme would be at most 5\% if the true NIE were zero. Having established that what we observed is ``unlikely'' in the former sense, one may proceed by examining whether the estimated effect size is %practically 
meaningful. %The estimated 
An effect size of 0.04 standard deviations of the self-esteem scale score may not be deemed relevant for understanding the mechanism by which CRT affects self-esteem. However, had the estimate been larger, one might conclude that there was a meaningful indirect effect through working memory, and that such an estimate would have been unlikely to have been observed if there were no true underlying indirect effect. Alternatively, one might consider a proportion mediated exceeding 10\%, say, to be meaningful, and would then conclude that the indirect effect through working memory indeed explains a meaningful amount of the total effect of cognitive remediation. However, as the total effect is not statistically significant in this case, the proportion mediated may be a less compelling criterion for evaluating practical significance. %Of course, this may be a type~1 error; however, if this is the case, it has only occurred with 5\% probability. 
%This is also a case where the minimax optimal tests and Bayes risk optimal tests disagree, as the latter all fail to reject. %as well as the minimax optimal $p$-value based test. 
In this case, neither of the traditional tests nor any of the Bayes risk optimal tests reject, but the vGvG test does reject, agreeing with the minimax optimal tests. %and its truncated and $p$-value based versions. 
While we have conducted nine tests for purposes of illustration, in practice one should decide \emph{a priori} which test to use based on how they wish to navigate the power--interpretability trade-off to avoid data dredging.

%We also report the $p$-value corresponding to the minimax optimal test being slightly greater than the significance level of 0.05, reflecting
%both the fact that our $p$-value definition does not correspond one-to-one to the hypothesis test at a particular level since it consists of non-nested rejection regions, and the somewhat conservative nature of this $p$-value, though it is far less so than those of the joint significance and delta method tests. We stress here that were we to specify \emph{a priori} that our analysis would be based on the minimax optimal test, we would conclude that we reject $H_0$ in favor of $H_1$ at $\alpha$ level 0.05 even though the $p$-value exceeds 0.05, as discussed in the Supplementary Materials.% Section \ref{sec:pval}.

\subsection{Normative Aging Study Analysis}
\cite{liu2022large} performed large-scale hypothesis testing on the Normative Aging Study (NAS) to test the mediated effect of smoking status on forced expiratory flow at 25\%--75\% of the Forced Expiratory Vital capacity, %(FEF$_{25–75\%}$), 
a measure of lung function, through 484,613 DNA methylation CpG cites among men in Eastern Massachusetts. We apply our proposed methodology to this data and compare results. See \cite{liu2022large} for details regarding this data. Based on the minimax optimal test with a Bonferroni correction controlling FWER at 0.05, we detect significant mediated effects through sites cg03636183, cg05575921, cg06126421, and cg21566642. The conservative version of the Benjamini--Hochberg correction controlling FDR at 0.05 additionally detects a mediated effect through the site cg05951221. The Bonferroni procedure ran locally on a MacBook Pro with 2.3 GHz Intel Core i5 in 7.3 seconds; the Benjamini--Hochberg procedure in 18.9 seconds. \cite{liu2022large} note that each of these sites has previously been found to be related to smoking status and/or lung cancer risk. 
%Using the $p$-value corresponding to the minimax optimal test with Benjamini-Hochberg correction to control FDR at 0.05, we detect significant mediated effects through ...
%In comparison, 
The FDR-adjusted joint significance test also identified all of these, while the FWER-adjusted joint significance test identified all except cg05951221. The FDR-adjusted DACT method \citep{liu2022large} found all of these to be significant as well as 14 others, the latter of which %they did not make any connections to previous findings in the literature. %These additional hits 
do not appear to have been found to be linked with either smoking or lung cancer. %These can potentially be addressed in subsequent confirmatory analysis; however, 
In our simulation study the DACT was found to have an inflated proportion of false positives---over 22 times that of the nominal proportion when using the appropriate cutoff---whereas our minimax optimal test never identified false positives at a rate any higher than 1.01 times the nominal proportion. These considerations cast doubt on whether the additional 14 CpG sites identified by the DACT method but not by ours are true positives. Thus, the FDR-adjusted minimax optimal test identified more CpG sites than the standard joint significance or delta method tests while also preserving approximate FDR control as shown theoretically and in the simulation study.

\section{\centering DISCUSSION}
\label{sec:discussion}
We have proposed two novel classes of tests for the composite null hypothesis $H_0:``\delta_x\delta_y=0"$ against $H_1:``\delta_x\delta_y\neq 0"$, which arises commonly in the context of mediation analysis. This is a challenging inferential task due to the non-uniform asymptotics of univariate test statistics of the product of coefficients. %, as well as the nonstandard structure of the null hypothesis space. 
We have constructed these tests to be both optimal in a decision theoretic sense, and to preserve type~1 error uniformly over the null hypothesis space---exactly so on both counts in the case of the minimax optimal test, and approximately so in the case of the Bayes risk optimal test. We have described procedures for carrying out large-scale hypothesis testing of many product-of-coefficient hypotheses, controlling for both familywise error rate and false discovery rate. %We defined a generalized $p$-value corresponding to the extended minimax optimal test. 
We have also considered some shortcomings of these tests in terms of interpretability, and have developed modifications to remedy some of these. Each of these comes at some cost of power, and we have illustrated a fundamental trade-off between the objectives of power and interpretability. Lastly, we have provided an R package to implement the proposed tests.

It is natural to ask which of the proposed tests should be used in practice. %There are a few trade-offs between these two tests. As seen in the simulation study, the Bayes risk optimal test offers very slight improvements in power for moderate effect sizes relative to standard error that are greater than the even slighter improvements in power offered by the minimax optimal test for smaller effect sizes. 
This is highly subjective, and depends on the criteria by which one judges hypothesis tests, as well as how much value one places on power in different parts of the alternative hypothesis space. We summarize the trade-offs with respect to power, scalability for large-scale hypothesis testing, and interpretability of the various tests discussed in this article in Table \ref{table:summary}. 
\begin{table}[h]
%\begin{center}
{\centering
\caption{Summary of properties of various tests}
\label{table:summary}
\begin{tabular}{lccccc}
\\
\hline\hline
Hypothesis test & Power & Scalability to & Monotonicity & Information & Bounded away\\
 &  &   mult.~testing  & in $\alpha$ & coherence & from null\\
\hline
Bayes risk optimal: &  &  &	    &	    &	 \\
-- 0-1 loss & highest & no &	can be    &	no    &	yes \\
-- quad.~loss  &	highest   &	no    &	can be    &	no    &	no \\
-- quad.~loss, constr. &	higher &	no    &	can be    &	no    &	yes \\
Minimax optimal: &  &  &	    &	    &	 \\
-- standard    &	highest   &	yes   &	no    &	no    &	no \\
-- truncated    &	higher    &	yes   &	no    &	no    &	yes \\
-- $p$-value  &	moderate  &	yes   &	yes   &	no    &	no \\
Delta method    &	lowest    &	yes   &	yes   &	yes   &	yes \\
Joint significance  &	lower &	yes   &	yes   &	yes   &	yes \\
vGvG    &	highest   &	no    &	yes   &	no    &	yes \\
\hline
%\\
\end{tabular}\par
}
%\end{center}
\end{table}
We have introduced tests that improve on the power of traditional tests, but not all of which have rejection regions bounded away from the null hypothesis space or that are monotonic in $\alpha$. If one is only concerned with maximizing power subject to the constraint of uniformly preserving type 1 error, then these tests will be desirable. If one is concerned about the possibility of rejecting with small values of the test statistic, then one may employ a truncated or constrained version of these tests. If one finds the nonmonotocity property inadmissible, then one can use the test based on the $p$-value corresponding to the minimax optimal test. If one is concerned about both, then one can use a truncated version of the $p$-value based test. The minimax optimal test has the advantage of being a closed form exact solution to its corresponding optimization problem as well as preserving type~1 error exactly in the limit. The minimax optimal test also can be generated almost instantaneously for all unit fraction values of $\alpha$, and the generalized test can be generated for other values of $\alpha$, whereas for the Bayes risk optimal test, the optimization step to generate the rejection region must be performed each time a test for a new value of $\alpha$ is desired. %That said, once the latter test is generated once for a value of~$\alpha$, it can be stored and used again without having to ever rerun the optimization. %For non-unit fraction values of $\alpha$, the Bayes risk optimal test will likely be better powered than the extended minimax optimal test, especially for values of $\alpha$ furthest away from a unit fraction. 
It is easy to obtain $p$-values for the minimax optimal test. %, whereas it is not yet clear how computationally feasible it would be to compute a $p$-value corresponding to the Bayes risk optimal test as defined in the Supplementary Materials. %Section \ref{sec:pval}. 
Lastly, the minimax optimal test is a strictly deterministic test, whereas the Bayes risk optimal test contains a few cells in the rejection region where one rejects with non-degenerate probability. If a non-random test is desired, the latter can be converted to a slightly more conservative non-random test. %by changing all non-degenerate rejection probabilities to zero. %The Bayes risk optimal test has the advantage of its rejection region being bounded away from the null hypothesis space, at least for $\alpha=0.05$. %Overall, we believe the strengths of the minimax optimal test outweigh those of the Bayes risk optimal test, and therefore we recommend the former for practical use. However, the comparison is subjective, and we also believe the Bayes risk optimal test is a strong alternative to apply in practice as well.

%There are a number of important future research directions stemming from this work. There are many inferential problems that face similar issues with non-uniform convergence. Examples commonly arise in partial identification bounds, where there are often minima and maxima of estimates involved, which exhibit similar asymptotic behavior. The approaches proposed in this article can likely be adapted to such settings. Under different models, mediated effects can %Not all mediated effect hypotheses take the form of a single product of coefficients, but rather 
%take the form of sums of products of coefficients. This considerably more complex setting will be important to address to allow for less restrictive models when performing inference on mediated effects. %Yet another important direction for future research is to extend to null hypotheses of the products of coefficients being equal to a nonzero scalar, which will allow for the inversion of the hypothesis test to produce confidence intervals with improved finite-sample performance. %We plan to further consider the properties of our $p$-value definition in Section \ref{sec:pval}, and explore its application to other settings where the more traditional $p$-value definition does not apply. 
%Lastly, we plan to further develop the tests for products of more than two coefficients discussed in the Supplementary Materials. %Section \ref{sec:extensions}.

We close with some practical recommendations for implementation of our proposed methodology. We reiterate that which test to use is subjective and comes down to how the investigator thinks about interpretability and weighs the properties summarized in Table \ref{table:summary}. One will often wish to avoid reporting significant findings that are not ``practically relevant'', in which case we would recommend pre-specifying one's criterion of ``practical relevance''. This might be in terms of the proportion mediated, say if the objective is to quantify how much of a significant total effect is transmitted through the candidate mediator, or it could be in terms of the indirect effect estimate itself if this is considered to be the more relevant measure to threshold. Alternatively, if one wishes to take into account uncertainty, then one might prefer to threshold based on the test statistics, which would correspond to using the truncated minimax optimal or constrained Bayes risk optimal test.

\section*{\centering SOFTWARE}
\label{sec:software}
\if1\blind { 
An R package implementing the minimax optimal test is available at \url{https://github.com/achambaz/mediation.test}.
} \fi

\if0\blind { 
An R package implementing the minimax optimal test is available at [author's github url].
} \fi

\if1\blind { 
\section*{\centering ACKNOWLEDGMENTS}
This publication was supported by the National Center for Advancing Translational Sciences, National Institutes of Health, through Grant Number KL2TR001874. The content is solely the responsibility of the authors and does not necessarily represent the official views of the NIH.
} \fi

\section*{\centering DATA AVAILABILITY STATEMENT}
\label{sec:software}
We do not have permission from the NIMH to share the DCTRS data, but interested parties can apply to the NIMH Data Archive for access. Summary statistics used for the NAS data analysis are available in the supplementary materials of \cite{liu2022large}.

\section*{\centering CONFLICT OF INTEREST}
\label{sec:coi}
The authors report there are no competing interests to declare.

\bibliographystyle{apalike}

\bibliography{references}

\newpage

\appendix
\pagebreak
%\widetext
\begin{center}
\textbf{\Large Supplemental Materials: Optimal Tests of the Composite Null Hypothesis
Arising in Mediation Analysis}
\end{center}

\renewcommand{\thetable}{S\arabic{table}}   
\renewcommand{\thefigure}{S\arabic{figure}}
\renewcommand{\thetheorem}{S\arabic{theorem}}
\renewcommand{\theequation}{S\arabic{equation}}
\renewcommand{\theassumption}{S\arabic{assumption}}
\renewcommand{\theremark}{S\arabic{remark}}
\setcounter{table}{0}
\setcounter{figure}{0}
\setcounter{theorem}{0}
\setcounter{equation}{0}
\setcounter{assumption}{0}
\setcounter{remark}{0}

\section{Mediation analysis details}
The composite null hypothesis $H_0$ arises naturally in mediation analysis
under certain modeling assumptions. In order to define the indirect effect of interest, we first introduce notation. First, suppose we observe $n$ i.i.d.~copies of
$(\bm{C}^\top,A,M,Y)^\top$, where $A$ is the exposure of interest, $Y$ is the
outcome of interest, $M$ is a potential mediator that is temporally
intermediate to $A$ and $Y$, and $\bm{C}$ is a vector of baseline covariates that
we will assume throughout to be sufficient to control for various sorts of
confounding needed for the indirect effect to be identified. %The \emph{natural indirect effect} (NIE) is the mediated effect of $A$ on $Y$ through $M$, i.e., the effect along the causal pathway $A\rightarrow M\rightarrow Y$. 

We now introduce counterfactuals in order to define the natural indirect effect. Let $Y(a,m)$ be
the counterfactual outcome that we would have observed (possibly contrary to
fact) had $A$ been set to the level $a$ and $M$ been set to the level
$m$. Similarly, let $M(a)$ be the counterfactual mediator value we would have
observed (possibly contrary to fact) had $A$ been set to $a$. Lastly, define
the nested counterfactual $Y\{a',M(a'')\}$ to be the counterfactual outcome we
would have observed had $A$ been set to $a'$ and $M$ been set to the
counterfactual value it would have taken had $A$ instead been set to $a''$.
The \emph{natural direct effect} (NDE) and \emph{natural indirect effect}
(NIE) on the difference scale are then defined to be
\begin{align*}
 \mathrm{NDE}(a',a'')&:= E\left[Y\{a',M(a'')\}\right]-E\left[Y\{a'',M(a'')\}\right]\\
 \mathrm{NIE}(a',a'')&:= E\left[Y\{a',M(a')\}\right]-E\left[Y\{a',M(a'')\}\right].
\end{align*}
These additively decompose the total effect, \[\mathrm{TE}(a',a''):= E[Y\{a',M(a')\}]-E[Y\{a'',M(a'')\}]=E\{Y(a')\}-E\{Y(a'')\},\] where $Y(a)$ is the counterfactual outcome we would have observed (possibly contrary to fact) had $A$ been set to $a$. Our focus will be on the natural indirect effect.

The natural indirect effect is nonparametrically identified under certain
causal assumptions, viz., consistency, positivity, and a number of no
unobserved confounding assumptions, which we will not review here for the sake
of brevity and because our focus is on statistical inference rather than
identification. These assumptions are well documented in the causal inference
literature (see \citet{vanderweele2015explanation} for an overview). The
identification formula for the natural indirect effect (also known as the
mediation formula) is
\begin{align*}
    \int_{\bf{c}}\int_{m}&E(Y\mid M, A=a', \bm{C})\\
    &\times\left\{f_{M\mid A, \bm{C}}(m\mid A=a', \bm{C})-f_{M\mid A, \bm{C}}(m\mid A=a'', \bm{C})\right\}d\mu(m)f_{\bm{C}}(c)d\mu(\bm{c}),
\end{align*}
 where $\mu(m)$ and $\mu(\bm{c})$ are dominating measures with respect to the distributions of $M$ and $\bm{C}$. 
 
 If the linear models with main
 effect terms given by
\begin{align}
 \label{eq:model1sm}
 E(M\mid A=a, \bm{C}=\bm{c}) &= \beta_0 + \beta_1a + \bm{\beta_2}^\top \bm{c}\\
 \label{eq:model2sm}
 E(Y\mid A=a,M=m,\bm{C}=\bm{c}) &= \theta_0 + \theta_1a + \theta_2m + \bm{\theta_3}^\top \bm{c},
\end{align}
are correctly-specified, then the identification formula for the
natural indirect effect reduces to $\beta_1\theta_2$. The \emph{product
 method} estimator estimates $\beta_1$ and $\theta_2$ by fitting the above
regression models and takes the product of these coefficient estimates. In
fact, following the path analysis literature of \citet{wright1921correlation},
\citet{baron1986moderator} originally defined the indirect effect to be
$\beta_1\theta_2$, where $\bm{C}$ is empty in the above model, rather than in terms
of counterfactuals, and proposed the product method estimator. This is an
extremely popular estimator of indirect
effects. \citet{vanderweele2009conceptual} demonstrated the consistency of the
product method for the natural indirect effect under the above linear models.
Under models \eqref{eq:model1sm} and \eqref{eq:model2sm} and standard regularity
conditions, the two factors of the product method estimator will satisfy the
uniform joint convergence statement in~\eqref{eq:CLT}. In fact,
\eqref{eq:CLT} will hold for a more general class of models, though there are
important limitations to this class. For instance, consider the outcome model
with exposure-mediator interaction replacing model \eqref{eq:model2sm}:
\begin{align}
 \label{eq:model2bissm}
 E(Y\mid A=a,M=m,\bm{C}=\bm{c}) &= \theta_0 + \theta_1a + \theta_2m + \theta_3am +
 \bm{\theta_4}^\top \bm{c}.
\end{align}
Under models \eqref{eq:model1sm} and \eqref{eq:model2bissm}, the natural indirect
effect is identified by
$\mathrm{NIE}(a',a'')=(\theta_2+\theta_3a')\beta_1(a'-a'')$, hence the
estimator $(\hat{\theta}_2+\hat{\theta}_3a')\hat{\beta}_1(a'-a'')$, where
$\hat{\theta}_2$, $\hat{\theta}_3$, and $\hat{\beta}_1$ are estimated by
fitting the linear regression models \eqref{eq:model1sm} and
\eqref{eq:model2bissm}, will also satisfy \eqref{eq:CLT} under standard
regularity conditions, with $\delta_x=\theta_2+\theta_3a'$ and
$\delta_y=\beta_1(a'-a'')$ and with the corresponding plug-in estimators for
$\hat{\delta}_x$ and $\hat{\delta}_y$. This is a key extension, as modern
causal mediation analysis' ability to account for the presence of
exposure-mediator interactions is one of its important advantages over
traditional mediation methods such as path analysis and the
\citet{baron1986moderator} approach. Unfortunately, the functional form of
$E(Y\mid A=a,M=m,\bm{C}=\bm{c})$ and $E(M\mid A=a, \bm{C}=\bm{c})$ does not yield NIE estimators
that factorize to satisfy \eqref{eq:CLT} in general, as the identification
formula can result in sums or integrals of products of coefficients. Handling
models that yield such identification formulas will be an important direction
in which to extend the work done in the present article. Extending the theory
in the present article to more general settings than \eqref{eq:CLT} is also
especially important because mediated effects through multiple mediators can
also be identified by sums of products of coefficients.

\section{Figure demonstrating the conservativeness of the asymptotic approximation to the true distributions of the delta method and joint significance tests}
\begin{figure}[h]
\centering
\begin{tabular}{ccc}
\includegraphics[width=.45\textwidth]{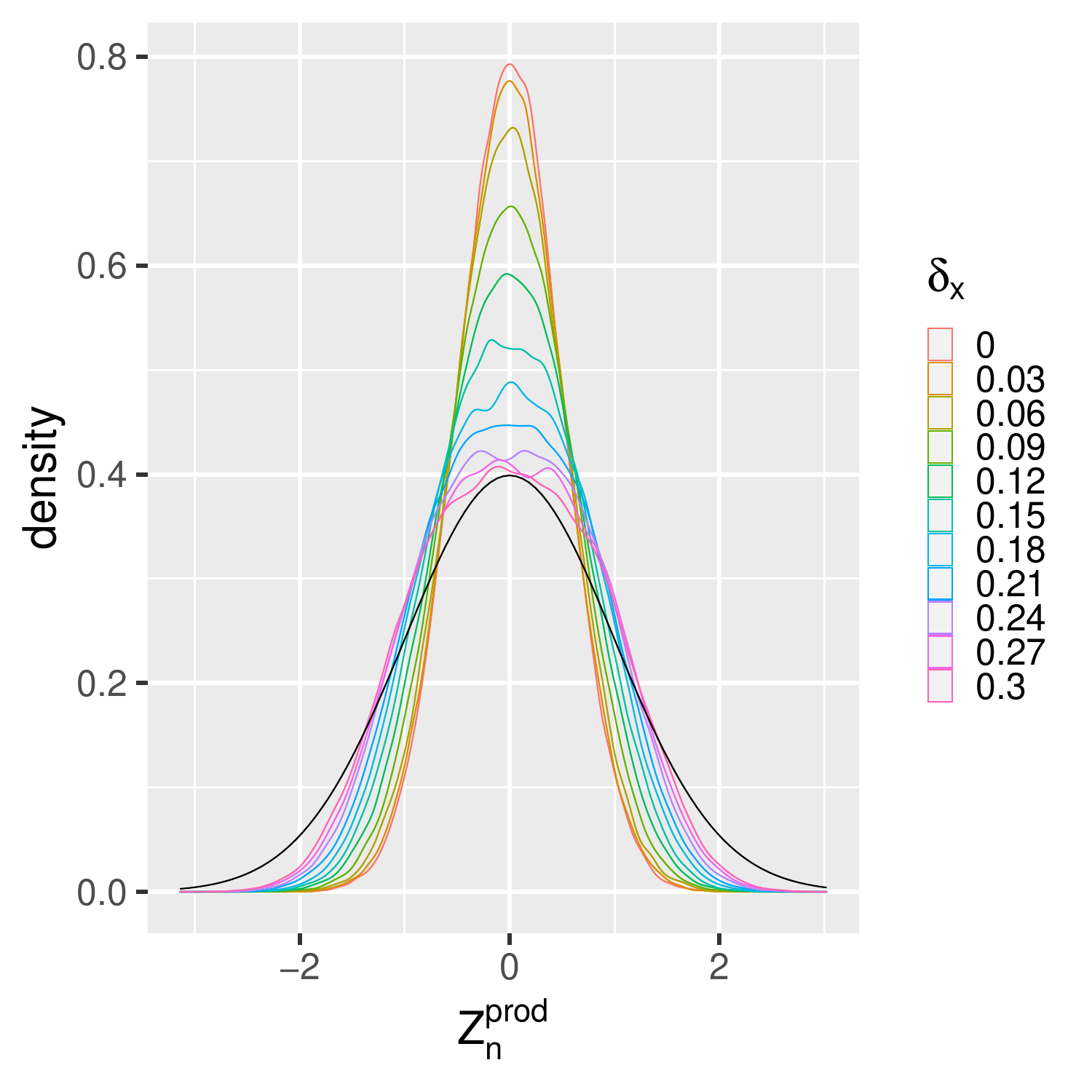}
& 
&
\includegraphics[width=.45\textwidth]{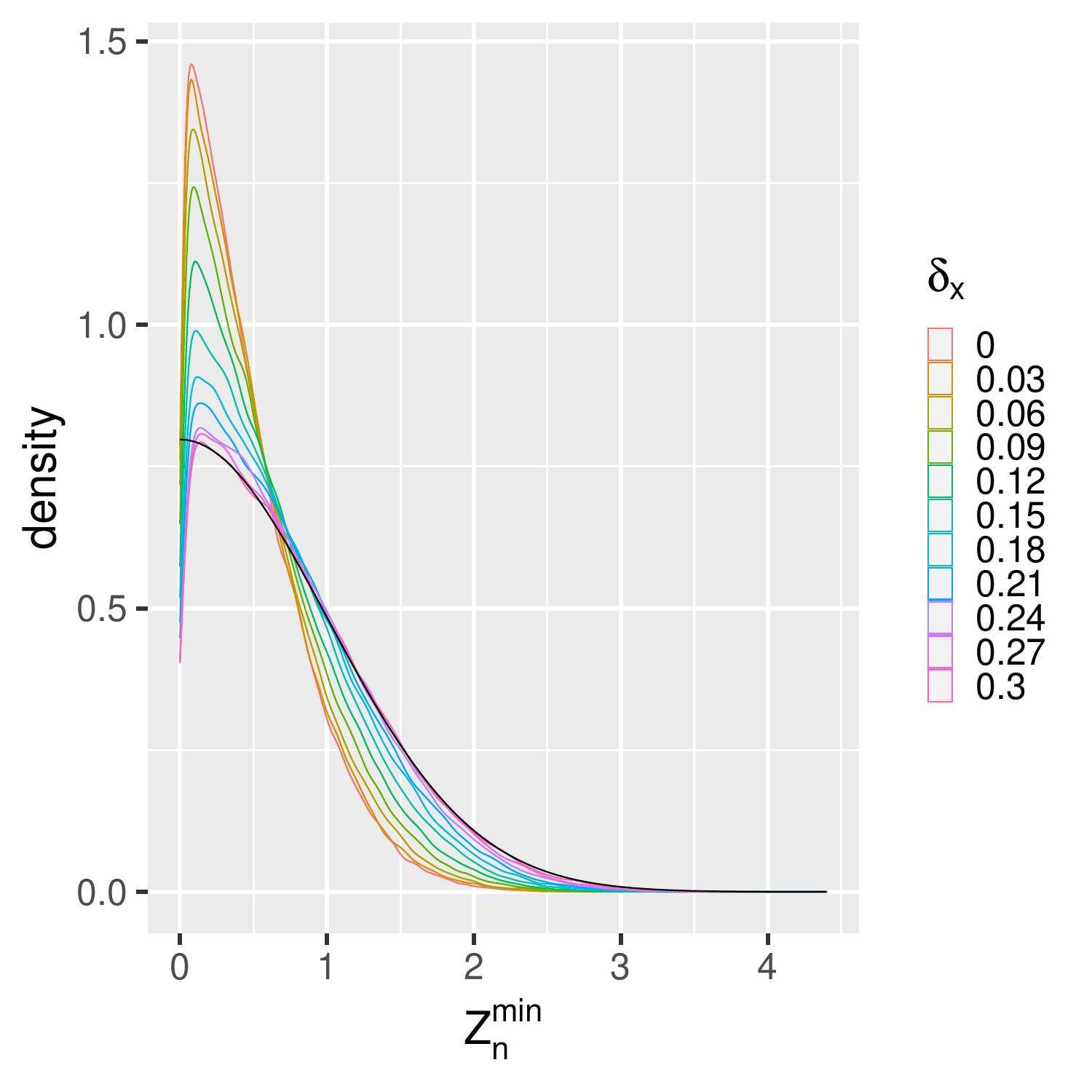}
\\
(a)	&	&	(b)
\end{tabular}
\caption{(a) Density plots estimated by Monte Carlo sampling of $Z_n^{\text{prod}}$ under $\delta_y=0$ and varying $\delta_x$ with $n=100$. The standard normal density curve is displayed in black. (b) Density plots estimated by Monte Carlo sampling of $Z^{\text{min}}_n$ under $\delta_y=0$ and varying $\delta_x$ with $n=100$. The folded standard normal density curve is displayed in black.}
\label{fig:density}
\end{figure}

\section{Plot of the power function surface of the minimax optimal test}

\begin{figure}[H]
\centering
 \includegraphics[width=1\textwidth]{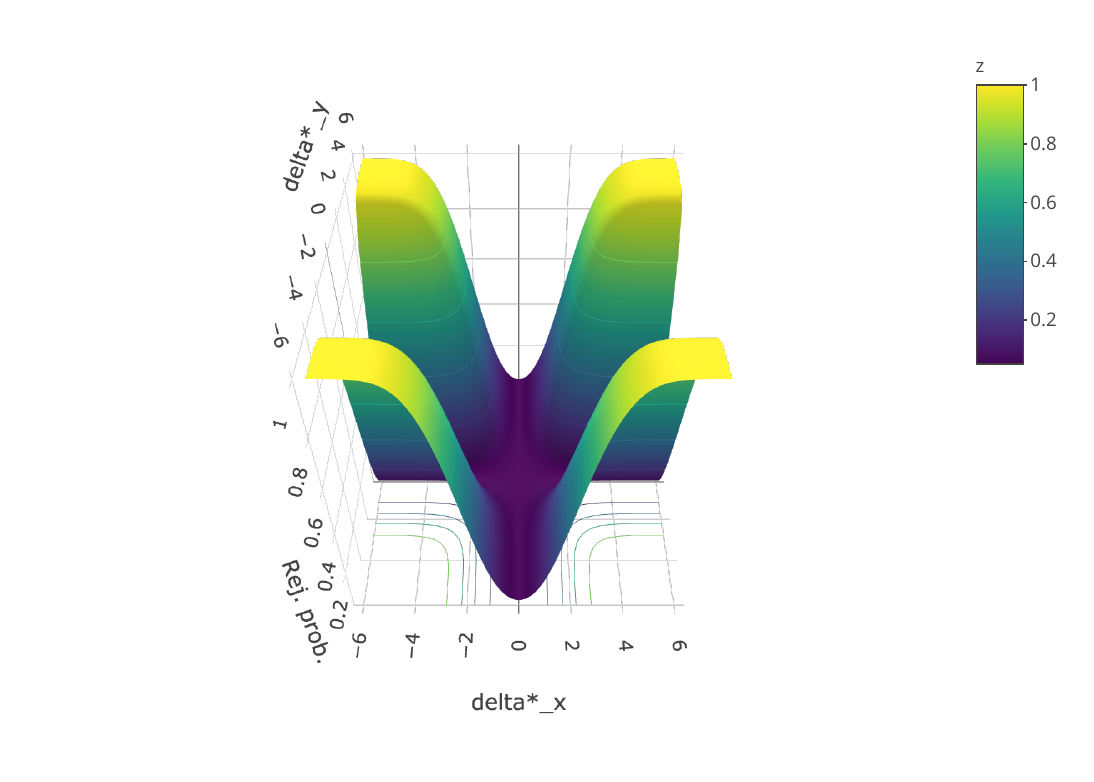} 
 \caption{The power function surface for the minimax optimal test over $(\delta^*_x,\delta^*_y)$. The power function takes the value 0.05 everywhere on the $\delta^*_x$ and $\delta^*_y$ axes, and is strictly greater than 0.05 everywhere else, while going to one as $\delta^*_x$ and $\delta^*_y$ both go to infinity.}
 \label{fig:power}
\end{figure}

\section{Details of the Bayes risk optimization problem formulation}
We denote the full null hypothesis space by $G:=\{(\delta_x^*,\delta_y^*):\delta_x^*\delta_y^*=0\}$ and a corresponding fine grid within the null hypothesis space by $G':= \{(-2b,0)^{\top},(-2b+b/K,0)^{\top},\ldots,(2b-b/K,0)^{\top},(2b,0)^{\top},
    (0,-2b)^{\top},(0,-2b+b/K)^{\top},\ldots,(0,2b-b/K)^{\top},(0,2b)^{\top}\}$
\begin{comment}
\begin{align*}
    G':= \{&(-2b,0)^{\top},(-2b(K-1)/K,0)^{\top},\ldots,(2b(K-1)/K,0)^{\top},(2b,0)^{\top},\\
    &(0,-2b)^{\top},(0,-2b(K-1)/K)^{\top},\ldots,(0,2b(K-1)/K)^{\top},(0,2b)^{\top}\}
\end{align*}

\begin{align*}
    G':= \{&(-2b,0)^{\top},(-2b+b/K,0)^{\top},\ldots,(2b-b/K,0)^{\top},(2b,0)^{\top},\\
    &(0,-2b)^{\top},(0,-2b+b/K)^{\top},\ldots,(0,2b-b/K)^{\top},(0,2b)^{\top}\}
\end{align*}
\end{comment}
for a suitably large integer $K$. This grid extends beyond $B$, because for distributions centered outside of, but near $B$, there will still be a high probability of being in $B$, and we need to control the type~1 error for such distributions. We then consider the $(2K+1)\times (2K+1)$ grid on $B$, %by $\{(x,y):x\in\{-b, -b(K-1)/K, \ldots, b(K-1)/K, b\}, y\in\{-b, -b(K-1)/K, \ldots, b(K-1)/K, b\}\}$, 
$\{-b, -b(K-1)/K, \ldots, b(K-1)/K, b\}^2$, 
and denote the squares induced by this grid %\[R_{k,k'}:=[2b(k-1)/K-b,2bk/K-b)\times[2b(k'-1)/K-b,2bk'/K-b)\] 
$R_{k,k'}:=[b(k-1)/K,bk/K)\times[b(k'-1)/K,bk'/K)$ 
for each $k$ and $k'$ in $-K+1, -K+2, \ldots, K-1, K$. Let $\mathcal{R}$ denote the collection of squares $R_{k,k'}$ for all $k$ and $k'$ in $-K+1, -K+2, \ldots, K-1, K$. Within $B$, we then define a class $\mathcal{M}_\mathcal{R}$ of random rejection regions whose realizations are constant within each square. That is, if $(z^x,z^y)\in R_{k,k'}$, $M(z^x,z^y,u)=M(z^{x'},z^{y'},u)$ for all $(z^{x'},z^{y'})\in R_{k,k'}$ and all $u\in[0,1]$. Denote the rejection probability corresponding to each square $r\in\mathcal{R}$ by $m_r$. These are the unknown arguments defining the rejection region that are to be optimized over.

Under a random test defined by the rejection probabilities $m_r$ on the squares in $\mathcal{R}$, the Bayes risk in \eqref{eq:bayesrisk} 
decomposes as the sum
  \begin{align*}
    &\int             E_{\delta_x^*,\delta_y^*}\left(L\left[I\left\{(Z^x_*,Z^y_*)\in
        R_{\bar{B}}\right\},\delta_x^*,\delta_y^*\right]
      I\left\{(Z^x_*,Z^y_*)\in \bar{B}\right\}\right)d\Lambda(\delta_x^*,\delta_y^*) \\
    +&                                                                     \int
    E_{\delta_x^*,\delta_y^*}\left[L\left\{M(Z^x_*,Z^y_*,U),\delta_x^*,\delta_y^*\right\}
      I\left\{(Z^x_*,Z^y_*)\in B\right\}\right]d\Lambda(\delta_x^*,\delta_y^*),
  \end{align*}
  of which  only the  second term depends  on $(m_{r})_{r\in\mathcal{R}}$,
  being equal to
  \begin{align*}
    &\int\sum_{r\in\mathcal{R}}
          E_{\delta_x^*,\delta_y^*}\left[L\left\{M(Z_*^x,Z_*^y,U),\delta_x^*,\delta_y^*\right\}I\left\{(Z^x_*,Z^y_*)\in r\right\}\right] d\Lambda\left(\delta_x^*,\delta_y^*\right) \\
        =&\int\sum_{r\in\mathcal{R}}
          E_{\delta_x^*,\delta_y^*}\left[L\left\{M(Z_*^x,Z_*^y,U),\delta_x^*,\delta_y^*\right\}\mid(Z^x_*,Z^y_*)\in
          r\right]
          \mathrm{Pr}_{\delta_x^*,\delta_y^*}\left\{(Z^x_*,Z^y_*)\in r\right\} d\Lambda\left(\delta_x^*,\delta_y^*\right) \\
        %=&\int\sum_{r\in\mathcal{R}}
          %\left\{L\left(1,\delta_x^*,\delta_y^*\right)m_{r}                  +
          %L\left(0,\delta_x^*,\delta_y^*\right)(1-m_{r}) \right\}\mathrm{Pr}_{\delta_x^*,\delta_y^*}\left\{(Z^x_*,Z^y_*)\in   r\right\}
          %d\Lambda\left(\delta_x^*,\delta_y^*\right) \\
        %=&\int                           L\left(0;\delta_x^*,\delta_y^*\right)
          %\mathrm{Pr}_{\delta_x^*,\delta_y^*}\left\{(Z^x_*,Z^y_*)\in   B\right\}d\Lambda\left(\delta_x^*,\delta_y^*\right)\\
        %&+ 
        =&\sum_{r\in\mathcal{R}} (1-m_{r}) \int
          L\left(0,\delta_x^*,\delta_y^*\right)
          \mathrm{Pr}_{\delta_x^*,\delta_y^*}\left\{(Z^x_*,Z^y_*)\in   r\right\}
          d\Lambda\left(\delta_x^*,\delta_y^*\right).
  \end{align*}
From the above, the objective function is affine in the unknown variables $(m_r)_{r\in\mathcal{R}}$.

We also define a discretized approximation of the type~1 error constraint \eqref{eq:type:1:constraint}. 
For each $(\delta_x^*,\delta_y^*)\in G'$, 
\begin{align}
    \alpha
    \geq&\;
      \mathrm{Pr}_{\delta_x^*,\delta_y^*}\left\{M(Z_*^x,Z_*^y,U)=1,
      (Z^x_*,Z^y_*)\in \bar{B}\right\}    +    \mathrm{Pr}_{\delta_x^*,\delta_y^*}\left\{M(Z^x_*,Z^y_*,U)=1,
      (Z^x_*,Z^y_*)\in B\right\} \nonumber \\ 
    %=&\;             \mathrm{Pr}_{\delta_x^*,\delta_y^*}\left\{M(Z_*^x,Z_*^y,U)=1,
      %(Z^x_*,Z^y_*)\in \bar{B}\right\}\nonumber\\                                                      &+
      %\sum_{r\in\mathcal{R}}\mathrm{Pr}_{\delta_x^*,\delta_y^*}\left\{M(Z^x_*,Z^y_*,U)=1,
      %(Z^x_*,Z^y_*)\in             r\right\}            \nonumber             \\
    =&\;             \mathrm{Pr}_{\delta_x^*,\delta_y^*}\left\{M(Z_*^x,Z_*^y,U)=1,
      (Z^x_*,Z^y_*)\in \bar{B}\right\} +
      \sum_{r\in\mathcal{R}}m_r\mathrm{Pr}_{\delta_x^*,\delta_y^*}\left\{(Z^x_*,Z^y_*)\in
      r\right\},\label{eq:t1econstraint} 
  \end{align}
which is also affine in $(m_{r})_{r\in\mathcal{R}}$. Lastly, we need the probability constraints $0\leq m_r\leq 1$ for all $r\in\mathcal{R}$. These %probability constraint 
inequalities along with \eqref{eq:t1econstraint} define a linear optimization problem that approximates the Bayes risk optimization problem in $(m_{r})_{r\in\mathcal{R}}$, where all other terms are known. 

The linear optimization problem is high-dimensional, with the dimension of the unknowns %, $\lvert\mathcal{R}\rvert$, 
being $4K^2$, and the number of constraints being $8K^2+8K+1$. %(including both probability constraints). %However, these can be reduced by about a factor of 64 by enforcing symmetry in the rejection region, as the number of unknowns can be cut to about an eighth. In any case, 
However, this is also a highly sparse optimization problem, which allows for efficient computation despite the dimensionality. The type~1 error constraints %in \eqref{eq:t1econstraint} 
consist of only $8K+1$ inequalities, and the matrix characterizing the remaining probability constraints contains only one nonzero element per row.

\section{Minimax optimal test $p$-values}
\label{sec:pval}
A common definition of a $p$-value for a simple null hypothesis is the probability of observing the same or more extreme values of the test statistic under the null hypothesis. %However, this definition will not apply to our setting for a number of reasons. The first is that we have a composite null hypothesis rather than a simple null hypothesis. 
For a composite null hypothesis, a $p$-value may be defined similarly, but for the least-favorable distribution in the null hypothesis space. %However, this definition is ambiguous for a vector-valued test statistic. In fact, it 
This definition is only useful when rejection regions for each $\alpha$ can be defined as level sets of some function of $\alpha$, such that ``more extreme'' than a certain value can be taken to mean the set of values that the test always rejects whenever it rejects that particular value. The tests we have defined in this article admit no such representation. %, as there are some values of the test statistic that will be rejected for one $\alpha$ level, but will not be rejected for another smaller $\alpha$ level. For instance, the minimax optimal test for $\alpha=1/3$ will reject 
For instance, if $(Z^x_n,Z^y_n)=(\Phi^{-1}(4/5),\Phi^{-1}(5/7))$, the minimax optimal test will reject $H_0$ for $H_1$ for $\alpha=1/3$ but not for $\alpha=1/2$. %(because $4/6 \leq t \leq 5/6$ is met for both $t = 4/5$ and $t = 5/7$ on the one hand, whereas  $3/4 \leq t \leq 4/4$ is met only for $t = 4/5$ on the other hand). %Defining $p$-values for the minimax optimal test is further complicated by the fact that it is only defined for levels of $\alpha$ that are unit fractions.

Alternatively, we adopt the more general definition of a $p$-value as a statistic whose law stochastically dominates that of the uniform distribution with support $[0,1]$ for all distributions in the null hypothesis space. We can define a tentative $p$-value %corresponding to the minimax optimal test 
to be $\hat{p}:= \int_0^1 I\{(Z^x_*,Z^y_*)\notin R_{mm}(\alpha)\}d\alpha$, where $R_{mm}(\alpha)$ is the rejection region for a level $\alpha$ extended minimax optimal test. For the minimax optimal test, %this requires the test to be defined for all levels $\alpha$ in the unit interval, and not just unit fractions. To this end, 
we use the extension of the minimax optimal test to non-unit fraction values of $\alpha$ described in Section \ref{sec:non-unit-fraction}. The form of this $p$-value is derived in Section \ref{sec:proofs}. 
%conservatively define the rejection region for any level $\alpha$ to be $R_{\lceil\alpha^{-1}\rceil^{-1}}$, i.e., the rejection region is constant for all $\alpha$ between unit fractions. 
We conjecture that this will result in a valid $p$-value in the sense that it will dominate the uniform distribution over [0,1] uniformly over the null hypothesis space. The plot in Figure \ref{fig:pval} shows the empirical cumulative distribution function of a Monte Carlo sample of this $p$-value for the minimax optimal test under the least-favorable distribution where $(\delta_x^*,\delta_y^*)=(0,0)$. 
\begin{figure}[h]
\centering
 \includegraphics[width=.4\textwidth]{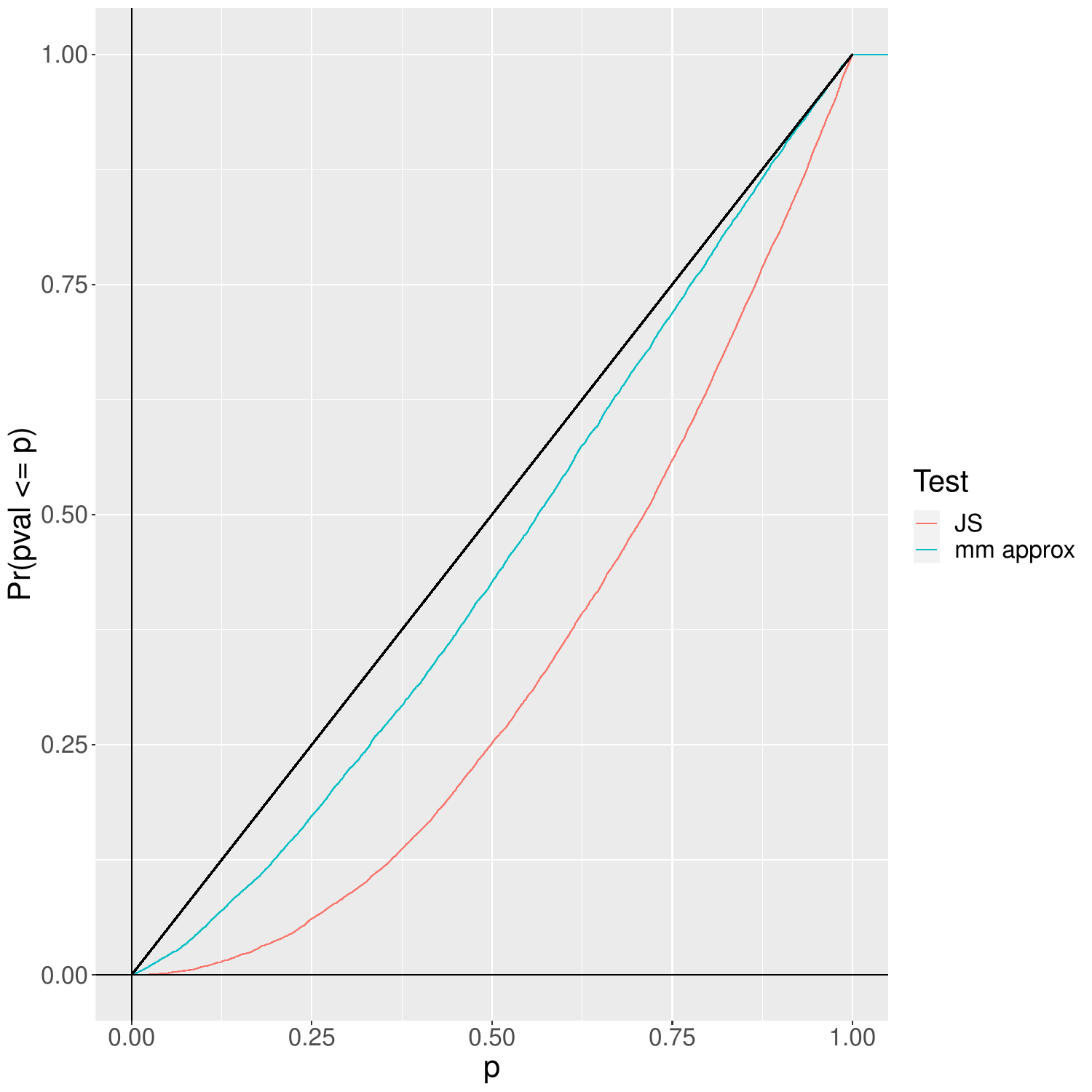} 
 \caption{Empirical cumulative distribution function plots of the $p$-values corresponding to the extended minimax optimal test and joint significance test based on 10,000 Monte Carlo samples from the bivariate normal distribution with $(\delta_x^*,\delta_y^*)=(0,0)$. The cumulative distribution function of the uniform[0,1] distribution is shown in black.}
 \label{fig:pval}
\end{figure}
We can see that this cumulative distribution function does indeed dominate that of the uniform[0,1] distribution, and yet it is dominated by that of the $p$-value corresponding to the joint-significance test. Having defined a $p$-value corresponding to the extended minimax optimal test, one may use these $p$-values to perform the Benjamini--Hochberg correction to many such independent tests to control the false discovery rate. When generating QQ-plots or volcano plots, no calibrations need to be made to the $p$-values. %; they will be valid uniformly over the composite null hypothesis parameter space. %They will also be conservative under and near the intersection null as illustrated in Figure~\ref{fig:pval}; however they will be less conservative than those corresponding to the joint significance or delta method tests. Therefore, the empirical $p$-value distribution will be a mixture of many uniform and super-uniform distributions. The degree of conservativeness of this distribution will depend on the context, but can likely be expected to be conservative in most cases. Such conservativeness would be the only source of bias in the sense of deviation from the uniform distribution. 
A major distinction of our approach to large-scale mediation hypothesis testing from existing approaches is that we make no assumptions about the distribution of the parameters under the null hypothesis and do not attempt to estimate such a distribution. As such, our $p$-values need to be valid under all parameter values under the composite null simultaneously, and we cannot recalibrate to get an exactly uniform distribution.

Since the extended minimax optimal test is conservative for non-unit fraction values of $\alpha$, the $p$-value will also be conservative. Thus, if one wishes to employ the standard minimax optimal test at a particular level of $\alpha$, one would reject based on the test defined by the rejection region $R_{mm}(\alpha)$ rather than comparing the $p$-value to $\alpha$, as these may not agree and the hypothesis test will be more powerful. However, as we discuss in the following section, one might prefer to instead define a test based on this $p$-value. %Similarly, the actual false discovery rate will be more conservative when applying the Benjamini--Hochberg procedure. %Sharpening the test for non-unit fraction $\alpha$ and hence for the corresponding $p$-value is an important topic for future research.

\begin{comment}
In principle, a $p$-value corresponding to the Bayes risk optimal test can be defined similarly. However, this definition requires the test to be defined on a continuum of $\alpha$, or at least a fine discretization of the unit interval. Given the computationally-intensive nature of solving for the rejection region of the Bayes risk optimal test, this may be overly burdensome to implement in practice, especially if close approximations are needed for small $p$-values if correcting for multiple testing.
\end{comment}

\section{DCTRS data analysis details}
\subsection{Table of baseline covariates}

\begin{table}[H]
\begin{center}
{\centering
\caption{Distribution of baseline covariates across studies}
\label{table:s1}
\begin{tabular}{llll}\\
\hline\hline
& \multicolumn{1}{c}{\textbf{\begin{tabular}[c]{@{}c@{}}Study 4155\\ (N=29)\end{tabular}}} & \multicolumn{1}{c}{\textbf{\begin{tabular}[c]{@{}c@{}}Study 8134\\ (N=32)\end{tabular}}} & \multicolumn{1}{c}{\textbf{\begin{tabular}[c]{@{}c@{}}Study 9212\\ (N=67)\end{tabular}}} \\
\hline
\textbf{Age (months)}                         &                                                                                          &                                                                                          &                                                                                          \\
\quad Mean (SD)                                     & \multicolumn{1}{c}{466 (117)}                                                            & \multicolumn{1}{c}{215 (28.8)}                                                           & \multicolumn{1}{c}{426 (106)}                                                            \\
\quad Median {[}Min, Max{]}                         & \multicolumn{1}{c}{444 {[}228, 768{]}}                                                   & \multicolumn{1}{c}{216 {[}168, 276{]}}                                                   & \multicolumn{1}{c}{408 {[}204, 732{]}}                                                   \\
\textbf{Sex}                                  &                                                                                          &                                                                                          &                                                                                          \\
\quad Female                                        & \multicolumn{1}{c}{5 (17.2\%)}                                                           & \multicolumn{1}{c}{11 (34.4\%)}                                                          & \multicolumn{1}{c}{16 (23.9\%)}                                                          \\
\quad Male                                          & \multicolumn{1}{c}{24 (82.8\%)}                                                          & \multicolumn{1}{c}{21 (65.6\%)}                                                          & \multicolumn{1}{c}{51 (76.1\%)}                                                          \\
\textbf{Race}                                 &                                                                                          &                                                                                          &                                                                                          \\
\quad Asian                                         & \multicolumn{1}{c}{3 (10.3\%)}                                                            & \multicolumn{1}{c}{2 (6.3\%)}                                                            & \multicolumn{1}{c}{4 (6.0\%)}                                                            \\
\quad Black or African American                     & \multicolumn{1}{c}{2 (6.9\%)}                                                            & \multicolumn{1}{c}{8 (25.0\%)}                                                           & \multicolumn{1}{c}{12 (17.9\%)}                                                          \\
\quad Unknown or not reported                       & \multicolumn{1}{c}{0 (0\%)}                                                              & \multicolumn{1}{c}{4 (12.5\%)}                                                           & \multicolumn{1}{c}{2 (3.0\%)}                                                            \\
\quad White                                         & \multicolumn{1}{c}{24 (82.8\%)}                                                          & \multicolumn{1}{c}{18 (56.3\%)}                                                          & \multicolumn{1}{c}{49 (73.1\%)}                                                          \\
\textbf{Education (years)}                    &                                                                                          &                                                                                          &                                                                                          \\
\quad Mean (SD)                                     & \multicolumn{1}{c}{12.3 (1.97)}                                                          & \multicolumn{1}{c}{11.2 (1.47)}                                                          & \multicolumn{1}{c}{11.4 (2.24)}                                                          \\
\quad Median {[}Min, Max{]}                         & \multicolumn{1}{c}{12.0 {[}9.00, 17.0{]}}                                                & \multicolumn{1}{c}{11.0 {[}8.00, 16.0{]}}                                                & \multicolumn{1}{c}{11.0 {[}2.00, 17.0{]}}                                                \\
\textbf{WAIS working memory} &                                                                                          &                                                                                          &                                                                                          \\
\textbf{digit span test} &                                                                                          &                                                                                          &                                                                                          \\
\quad Mean (SD)                                     & \multicolumn{1}{c}{12.1 (4.55)}                                                          & \multicolumn{1}{c}{11.7 (3.29)}                                                          & \multicolumn{1}{c}{14.8 (3.95)}                                                          \\
\quad Median {[}Min, Max{]}                         & \multicolumn{1}{c}{12.0 {[}3.00, 23.0{]}}                                                   & \multicolumn{1}{c}{11.0 {[}7.00, 18.0{]}}                                                & \multicolumn{1}{c}{15.0 {[}6.0, 23.0{]}}                                                   \\
\quad Missing                         &
\multicolumn{1}{c}{2 (6.9\%)}      &
\multicolumn{1}{c}{2 (6.3\%)}      &
\multicolumn{1}{c}{0 (0\%)}      \\
\textbf{Rosenberg self-esteem scale}                   &                                                                                          &                                                                                          &                                                                                          \\
\quad Mean (SD)                                     & \multicolumn{1}{c}{31.5 (7.61)}                                                          & \multicolumn{1}{c}{33.4 (6.63)}                                                          & \multicolumn{1}{c}{32.0 (7.65)}                                                          \\
\quad Median {[}Min, Max{]}                         & \multicolumn{1}{c}{30.5 {[}12.0, 50.0{]}}                                                & \multicolumn{1}{c}{34.0 {[}19.0, 48.0{]}}                                                   & \multicolumn{1}{c}{34.0 {[}10.0, 48.0{]}}                                                   \\
\quad Missing                         &
\multicolumn{1}{c}{1 (3.4\%)}      &
\multicolumn{1}{c}{0 (0\%)}      &
\multicolumn{1}{c}{0 (0\%)}    \\
\hline
\\
\end{tabular}}
\end{center}
\end{table}

\subsection{Variable selection, causal assumptions, and statistical models}
Although many additional measures of cognition and functioning at baseline are available in the data, we only included baseline measures of the WAIS working memory digit span test and the Rosenberg self-esteem scale %social behavioral scale and WCST non-perseverative errors 
for three reasons. The first is that since we are only concerned with controlling for confounding of the effect of working memory at the end of treatment on self-esteem at the end of follow-up, their corresponding baseline measures would seem to be the most relevant in terms of capturing baseline common causes from the cognition and functioning/quality-of-life domains. The second reason is that the sample size is not very large---certainly not large enough to include all baseline cognition and functioning measures---and so we favored a more parsimonious model that adjusted for the most relevant baseline measures of cognition and functioning/quality-of-life. The third is that many of these measures are not observed in all three studies. As is always the case with mediation analysis (even in the context of randomized experiments), our analyses are subject to bias due to residual confounding. Given that the purpose of this analysis is to illustrate our methodology, a full sensitivity analysis is beyond the scope of this article, though methods for such sensitivity analyses are available \citep{imai2010general,ding2016sharp}, and will be important to incorporate into our testing procedures in future work. However, we do explore the impact of including additional covariates in $\bm{C}$ in Section \ref{subsec:supp:DCTRS}.

Another identification assumption for the NIE is that no common cause of $M$ and $Y$ is affected by $A$, which, as is always the case in mediation analysis, may be violated in our study. However, this issue is also beyond the scope of this article. To address exposure-induced confounding, one could alternatively estimate partial identification bounds \citep{miles2017partial}, or focus on a different target estimand, such as the path-specific effect not through the exposure-induced confounder \citep{avin2005identifiability,miles2017quantifying,miles2020semiparametric}. Both approaches require observation of all exposure-induced confounders, say $\bm L$. The latter additionally requires the effect of $A$ on $\bm L$ to itself have no residual confounding (which is guaranteed when $A$ is randomized), and no unobserved confounding of the effect of $\bm L$ on $M$ and $Y$. For a more detailed discussion on the assumptions needed to identify this path-specific effect, see \citet{avin2005identifiability} and \citet{ miles2017quantifying, miles2020semiparametric}.

The DCTRS analysis involved pooling data from three different trials. Formally, this is justified if when the study ID indicators are included in $\bm C$, the causal assumptions (consistency, positivity, and no unobserved confounding assumptions) needed for identification of the NIE are satisfied, and if the statistical models for $E(Y\mid A, M, \bm{C})$ and $E(M\mid A, \bm{C})$ are correctly specified. Model \ref{eq:model2bissm}, which only contains main effect terms in study ID indicators, amounts to an assumption that the effect of the exposure on the mediator is homogeneous across studies. Likewise, model \ref{eq:model1sm} amounts to an assumption that the effects of the exposure and the mediator on the outcome are homogeneous across studies. 
\begin{comment}
Therefore, for our analysis, we began with augmented version of the models \ref{eq:model1sm} and \ref{eq:model2bissm}, with study ID indicators included in $\bm C$. In particular, we tested for interactions between study ID dummy variables and $A$ and between study ID dummy variables and $M$ in the outcome model, and interactions between study ID dummy variables and $A$ in the mediator model. The interaction terms in the mediator model were significant, while those in the outcome model were not. Thus, in the analysis, we used model \ref{eq:model2bissm} and an augmented version of model \ref{eq:model1sm} including the study ID--exposure interaction terms:
\begin{align*}
    E(M\mid A=a, \bm{C}=\bm{c}) &= \beta_0 + \beta_1a + \beta_2 I(\text{study}=8134)a + \beta_3I(\text{study}=9212)a + \bm{\beta_4}^\top \bm{c}.
\end{align*}
Substituting these models into the mediation formula yields \[\mathrm{NIE} = (\theta_2 + \theta_3)\{\beta_1+\beta_2\text{Pr}(\text{study}=8134)+\beta_3\text{Pr}(\text{study}=9212)\}.\]
Thus, we have $\delta_x = \theta_2 + \theta_3$ and $\delta_y = \beta_1+\beta_2\text{Pr}(\text{study}=8134)+\beta_3\text{Pr}(\text{study}=9212)$, and by treating the study sample sizes (and thereby their respective probabilities) as fixed, we simply substitute regression coefficient estimates to attain the corresponding estimates $\hat{\delta}_x$ and $\hat{\delta}_y$. This demonstrates that the framework of the problem we consider can accommodate exposure-covariate interaction in the mediator model.
\end{comment}

\subsection{Plots of the bivariate test statistic from the DCTRS data analysis on the rejection regions corresponding to various tests}

\begin{figure}[H]
\centering
 \includegraphics[width=.95\textwidth]{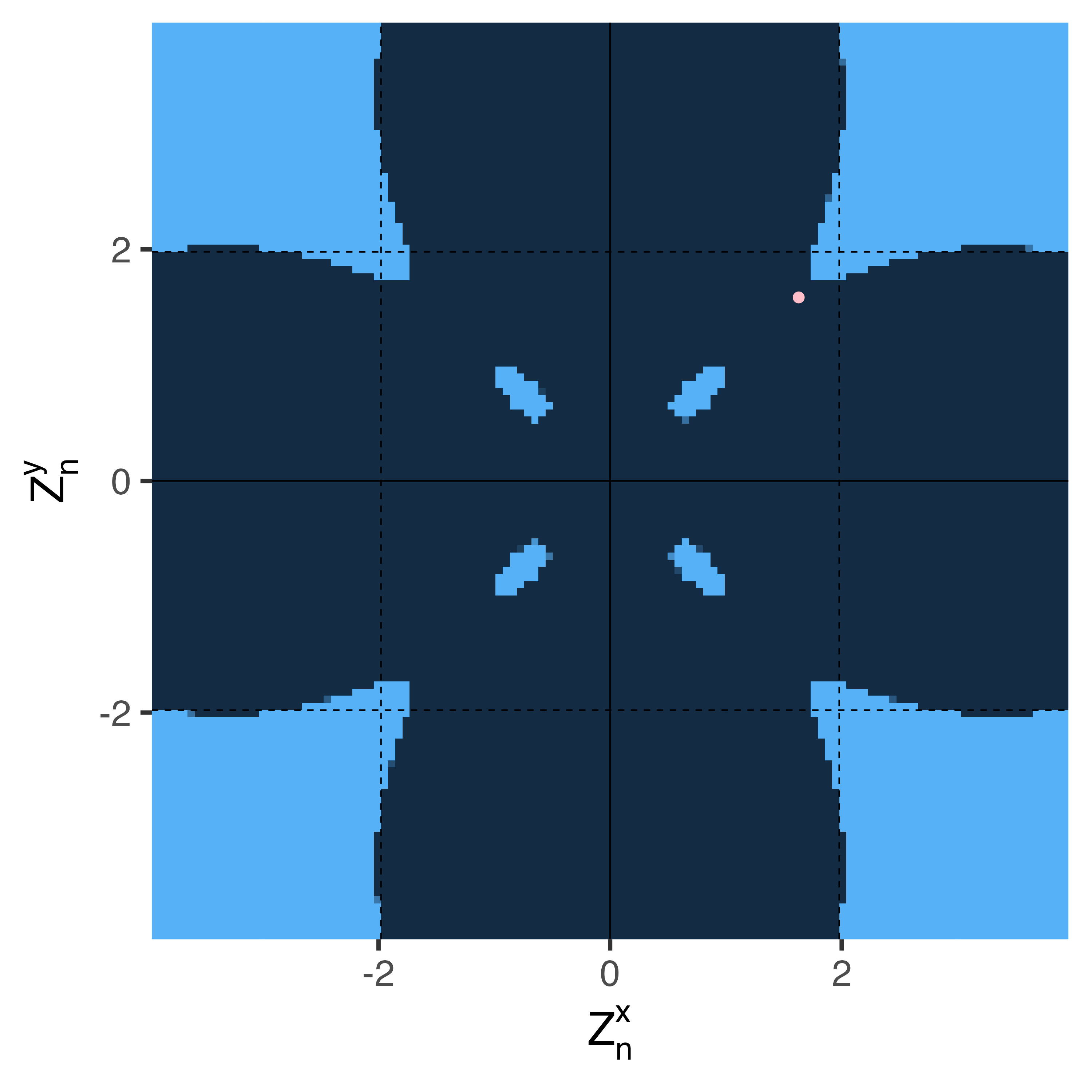} 
 \caption{Plot of the test statistic from the DCTRS data analysis on the rejection region of the Bayes risk optimal test with 0-1 loss.}
 \label{fig:scatter_dctrs_bro01}
\end{figure}

\begin{figure}[H]
\centering
 \includegraphics[width=.95\textwidth]{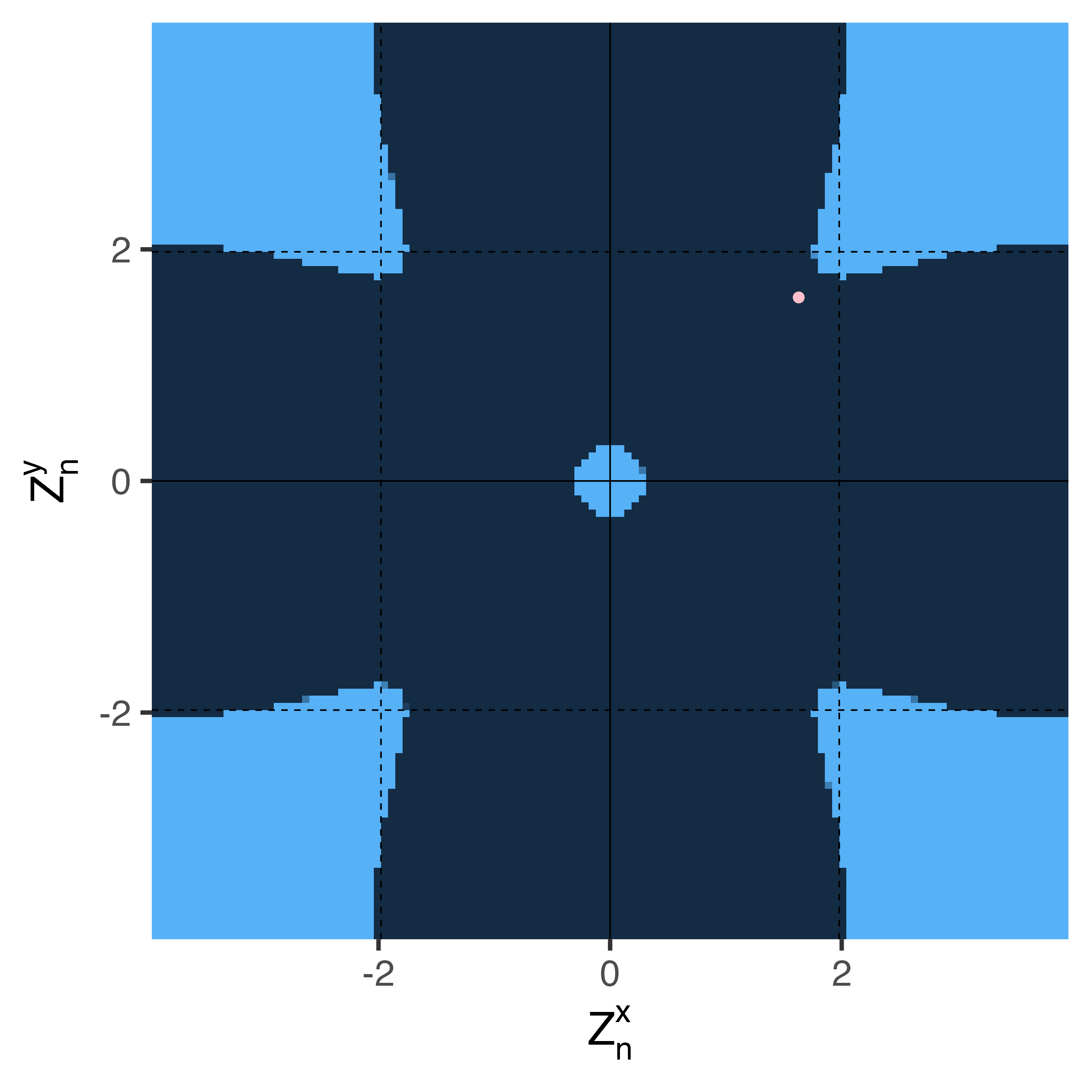} 
 \caption{Plot of the test statistic from the DCTRS data analysis on the rejection region of the Bayes risk optimal test with quadratic loss.}
 \label{fig:scatter_dctrs_bro_quad_}
\end{figure}

\begin{figure}[H]
\centering
 \includegraphics[width=.95\textwidth]{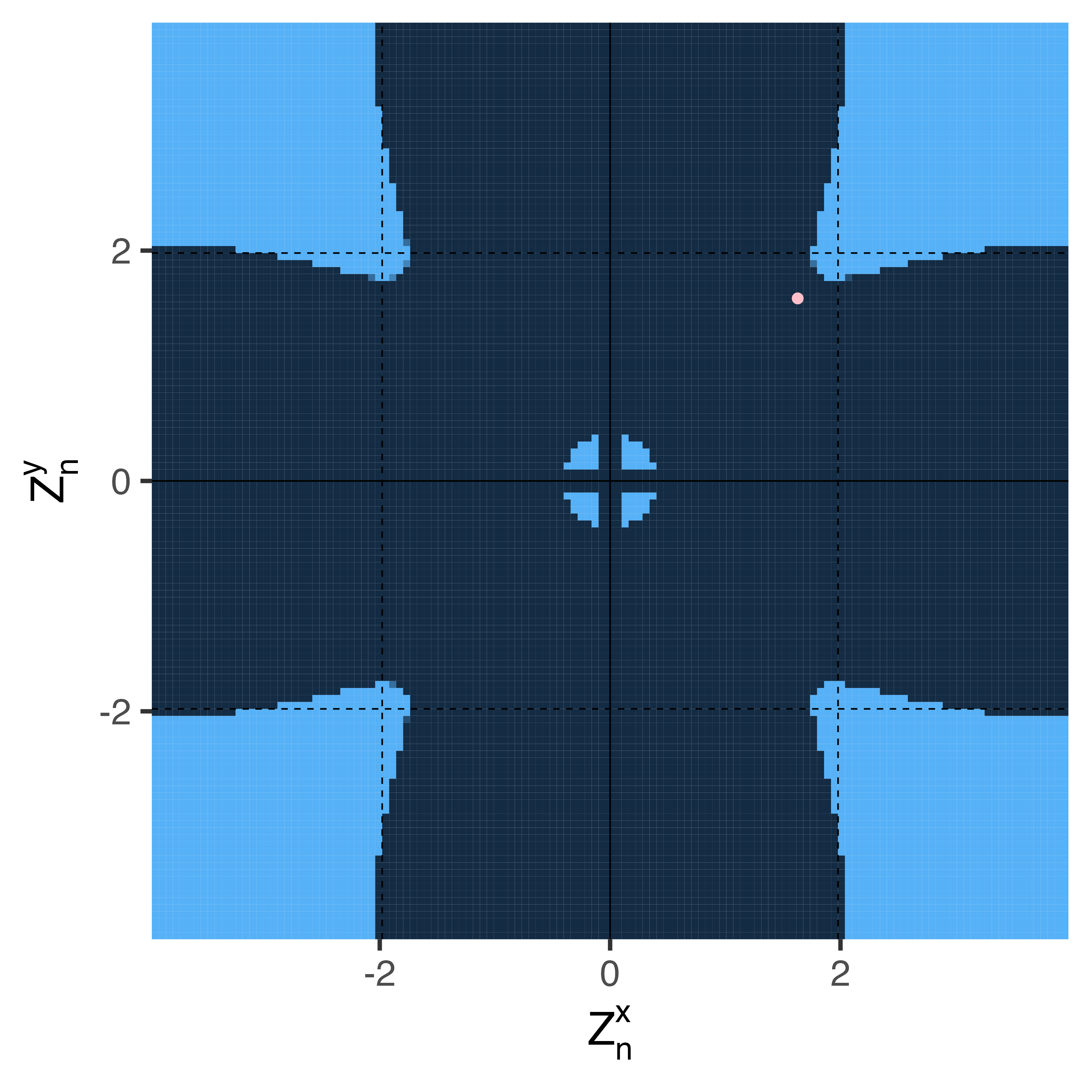} 
 \caption{Plot of the test statistic from the DCTRS data analysis on the rejection region of the constrained Bayes risk optimal test with quadratic loss and $d=0.1$.}
 \label{fig:scatter_dctrs_bro_quad_constr}
\end{figure}

\begin{comment}
\begin{figure}[H]
\centering
 \includegraphics[width=.95\textwidth]{scatter_dctrs_js_wcst_nonpersev.png} 
 \caption{Plot of the test statistic from the DCTRS data analysis on the rejection region of the joint significance test.}
 \label{fig:scatter_dctrs_js}
\end{figure}
\end{comment}

\begin{figure}[H]
\centering
 \includegraphics[width=.95\textwidth]{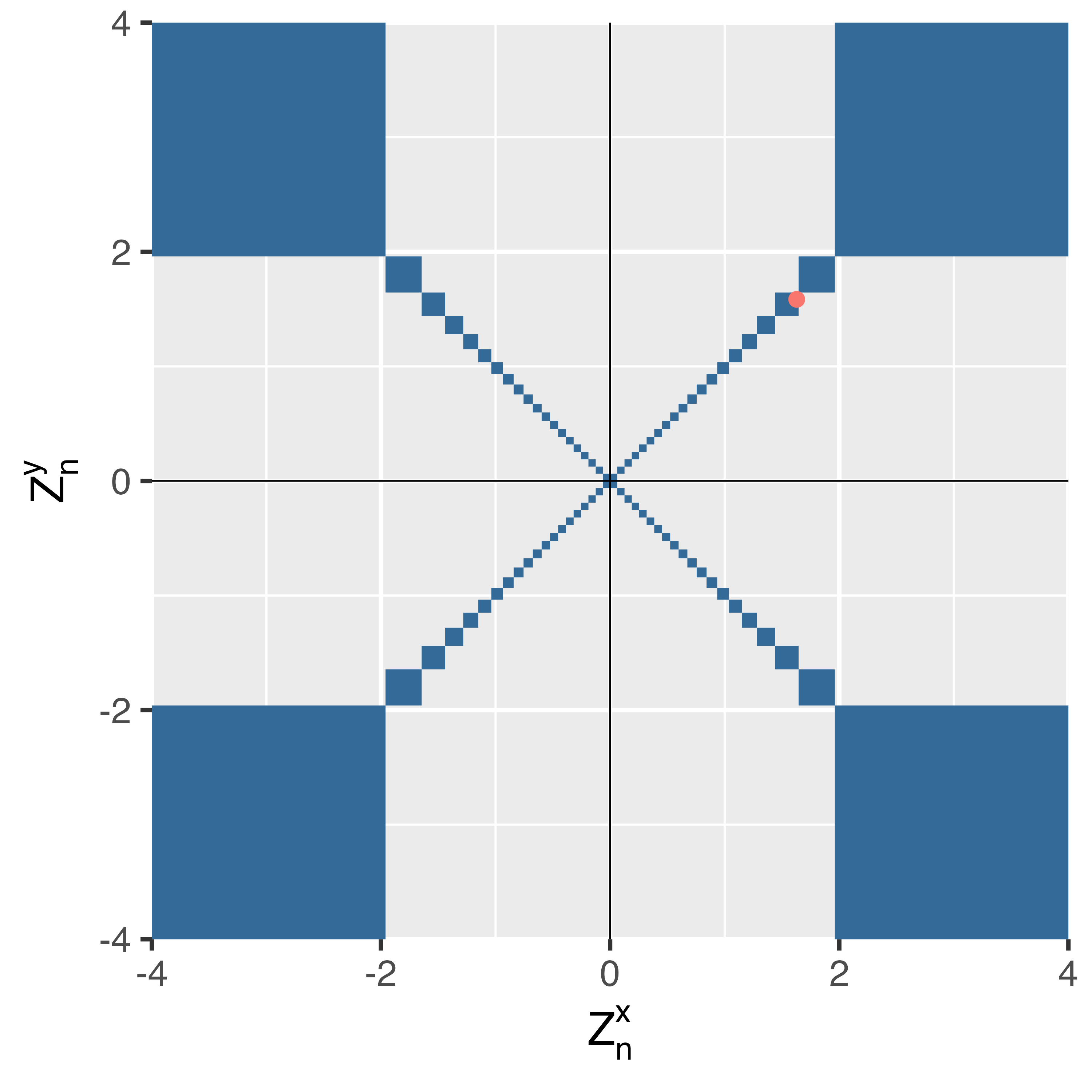} 
 \caption{Plot of the test statistic from the DCTRS data analysis on the rejection region of the minimax optimal test.}
 \label{fig:scatter_dctrs_mmo}
\end{figure}

\begin{figure}[H]
\centering
 \includegraphics[width=.95\textwidth]{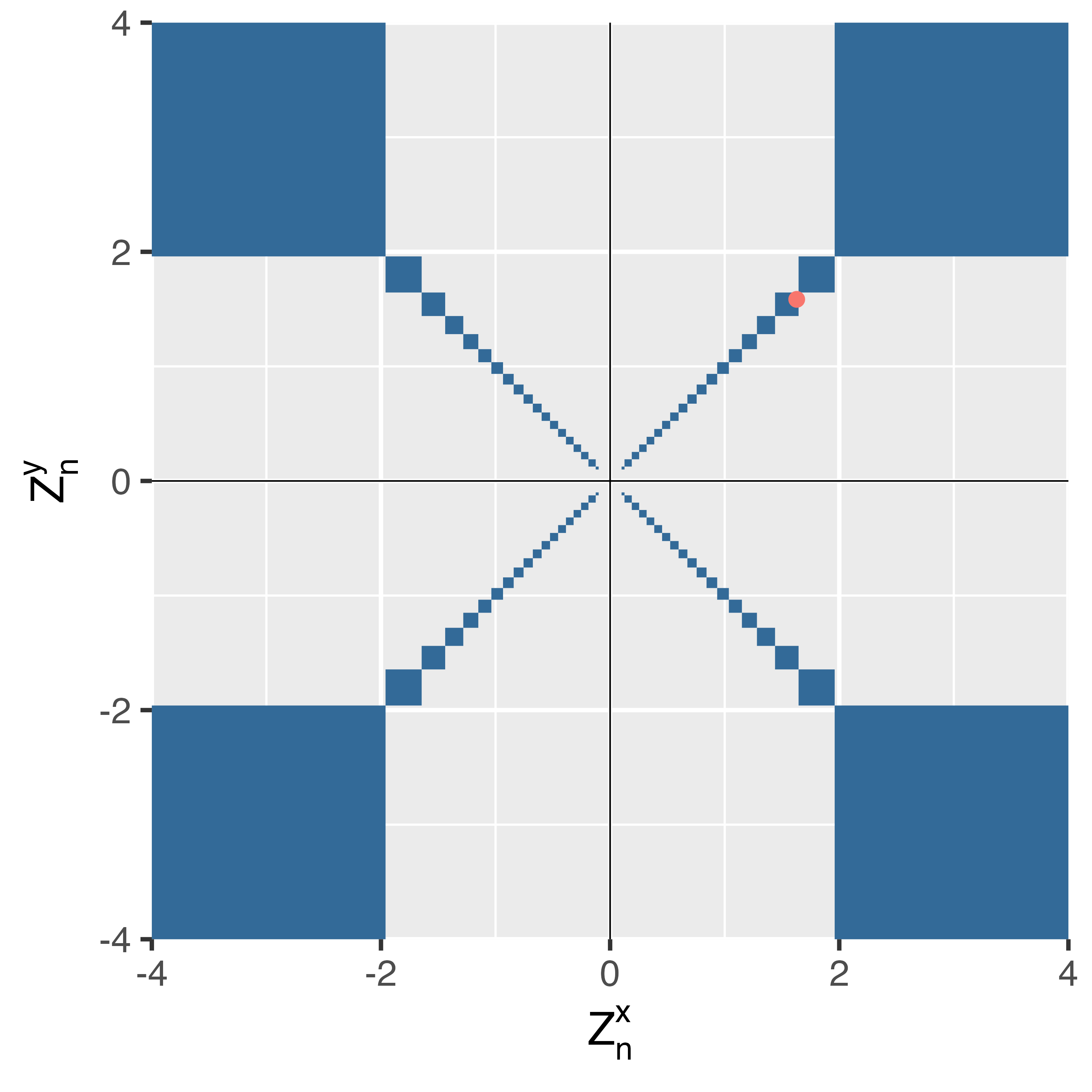} 
 \caption{Plot of the test statistic from the DCTRS data analysis on the rejection region of the truncated minimax optimal test with $d=0.1$.}
 \label{fig:scatter_dctrs_mmo_trunc}
\end{figure}

\subsection{Analysis results by study}
\begin{table}[H]
\begin{center}
{\centering
\caption{Hypothesis test results for the NIE in Study 4155}
\label{table:dctrs4155}
\begin{tabular}{lc r@{.}l}
\\
\hline\hline
Hypothesis test & Reject &\multicolumn{2}{c}{$p$-value}\\
\hline
Bayes risk optimal, 0-1 loss & No & \multicolumn{2}{c}{Undefined} \\
Bayes risk optimal, quadratic loss & No & \multicolumn{2}{c}{Undefined} \\
Bayes risk optimal, quadratic loss, constrained & No & \multicolumn{2}{c}{Undefined} \\
Minimax optimal & No & 0&164 \\
Minimax optimal, truncated & No & 0&164 \\
Minimax optimal, $p$-value & No & 0&164 \\
Delta method & No & 0&223 \\
Joint significance & No & 0&164 \\
van Garderen and van Giersbergen    &   No  &   \multicolumn{2}{c}{Undefined} \\
\hline
\\
\end{tabular}}
\end{center}
\end{table}
\begin{comment}
\begin{table}[H]
\begin{center}
{\centering
\caption{Hypothesis test results for the NIE in Study 4155}
\label{table:dctrs4155}
\begin{tabular}{lc r@{.}l}
\\
\hline\hline
Hypothesis test & Reject &\multicolumn{2}{c}{$p$-value}\\
\hline
Bayes risk optimal, 0-1 loss & No & \multicolumn{2}{c}{Undefined} \\
Bayes risk optimal, quadratic loss & No & \multicolumn{2}{c}{Undefined} \\
Bayes risk optimal, quadratic loss, constrained & No & \multicolumn{2}{c}{Undefined} \\
Delta method & No & 0&429 \\
Joint significance & No & 0&386 \\
Minimax optimal & No & 0&386 \\
Minimax optimal, truncated & No & 0&386 \\
Minimax optimal, $p$-value & No & 0&386 \\
van Garderen and van Giersbergen    &   No  &   \multicolumn{2}{c}{Undefined} \\
\hline
\\
\end{tabular}}
\end{center}
\end{table}
\end{comment}

\begin{table}[H]
\begin{center}
{\centering
\caption{Hypothesis test results for the NIE in Study 8134}
\label{table:dctrs8134}
\begin{tabular}{lc r@{.}l}
\\
\hline\hline
Hypothesis test & Reject &\multicolumn{2}{c}{$p$-value}\\
\hline
Bayes risk optimal, 0-1 loss & No & \multicolumn{2}{c}{Undefined} \\
Bayes risk optimal, quadratic loss & No & \multicolumn{2}{c}{Undefined} \\
Bayes risk optimal, quadratic loss, constrained & Yes & \multicolumn{2}{c}{Undefined} \\
Minimax optimal & No & 0&303 \\
Minimax optimal, truncated & No & 0&303 \\
Minimax optimal, $p$-value & No & 0&303 \\
Delta method & No & 0&880 \\
Joint significance & No & 0&859 \\
van Garderen and van Giersbergen    &   No  &   \multicolumn{2}{c}{Undefined} \\
\hline
\\
\end{tabular}}
\end{center}
\end{table}

\begin{comment}
    \begin{table}[H]
\begin{center}
{\centering
\caption{Hypothesis test results for the NIE in Study 8134}
\label{table:dctrs8134}
\begin{tabular}{lc r@{.}l}
\\
\hline\hline
Hypothesis test & Reject &\multicolumn{2}{c}{$p$-value}\\
\hline
Bayes risk optimal, 0-1 loss & No & \multicolumn{2}{c}{Undefined} \\
Bayes risk optimal, quadratic loss & No & \multicolumn{2}{c}{Undefined} \\
Bayes risk optimal, quadratic loss, constrained & No & \multicolumn{2}{c}{Undefined} \\
Delta method & No & 0&822 \\
Joint significance & No & 0&815 \\
Minimax optimal & No & 0&746 \\
Minimax optimal, truncated & No & 0&746 \\
Minimax optimal, $p$-value & No & 0&746 \\
van Garderen and van Giersbergen    &   No  &   \multicolumn{2}{c}{Undefined} \\
\hline
\\
\end{tabular}}
\end{center}
\end{table}
\end{comment}

\begin{table}[H]
\begin{center}
{\centering
\caption{Hypothesis test results for the NIE in Study 9212}
\label{table:dctrs9212}
\begin{tabular}{lc r@{.}l}
\\
\hline\hline
Hypothesis test & Reject &\multicolumn{2}{c}{$p$-value}\\
\hline
Bayes risk optimal, 0-1 loss & No & \multicolumn{2}{c}{Undefined} \\
Bayes risk optimal, quadratic loss & No & \multicolumn{2}{c}{Undefined} \\
Bayes risk optimal, quadratic loss, constrained & No & \multicolumn{2}{c}{Undefined} \\
Minimax optimal & No & 0&343 \\
Minimax optimal, truncated & No & 0&343 \\
Minimax optimal, $p$-value & No & 0&343 \\
Delta method & No & 0&794 \\
Joint significance & No & 0&749 \\
van Garderen and van Giersbergen    &   No  &   \multicolumn{2}{c}{Undefined} \\
\hline
\\
\end{tabular}}
\end{center}
\end{table}

\begin{comment}
    \begin{table}[H]
\begin{center}
{\centering
\caption{Hypothesis test results for the NIE in Study 9212}
\label{table:dctrs9212}
\begin{tabular}{lc r@{.}l}
\\
\hline\hline
Hypothesis test & Reject &\multicolumn{2}{c}{$p$-value}\\
\hline
Bayes risk optimal, 0-1 loss & No & \multicolumn{2}{c}{Undefined} \\
Bayes risk optimal, quadratic loss & No & \multicolumn{2}{c}{Undefined} \\
Bayes risk optimal, quadratic loss, constrained & No & \multicolumn{2}{c}{Undefined} \\
Delta method & No & 0&291 \\
Joint significance & No & 0&192 \\
Minimax optimal & No & 0&193 \\
Minimax optimal, truncated & No & 0&192 \\
Minimax optimal, $p$-value & No & 0&192 \\
van Garderen and van Giersbergen    &   No  &   \multicolumn{2}{c}{Undefined} \\
\hline
\\
\end{tabular}}
\end{center}
\end{table}
\end{comment}

\subsection{DCTRS analysis with additional baseline covariates}
\label{subsec:supp:DCTRS}
To assess the sensitivity to the choice of adjustment variables, we also conducted the analysis with additional covariates. In particular, we fit the linear models for the outcome and mediator including all variables available at baseline (after a round of imputation using random forest to fill in missing baseline covariates) using LASSO. We then augmented the adjustment set used in the initial analysis with the union of covariates selected in both LASSO fits, as motivated by the post-double-selection method of \cite{belloni2014inference} for estimating average treatment effects. In our case, we only need to control for mediator-outcome confounding since the treatment is randomized, hence the determination of which variables to add to the adjustment set by models for the mediator and the outcome. %When using ``lambda.1se'' to tune the LASSO in glmnet (i.e., the most regularized model such that error is within one standard error of the minimum cross-validated MSE), no variables were selected that were not already included in our original adjustment set. 
When using ``lambda.min'' to tune the LASSO in glmnet (i.e., minimizing cross-validated error MSE), the selected variables consisted of the following baseline measurements: study ID, marital status, Rosenberg self-esteem scale (self-confirmation factor and total scores), Social Behaviour Scale (SBS) items: (attention-seeking behavior, coherence of conversation, concentration, depression, hostility/friendliness, personal appearance and hygiene, laughing and talking to self, socially unacceptable manners or habits, other behaviors that impede progress, panic attacks and phobias, slowness, verbal fluency test using the letters F, A, and S (FAS): total number of correct responses, trailmaking test part A (TMTA) (paper \& pencil): number of errors, trailmaking test part B (TMTB) (paper \& pencil): number of errors, Wechsler Adult Intelligence Scale (WAIS) (working memory digit span test raw score, picture completion raw and scaled scores, vocabulary raw score), Wisconsin Card Sorting Test (WCST): categories achieved, Positive and Negative Syndrome Scale (PANSS) items: P1 Delusions, P4 Excitement, P6 Suspiciousness, N5 Difficulty in Abstract Thinking, G2 Anxiety, G3 Guilt Feelings, G9 Unusual Thought Content, G12 Lack of Judgment and Insight, G16 Active Social Avoidance.

Augmenting our adjustment set with these variables yielded the results in Table \ref{table:augmented}. Evidently, the results are indeed sensitive to which variables we adjust for. Rejection by more of the tests in this case could be a result of instability due to the number of covariates being adjusted for being large relative to the sample size. On the other hand, the fact that some of the tests failed to reject in the previous analysis could be a result of residual confounding bias. We cannot be sure of which is the case. However, it is certainly true that the validity of our proposed tests hinge on the validity of the identification assumptions.
\begin{table}
\begin{center}
{\centering
\caption{Hypothesis test results for the NIE when using the augmented set of adjustment variables}
\label{table:augmented}
\begin{tabular}{lc r@{.}l}
\\
\hline\hline
Hypothesis test & Reject &\multicolumn{2}{c}{$p$-value}\\
\hline
Bayes risk optimal, 0-1 loss & Yes & \multicolumn{2}{c}{Undefined} \\
Bayes risk optimal, quadratic loss & Yes & \multicolumn{2}{c}{Undefined} \\
Bayes risk optimal, quadratic loss, constrained & Yes & \multicolumn{2}{c}{Undefined} \\
Minimax optimal & Yes & 0&034 \\
Minimax optimal, truncated & Yes & 0&034 \\
Minimax optimal, $p$-value & Yes & 0&034 \\
Delta method & No & 0&127 \\
Joint significance & Yes & 0&040 \\
van Garderen and van Giersbergen    &   Yes  &   \multicolumn{2}{c}{Undefined} \\
\hline
\\
\end{tabular}}
\end{center}
\end{table}

To demonstrate that adjusting for more covariates does not necessarily make tests more likely to agree, we also present results from an unadjusted analysis, i.e., where we do not adjust for any baseline covariates. Results from the unadjusted analysis in Table \ref{table:unadjusted} show that all tests agree, whereas adjusting for the covariate set used for the analysis in the main article yields test results that disagree. 
\begin{table}
\begin{center}
{\centering
\caption{Hypothesis test results for the NIE when not adjusting for any baseline covariates}
\label{table:unadjusted}
\begin{tabular}{lc r@{.}l}
\\
\hline\hline
Hypothesis test & Reject &\multicolumn{2}{c}{$p$-value}\\
\hline
Bayes risk optimal, 0-1 loss & No & \multicolumn{2}{c}{Undefined} \\
Bayes risk optimal, quadratic loss & No & \multicolumn{2}{c}{Undefined} \\
Bayes risk optimal, quadratic loss, constrained & No & \multicolumn{2}{c}{Undefined} \\
Minimax optimal & No & 0&498 \\
Minimax optimal, truncated & No & 0&498 \\
Minimax optimal, $p$-value & No & 0&498 \\
Delta method & No & 0&523 \\
Joint significance & No & 0&498 \\
van Garderen and van Giersbergen    &   No  &   \multicolumn{2}{c}{Undefined} \\
\hline
\\
\end{tabular}}
\end{center}
\end{table}
Thus, we see empirically across Tables \ref{table:dctrs}, \ref{table:augmented}, and \ref{table:unadjusted} that adjusting for more covariates does not generally mean that tests are more or less likely to agree.

\begin{comment}
    \begin{table}
\begin{center}
{\centering
\caption{Hypothesis test results for the NIE when using the augmented set of adjustment variables}
\label{table:augmented}
\begin{tabular}{lc r@{.}l}
\\
\hline\hline
Hypothesis test & Reject &\multicolumn{2}{c}{$p$-value}\\
\hline
Bayes risk optimal, 0-1 loss & No & \multicolumn{2}{c}{Undefined} \\
Bayes risk optimal, quadratic loss & No & \multicolumn{2}{c}{Undefined} \\
Bayes risk optimal, quadratic loss, constrained & No & \multicolumn{2}{c}{Undefined} \\
Delta method & No & 0&739 \\
Joint significance & No & 0&732 \\
Minimax optimal & No & 0&732 \\
Minimax optimal, truncated & No & 0&732 \\
Minimax optimal, $p$-value & No & 0&732 \\
van Garderen and van Giersbergen    &   No  &   \multicolumn{2}{c}{Undefined} \\
\hline
\\
\end{tabular}}
\end{center}
\end{table}
\end{comment}

%sbs\_hostilityfriendliness, sbs\_appearancehygiene, sbs\_mannershabits, sbs\_panicphobias, sbs\_total, lns\_raw, pos\_p2, gps\_g7, gps\_g10, tmtb\_errors, wais\_digsym\_scaled, study\_id, sbs\_inapp\_conversation, sbs\_posturing, sbs\_slowness, wcst\_percent\_conceptual, wcst\_categories, wcst\_nonpersev\_errors, gps\_g12, gps\_g13, tmtb\_comp\_raw.

\section{Tests of products of more than two coefficients}
\label{sec:extensions}
We now consider the more general hypothesis testing setting of $H_0^{\ell}:``\prod_{j=1}^{\ell}\delta_j=0"$ against $H_1^\ell: ``\prod_{j=1}^\ell \delta_j \neq 0"$
 when there exists an asymptotically normal estimator $\bm{\hat{\delta}}:=(\hat{\delta}_1,\ldots,\hat{\delta}_{\ell})^{\top}$ of $\bm{\delta}:=(\delta_1, \ldots, \delta_{\ell})^{\top}$ such that the convergence $\sqrt{n}\bm{\Sigma_n}^{-1/2}\left(\bm{\hat{\delta}} - \bm{\delta}\right)\rightsquigarrow\mathcal{N}\left(\bm{0_{\ell}}, \bm{I_{\ell}}\right)$ is uniform in $\bm{\delta}\in\mathbb{R}^{\ell}$, where $\bm{\Sigma_n}$ is again a consistent estimator of the asymptotic covariance matrix, $\bm{0_{\ell}}$ is the $\ell$-vector of zeros, and $\bm{I_{\ell}}$ is the $\ell\times\ell$ identity matrix. This setting arises in linear structural equation modeling when one wishes to test the effect of an exposure along a chain of intermediate variables. Unfortunately, such effects are seldom identifiable in the causal mediation framework; however, this setting can arise in other applications as well. For instance, given multiple candidate instrumental variables, \citet{kang2020two} characterize the null hypothesis of no treatment effect and one valid instrument that is uncorrelated with the other candidates to be the product of multiple coefficients equaling zero.

The over-conservativeness and underpoweredness problems faced by the traditional tests of products of coefficients are in fact exacerbated in higher dimensions. For instance, the generalization of the joint significance test to $\ell$ rejects when the $\ell$ Wald tests corresponding to %$H_{1,0}, H_{2,0}, \ldots, H_{\ell,0}$ 
``$\delta_j=0$'' against ``$\delta_j\neq 0$'' reject %, where $H_{i,0}: \delta_i=0$ for all $i=1,\ldots,\ell$. 
the nulls for their corresponding alternatives for all $j=1,\ldots,\ell$. 
Thus, under $\bm{\delta}=\bm{0_{\ell}}$, the rejection probability of the joint significance test will be $\alpha^{\ell}$. For $\alpha=0.05$ and $\ell=3$, for example, this will be 0.000125.

The minimax optimal test lacks an obvious unique natural extension to the setting with more than two coefficients. Nevertheless, for unit fraction $\alpha$, we can prove the existence of deterministic similar tests that are symmetric with respect to negations. For $\ell=3$, these tests can be constructed using Latin squares of order $\alpha^{-1}$. If the Latin square is totally symmetric, then the test will also be symmetric with respect to permutations. We only provide the result for $\ell=3$; however, we conjecture that these can be readily generalized to higher dimensions using Latin hypercubes, provided they exist for the choice of $\ell$ and $\alpha$.

Given a Latin square of order $\alpha^{-1}$, $\bm{A}$, let $1,\ldots,\alpha^{-1}$ be the set of symbols populating $\bm{A}$, and let $A_{i,j}$ denote the symbol in the $i$-th row and $j$-th column of $\bm{A}$. Define $c_k:= \Phi^{-1}\{(1 + k\alpha)/2\}$ for all $k=0,\ldots,\alpha^{-1}$. We define the rejection region $R^{\dag}\in\mathbb{R}^3$ corresponding to $\bm{A}$ to be the following union of open hyperrectangles in the nonnegative orthant of $\mathbb{R}^3$:
\begin{equation}
  \label{eq:Latin:square}
R^{\dag}:=\bigcup_{i=1}^{\alpha^{-1}}\bigcup_{j=1}^{\alpha^{-1}}(c_{i-1},c_i)\times(c_{j-1},c_j)\times(c_{A_{i,j}-1},c_{A_{i,j}}).
\end{equation}
The test corresponding to $R^{\dag}$ rejects $H_0^{\ell}$ for $H_1^{\ell}$ if $(\lvert Z_*^1\rvert, \lvert Z_*^2\rvert, \lvert Z_*^3\rvert)^{\top}\in R^{\dag}$. 

\begin{theorem}
\label{thm:latin}
For any unit fraction $\alpha\in(0,1)$ and Latin square of order $\alpha^{-1}$, $\bm{A}$, the corresponding test defined by the rejection region $R^{\dag}$ in \eqref{eq:Latin:square} is a similar test of $H_0^3$ against $H_1^3$ that is symmetric with respect to negations. If $\bm{A}$ is totally symmetric, then the test is also symmetric with respect to permutations.
%If $\alpha$ is a unit fraction, there exists a level-$\alpha$ deterministic similar test of $H_0^3$ that is symmetric with respect to negations. If a totally symmetric Latin square of order $\alpha^{-1}$ exists, then there exists such a test that is also symmetric with respect to permutations.
\end{theorem}

For the test to also be consistent, $\bm{A}$ must satisfy $A_{\alpha^{-1},\alpha^{-1}}=\alpha^{-1}$. This is because under the alternative hypothesis, $\hat{\delta}\rightarrow (\infty,\infty,\infty)^{\top}$ as $n\rightarrow\infty$, hence the power will only converge to one if the hyperrectangle $(\Phi^{-1}(1-\alpha/2),\infty)\times (\Phi^{-1}(1-\alpha/2),\infty)\times (\Phi^{-1}(1-\alpha/2),\infty)$ is included in the rejection region. For any given $\bm{A}$, there exists at least one isotopy satisfying this property, i.e., one can permute the rows, columns, and symbols of $\bm{A}$ to produce another Latin square with this property. The arrangement of the other items in the Latin square are relatively less important, and likely result in varying amounts of power in different parts of the parameter space. It is not clear whether one Latin square yields a test that uniformly dominates all others in terms of power.

The Bayes risk optimal test can in theory also be extended to tests of products of $\ell$ coefficients. The grid in $\mathbb{R}^2$ needs to be expanded to $\mathbb{R}^{\ell}$, and the grid on the null hypothesis space must be expanded to each of the $\mathbb{R}^{\ell-1}$ hyperplanes composing the new composite null hypothesis space. Further, the prior on the coefficients must be expanded to an $\ell$-dimensional distribution. However, the computational burden will clearly escalate rapidly with $\ell$, and may not be feasible even for dimensions of four or possibly even three without access to tremendous computing power and/or memory.

\section{Proofs}
\label{sec:proofs}
\begin{proof}[Proof of Theorem \ref{thm:similar}]
When $\delta^*_y=0$, for any $\delta^*_x\in\mathbb{R}$, 
\begin{align*}
 &\mathrm{Pr}_{\delta^*_x,\delta^*_y}\{(Z^x_*,Z^y_*)\in R_{mm}\}\\
 =&E_{\delta^*_x}\left[\mathrm{Pr}_{\delta^*_y}\left\{(Z^x_*,Z^y_*)\in R_{mm}\mid Z^x_*\right\}\right]\\
 =& E_{\delta^*_x}\left[\sum_{k=1}^{2/\alpha}I\left\{Z_*^x\in(a_{k-1},a_k)\right\}\mathrm{Pr}_{\delta^*_y}\left\{Z_*^y\in(a_{k-1},a_k)\cup (-a_k,-a_{k-1})\mid Z^x_*\right\}\right]\\
 =& \alpha E_{\delta^*_x}\left[\sum_{k=1}^{2/\alpha}I\left\{Z_*^x\in(a_{k-1},a_k)\right\}\right]\\
 =& \alpha.
\end{align*}
 The same holds for any $\delta^*_y\in\mathbb{R}$ when $\delta^*_x=0$ by symmetry. Thus, \eqref{eq:similarity} holds with $R_{mm}$ substituted for $R$, and the test generated by the rejection region $R_{mm}$ is a similar test (because the null hypothesis space is its own boundary).
 %$\mathrm{Pr}_{\delta^*_x,\delta^*_y}\{(Z_*^x,Z_*^y)\in R_{mm}\}=\alpha$ for all $(\delta^*_x,\delta^*_y)\in\{(\delta^*_x,\delta^*_y):\delta^*_x\delta^*_y=0\}$, which is equal to the boundary of the null hypothesis space.
\end{proof}

\begin{proof}[Proof of Theorem \ref{thm:unique}]
For $\delta^*_y=0$, we wish to find an $R$ such that for all $\delta^*_x$, we will have
\begin{align*}
 \alpha &= \mathrm{Pr}_{(\delta^*_x,0)}\{(Z^x_*,Z^y_*)\in R^*\}
 = E_{\delta^*_x}\left[\mathrm{Pr}_{\delta^*_y=0}\left\{(Z^x_*,Z^y_*)\in R^*\mid Z^x_*\right\}\right],
\end{align*}
which implies that
$E_{\delta^*_x}\left[\mathrm{Pr}_{\delta^*_y=0}\left\{(Z^x_*,Z^y_*)\in R^*\mid
 Z^x_*\right\}-\alpha\right]=0$ for all $\delta^*_x$. Since
$Z^x_*\sim \mathcal{N}(\delta^*_x,1)$, which is a full-rank exponential
family, $Z^x_*$ is a complete statistic, and
\begin{align}
\mathrm{Pr}_{\delta^*_y=0}\allowbreak \left\{(Z^x_*,\allowbreak Z^y_*)\allowbreak \in \allowbreak R^*\mid Z^x_*\right\}\allowbreak =\allowbreak \alpha \;\; \text{almost everywhere.} \label{eq:condt1e}
\end{align}

\begin{lemma}
\label{lemma:step}
Suppose for any $k\in\{0,1,\ldots,1/\alpha-1\}$, $f(x)\geq \Phi^{-1}(\frac{1+k\alpha}{2})$ for all $x>\Phi^{-1}(\frac{1+k\alpha}{2})$. Then (i) $f(x)=\Phi^{-1}(\frac{1+k\alpha}{2})$ for all $x\in(\Phi^{-1}(\frac{1+k\alpha}{2}),\Phi^{-1}\{\frac{1+(k+1)\alpha}{2}\})$, %(ii) $f^{-1}(x)=\Phi^{-1}\{\frac{1+(k+1)\alpha}{2}\})$ for all $x\in(\Phi^{-1}(\frac{1+k\alpha}{2}),\Phi^{-1}\{\frac{1+(k+1)\alpha}{2}\})$, 
and (ii) $f(x)\geq \Phi^{-1}\{\frac{1+(k+1)\alpha}{2}\}$ for all $x>\Phi^{-1}\{\frac{1+(k+1)\alpha}{2}\}$.
\end{lemma}
\begin{proof}[Proof of Lemma \ref{lemma:step}]
(i) Suppose $f(x_0)=y_0>\Phi^{-1}(\frac{1+k\alpha}{2})$ for some \[x_0 \in \left(\Phi^{-1}\left(\frac{1+k\alpha}{2}\right), \Phi^{-1}\left\{\frac{1+(k+1)\alpha}{2}\right\}\right),\] and let $f^{-1}$ be the generalized function inverse $f^{-1}(y):=\inf\{x:f(x)\geq y\}$. Then $\Phi^{-1}\{\frac{1+(k+1)\alpha}{2}\}>x_0\geq\inf\{x:f(x)\geq y_0\}=f^{-1}(y_0)$ by the nondecreasing monotonicity of $f^{-1}$. For all $x'\in[\Phi^{-1}(\frac{1+k\alpha}{2}),y_0]$,
\begin{align*}
 \mathrm{Pr}_{\delta_y^*=0}\left\{(Z^x_*, Z^y_*) \in R^*\mid Z^x_*=x'\right\} &= 2\left[\Phi\left\{f^{-1}(x')\right\}-\Phi\left\{f(x')\right\}\right]\\
 &\leq 2\left[\Phi\left\{f^{-1}(y_0)\right\}-\Phi\left\{\Phi^{-1}\left(\frac{1+k\alpha}{2}\right)\right\}\right]\\
 &< 2\left(\Phi\left[\Phi^{-1}\left\{\frac{1+(k+1)\alpha}{2}\right\}\right]-\frac{1+k\alpha}{2}\right)\\
 &= 2\left\{\frac{1+(k+1)\alpha}{2}-\frac{1+k\alpha}{2}\right\}\\
 &= \alpha.
\end{align*}
Since $[\Phi^{-1}(\frac{1+k\alpha}{2}),y_0]$ is a set of positive measure if $y_0>\Phi^{-1}(\frac{1+k\alpha}{2})$, this contradicts \eqref{eq:condt1e}. Thus, $f(x)=\Phi^{-1}(\frac{1+k\alpha}{2})$ for all $x\in(\Phi^{-1}(\frac{1+k\alpha}{2}),\Phi^{-1}\{\frac{1+(k+1)\alpha}{2}\})$.

(ii) By \eqref{eq:condt1e}, for almost every $x\in(\Phi^{-1}(\frac{1+k\alpha}{2}),\Phi^{-1}\{\frac{1+(k+1)\alpha}{2}\})$,
\begin{align*}
 \alpha &= \mathrm{Pr}_{\delta_y^*=0}\left\{(Z^x_*, Z^y_*) \in R^*\mid Z^x_*=x\right\}\\
 &= 2\left[\Phi\left\{f^{-1}(x)\right\}-\Phi\left\{f(x)\right\}\right]\\
 &= 2\left[\Phi\left\{f^{-1}(x)\right\}-\frac{1+k\alpha}{2}\right],
\end{align*}
hence $f^{-1}(x)=\Phi^{-1}\{\frac{1+(k+1)\alpha}{2}\}$ almost everywhere in $(\Phi^{-1}(\frac{1+k\alpha}{2}),\Phi^{-1}\{\frac{1+(k+1)\alpha}{2}\})$. By nondecreasing monotonicity of $f^{-1}$, $f^{-1}(x)=\Phi^{-1}\{\frac{1+(k+1)\alpha}{2}\}$ for all $x\allowbreak \in\allowbreak (\Phi^{-1}(\frac{1+k\alpha}{2}),\allowbreak \Phi^{-1}\{\frac{1+(k+1)\alpha}{2}\})$.

Now suppose $x_1>\Phi^{-1}\{\frac{1+(k+1)\alpha}{2}\}$ and $f(x_1)<\Phi^{-1}\{\frac{1+(k+1)\alpha}{2}\}$. Then for $x''\allowbreak \in\allowbreak (f(x_1),\allowbreak \Phi^{-1}\{\frac{1+(k+1)\alpha}{2}\})$, $f^{-1}(x'')\geq x_1>\Phi^{-1}\{\frac{1+(k+1)\alpha}{2}\}$, which contradicts $f^{-1}(x'')=\Phi^{-1}\{\frac{1+(k+1)\alpha}{2}\}$ since $x''\in(\Phi^{-1}(\frac{1+k\alpha}{2}),\Phi^{-1}\{\frac{1+(k+1)\alpha}{2}\})$. Thus, $f(x)\geq \Phi^{-1}\{\frac{1+(k+1)\alpha}{2}\}$ for all $x\allowbreak >\allowbreak \Phi^{-1}\{\frac{1+(k+1)\alpha}{2}\}$.
\end{proof}
By the constraints on $\mathcal{F}$, $f(x)\geq 0=\Phi^{-1}(1/2)$ for all $x> 0=\Phi^{-1}(1/2)$. Applying Lemma \ref{lemma:step}, for each $k\in\{0,1,\ldots,2/\alpha-1\}$, we have $f(x)=\Phi^{-1}(\frac{1+k\alpha}{2})$ for all $x\in(\Phi^{-1}(\frac{1+k\alpha}{2}),\Phi^{-1}\{\frac{1+(k+1)\alpha}{2}\})$ by induction. That is, $f(x)=a_{k-1}$ for all $x\in (a_{k-1},a_k)$. Additionally, $f(0)=0$ must hold. The only set of points on which we have not defined $f$ is \[\left\{a_{1/\alpha+1},a_{1/\alpha+2},\ldots,a_{2/\alpha}\right\}=\left\{\Phi^{-1}\left(\frac{1+\alpha}{2}\right),\Phi^{-1}\left(\frac{1+2\alpha}{2}\right),\ldots,\Phi^{-1}(1)\right\},\] which is a set of measure zero. Thus, any function $f$ satisfying \eqref{eq:condt1e} and hence generating a similar test must be equal to $f_{mm}$ everywhere except for the above set of measure zero.
\end{proof}

\begin{proof}[Proof of Theorem \ref{thm:impossibility}]
Suppose $R^{\dag}\subset\mathbb{R}^2$ is a rejection region bounded away from $\{(Z^x_*, Z^y_*): Z^x_*Z^y_* = 0\}$. Then there is some $\varepsilon>0$ for which $R^{\dag} \cap \{(Z^x_*, Z^y_*)^{\top}: -\varepsilon<Z^x_*<\varepsilon\} = \emptyset$. For all $Z^x_*\in(-\varepsilon,\varepsilon)$, a set of positive measure, $\mathrm{Pr}\left\{(Z^x_*, Z^y_*)^{\top}\in R^{\dag}\mid Z^x_*\right\}=0$, hence \eqref{eq:condt1e} does not hold for $R^{\dag}$. Thus, $\mathrm{Pr}\left\{(Z^x_*, Z^y_*)^{\top}\in R^{\dag}\right\}=0.05$ cannot hold for all $(\delta_x^*,\delta_y^*)$ in $H_0$, and $R^{\dag}$ does not generate a similar test.
\end{proof}

\begin{proof}[Proof of Theorem \ref{thm:non-unit:fraction}]
    It won't  come as a  surprise that  the proof is  very similar to  that of
    Theorem \ref{thm:similar}.          For         brevity,         introduce
    $B_{k}:=      (b_{k-1},b_k)\cup       (-b_k,-b_{k-1})$      for      every
    $k = 1, \ldots, \lfloor \alpha^{-1}\rfloor$.  By symmetry, we can focus on
    the case  where $\delta^*_y=0$  and $\delta^*_x$  is fixed  arbitrarily in
    $\mathbb{R}$.                           Note                          that
    $\mathrm{Pr}_{\delta^*_y=0}(Z_*^y\in  B_{k}\mid  Z^x_*)  =  \alpha$  for
    every $k=1, \ldots, \lfloor\alpha^{-1}\rfloor$. Therefore, it holds that
    \begin{align*}
      &\mathrm{Pr}_{\delta^*_x,\delta^*_y=0}
        \{(Z^x_*,Z^y_*)\in R_{mm}'\}\\
      &=E_{\delta^*_x}\Bigg[
        \mathrm{Pr}_{\delta^*_y=0}\left\{(Z^x_*,Z^y_*)\in R_{mm}'\mid Z^x_*\right\}\Bigg]\\
      &= E_{\delta^*_x}\Bigg[
        \sum_{k=1}^{\lfloor\alpha^{-1}\rfloor}I\left\{Z_*^x\in
        (b_{k-1},b_{k})\right\}\mathrm{Pr}_{\delta^*_y=0}\left(Z_*^y\in
        B_{k}\mid Z^x_*\right)\\
      &\qquad\qquad + \sum_{k=1}^{\lfloor\alpha^{-1}\rfloor}I\left\{Z_*^x\in
        (-b_{k},-b_{k-1})\right\}\mathrm{Pr}_{\delta^*_y=0}\left(Z_*^y\in
        B_{k}\mid Z^x_*\right)\\
      &\qquad\qquad + I\left\{Z_*^x\in
        (-b_{0},b_{0})\right\}\mathrm{Pr}_{\delta^*_y=0}\left\{Z_*^y\in
        (-b_{0},b_{0})\mid Z^x_*\right\} \Bigg]\\
      &=                                                                 \alpha
        \mathrm{Pr}_{\delta^*_x}
        \left\{Z_*^x\not\in(-b_{0},b_{0})\right\}     +
        \mathrm{Pr}_{\delta^*_x}\left\{Z_*^x\in(-b_{0},b_{0})\right\}
        \mathrm{Pr}_{\delta^*_y=0}\left\{Z_*^y\in(-b_{0},b_{0})\right\}\\
      &= \alpha
        \mathrm{Pr}_{\delta^*_x}
        \left\{Z_*^x\not\in(-b_{0},b_{0})\right\} +
        \mathrm{Pr}_{\delta^*_x}\left\{Z_*^x\in(-b_{0},b_{0})\right\}
        \left(1 - \lfloor \alpha^{-1}\rfloor\alpha\right)\\
      &= \alpha
        +        \mathrm{Pr}_{\delta^*_x}\left\{Z_*^x\in(-b_{0},b_{0})\right\}
        \left(1 - \lfloor \alpha^{-1}\rfloor \alpha - \alpha\right).  
    \end{align*}
    Because $(1  - \lfloor  \alpha^{-1}\rfloor \alpha -  \alpha) \leq  0$, the
    right-hand side  expression is  smaller than $\alpha$,  which is  also the
    limit  as $\delta_{x}^{*}  \to  +\infty$.  Moreover,  the right-hand  side
    expression  is  minimized  at  $\delta_{x}^{*}   =  0$,  where  it  equals
    $\lfloor \alpha^{-1}\rfloor  \alpha^{2} + (1 -  \lfloor \alpha^{-1}\rfloor
    \alpha)^{2}$. This completes the proof.
\end{proof}

\begin{proof}[Derivation of the $p$-value formula]
Set
$z  =  (z_{1},  z_{2})\in  \mathbb{R}^{2}\setminus\{t   \in  \mathbb{R}^{2}  :  |t_{1}|  =
|t_{2}|\}$     and     define      $\bz     =     \max(|z_{1}|,     |z_{2}|)$,
$\zb =  \min(|z_{1}|, |z_{2}|)$.  For  reasons that will  become apparent
  later,                 suppose                 moreover                 that
  $\Phi(\bz) \not\in  \{(2+\ell)/[2(\ell + 1)] :  \ell \in \mathbb{N}\}$. If  $z$ is
  the realization  of a  random variable  $Z$ drawn in  $\mathbb{R}^{2}$ from  a law
  dominated by the Lebesgue measure, then  $z$ meets all the conditions almost
  surely. 

Set  $\ell   \geq  1$  and   $\alpha  \in  ](\ell+1)^{-1},  \ell]$,   so  that
$\lfloor \alpha^{-1}\rfloor  = \ell$. It holds  that $z \in R(\alpha)$  if and
only if (iff)
\begin{enumerate}
\item $\alpha \ell \leq 2(1 - \Phi(\bz))$, or
\item $\alpha \geq 2(1 - \Phi(\zb))$, or
\item $\exists k \in \{1, \ldots, \ell-1\}$ such that
  \begin{equation}
    \label{eq:encadrement}
    \frac{2(1-\Phi(\zb))}{k+1} \leq \alpha \leq \frac{2(1-\Phi(\bz))}{k}.
  \end{equation}
\end{enumerate}
Note that the  three conditions are exclusive. Moreover, there  exists at most
one $k$ such that \eqref{eq:encadrement} holds true. Therefore,
\begin{align}
  \notag
  & \int_{0}^{1} \1\{z \in R(\alpha)\} d\alpha\\
  \label{eq:term1}
  & = \sum_{\ell \geq 1} \int_{(\ell+1)^{-1}}^{\ell^{-1}} \1\left\{\alpha\ell \leq
    2(1 - \Phi(\bz))\right\} d\alpha\\
  \label{eq:term2}
  &    \quad   +    \sum_{\ell   \geq    1}   \int_{(\ell+1)^{-1}}^{\ell^{-1}}
    \1\left\{\alpha \geq 2(1 - \Phi(\zb))\right\} d\alpha\\
  \label{eq:term3}
  &    \quad   +    \sum_{\ell   \geq    1}   \int_{(\ell+1)^{-1}}^{\ell^{-1}}
    \sum_{k=1}^{\ell-1}\1\left\{\frac{2(1-\Phi(\zb))}{k+1}  \leq  \alpha  \leq
    \frac{2(1-\Phi(\bz))}{k}\right\} d\alpha. 
\end{align}
Next,    we   study    the   summands    \eqref{eq:term1},   \eqref{eq:term2},
\eqref{eq:term3} in turn.

Introduce for convenience notation for the $j$th harmonic number $S_{1}(j) = \sum_{i=1}^{j} i^{-1}$
% and $S_{2}(j) = \sum_{i=1}^{j} [i(i+1)]^{-1} = 1 - (j+1)^{-1}$
for any $j \geq 1$. Set also $S_{1}(0)=0$.
\paragraph*{Summand \eqref{eq:term1}.} 
  Define
  \begin{equation}
    \label{eq:L}
    L   =   \left\lfloor  \frac{2(1-\Phi(\bz))}{2\Phi(\bz)-1}\right\rfloor   <
    \frac{2(1-\Phi(\bz))}{2\Phi(\bz)-1}.
  \end{equation}
  The  inequality is  entailed by  the constraints  imposed on  $z$. Moreover,
  $L\geq 1$  iff $\Phi(\bz) \leq  \tfrac{3}{4}$.  Now, observe that,  for any
integer $\ell \geq 1$,
\begin{equation*}
  \frac{1}{\ell+1}  \leq  \frac{2(1  - \Phi(\bz))}{\ell}  \leq  \frac{1}{\ell}
  \text{      iff       }      \ell      \leq      L.
\end{equation*}
It follows that
\begin{align}
  \notag
  \eqref{eq:term1}
  &=  \1\{L  \geq   1\}  \sum_{\ell=1}^{L}  \int  _{(\ell+1)^{-1}}^{\ell^{-1}}
    \1\left\{\alpha\ell \leq 2(1 - \Phi(\bz))\right\} d\alpha\\ 
  \notag
  &=   \1\{L  \geq   1\}  \sum_{\ell=1}^{L}    \left(\frac{2(1     -    \Phi(\bz))}{\ell}     -
    \frac{1}{\ell+1}\right) \\
  \notag
  &= \1\{L \geq 1\} \left[2(1-\Phi(\bz)) S_{1}(L) - S_{1}(L+1) + 1\right]\\
  \notag
  &= \1\{L \geq 1\} \left[1 - (2\Phi(\bz)-1) S_{1}(L) - (L+1)^{-1}\right]\\
  \label{eq:term1:bis}
  &= 1 - (2\Phi(\bz)-1) S_{1}(L) - (L+1)^{-1}. %\qquad \text{(the  indicator can  be removed indeed)}. 
\end{align}

\paragraph*{Summand \eqref{eq:term2}.} It is obvious that
\begin{equation}
  \label{eq:term2:bis}
  \eqref{eq:term2} = 1 - 2(1-\Phi(\zb)) = 2 \Phi(\zb) - 1.
\end{equation}

\paragraph*{Summand \eqref{eq:term3}.} We first note that, if $k \geq 1$ is an
integer, then
\begin{equation*}
  \frac{2(1-\Phi(\zb))}{k+1} \leq \frac{2(1-\Phi(\bz))}{k}  \text{ iff } k
  \leq    K    =    \left\lfloor     \frac{1    -    \Phi(\bz)}{\Phi(\bz)    -
      \Phi(\zb)}\right\rfloor.  
\end{equation*}
Moreover,   $K  \geq   L$  and   $K  \geq   1$  iff   $2\Phi(\bz)  \leq   1  +
\Phi(\zb)$. Consequently,
\begin{align}
  \notag
  \eqref{eq:term3}
  & = \1\{K \geq 1\} \sum_{\ell \geq 1} \int_{(\ell+1)^{-1}}^{\ell^{-1}}
    \sum_{k=1}^{K}  \1\{k  \leq  \ell-1\}  \1\left\{\frac{2(1-\Phi(\zb))}{k+1}
    \leq \alpha \leq \frac{2(1-\Phi(\bz))}{k}\right\} d\alpha\\
  \notag
  &  = \1\{K \geq 1\} \sum_{k=1}^{K}  \sum_{\ell  \geq k+1}  \int_{(\ell+1)^{-1}}^{\ell^{-1}}
    \1\left\{\frac{2(1-\Phi(\zb))}{k+1}        \leq         \alpha        \leq
    \frac{2(1-\Phi(\bz))}{k}\right\} d\alpha\\
  \label{eq:term3:bis}
  & = \1\{K \geq 1\} \sum_{k=1}^{K} \int_{0}^{(k+1)^{-1}} 
    \1\left\{\frac{2(1-\Phi(\zb))}{k+1}        \leq         \alpha        \leq
    \frac{2(1-\Phi(\bz))}{k}\right\} d\alpha.
\end{align}
Now, note that, for any $k\geq 1$,

  \begin{equation*}
    \frac{2(1-\Phi(\bz))}{k} \leq \frac{1}{k+1}\quad
    \text{ iff } \quad k \geq \frac{2(1-\Phi(\bz))}{2\Phi(\bz)-1} \quad
    \text{ iff } \quad k \geq L+1,
  \end{equation*}
  the second equivalence  following from the inequality  in \eqref{eq:L}.  In
light of \eqref{eq:term3:bis}, it thus holds that
% \begin{align}
%   \notag
%   \eqref{eq:term3}
%   &    =   \1\{L    \geq    2\}    \sum_{k=1}^{L-1}   \left(\frac{1}{k+1}    -
%     \frac{2(1-\Phi(\zb))}{k+1}\right) + 
%     \1\{K\geq 1\}\sum_{k=1\vee L}^{K} \left(\frac{2(1-\Phi(\bz))}{k} -
%     \frac{2(1-\Phi(\zb))}{k+1}\right)\\
%   \notag
%   & = \1\{L \geq 2\}(2\Phi(\zb)-1) (S_{1}(L)-1)\\
%   \notag
%   & \quad + 2\1\{K\geq 1\}(1-\Phi(\bz)) [S_{2}(K) - S_{2}(0 \vee (L-1))] \\
%   \label{eq:term3:ter}
%   & \quad - 2\1\{K\geq 1\}
%     (\Phi(\bz) - \Phi(\zb)) [S_{1}(K+1) - S_{1}(1 \vee L)].
%   \end{align}
\begin{align}
  \notag
  \eqref{eq:term3}
  &  = \1\{K  \geq  1 >  L=0\}  \sum_{k=L+1}^{K} \left(\frac{2(1-\Phi(\bz))}{k}  -
    \frac{2(1-\Phi(\zb))}{k+1}\right)\\
  \notag
  & \quad + \1\{L \geq 1\} \sum_{k=1}^{L} \left(\frac{1}{k+1} -
    \frac{2(1-\Phi(\zb))}{k+1}\right)\\
  \notag
  & \quad + \1\{K > L \geq 1\} \sum_{k=L+1}^{K} \left(\frac{2(1-\Phi(\bz))}{k} - 
    \frac{2(1-\Phi(\zb))}{k+1}\right)\\
  \notag
  & = \1\{L \geq 1\} [(2\Phi(\zb)-1) (S_{1}(L+1) - 1)]\\
  \notag
  & \quad + \left[\1\{K \geq 1 > L=0\} + \1\{K > L \geq 1\} \right]
    \sum_{k=L+1}^{K}               \left(\frac{2(1-\Phi(\bz))}{k}              -
    \frac{2(1-\Phi(\zb))}{k+1}\right) \\
  \notag
  &  = (2\Phi(\zb)-1)  (S_{1}(L+1) -  1) \\ %\qquad \text{(the  indicator can  be removed indeed)}\\
  \notag
  & \quad + \1\{K > L\} \sum_{k=L+1}^{K} \left(\frac{2(1-\Phi(\bz))}{k} - 
    \frac{2(1-\Phi(\zb))}{k+1}\right) \\
  \notag
  &  =  (2\Phi(\zb)-1) (S_{1}(L+1) - 1)\\
  \notag 
  & \quad +  \1\{K  >  L\}  \left\{2(1   -  \Phi(\bz))  [S_{1}(K)  -   S_{1}(L)]  -
    2(1-\Phi(\zb)) [S_{1}(K+1) - S_{1}(L+1)]\right\} \\
  \notag
  &  =  (2\Phi(\zb)-1) (S_{1}(L+1) - 1)\\
  \notag 
  & \quad - 2\1\{K > L\} \{(\Phi(\bz) - \Phi(\zb)) [S_{1}(K) - S_{1}(L)]\}\\
  \label{eq:term3:ter}
  & \quad + 2\1\{K > L\} \{(1 - \Phi(\zb)) [(L+1)^{-1} - (K+1)^{-1}]\}.
\end{align}

In view  of \eqref{eq:term1:bis},  \eqref{eq:term2:bis}, \eqref{eq:term3:ter},
we obtain 
  \begin{align*}
    \int_{0}^{1} \1\{z \in R(\alpha)\} d\alpha
    & = 2 \Phi(\zb) - 1\\
    & \quad + 1 - (2\Phi(\bz)-1) S_{1}(L) - (L+1)^{-1}\\ 
    & \quad +  (2\Phi(\zb)-1) (S_{1}(L+1) - 1)\\
    & \quad - 2\1\{K > L\} \{(\Phi(\bz) - \Phi(\zb)) [S_{1}(K) - S_{1}(L)]\}\\
    & \quad + 2\1\{K > L\} \{(1 - \Phi(\zb)) [(L+1)^{-1} - (K+1)^{-1}]\}\\
    & = 1 - 2(1 - \Phi(\zb)) (L+1)^{-1} - 2(\Phi(\bz) - \Phi(\zb)) S_{1}(L)\\
    & \quad - 2\1\{K > L\} \{(\Phi(\bz) - \Phi(\zb)) [S_{1}(K) - S_{1}(L)]\}\\
    & \quad + 2\1\{K > L\} \{(1 - \Phi(\zb)) [(L+1)^{-1} - (K+1)^{-1}]\}\\
    &=  1 -  2(1  - \Phi(\zb))  (K\vee L+1)^{-1}  -  2(\Phi(\bz) -  \Phi(\zb))
      S_{1}(K\vee L)\\
    &=  1 -  2(1  - \Phi(\zb))  (K+1)^{-1}  -  2(\Phi(\bz) -  \Phi(\zb))
      S_{1}(K).
  \end{align*}

We conclude that 
  \begin{equation}
    \label{eq:p-val}
    \int_{0}^{1}  \1\{z  \not\in  R(\alpha)\}   d\alpha  =  2\{1  -  \Phi(\zb)\}
    (K+1)^{-1} + 2\{\Phi(\bz) - \Phi(\zb)\} S_{1}(K).  
  \end{equation}
\end{proof}

\begin{proof}[Proof of Theorem \ref{thm:latin}]
\begin{comment}
First, we will construct a rejection region corresponding to a given Latin square of order $\alpha^{-1}$, $A$. Let $1,\ldots,\alpha^{-1}$ be the set of symbols populating $A$, and let $A_{i,j}$ denote the symbol in the $i$-th row and $j$-th column of $A$. Define $c_0:= 0$ and $c_k:= \Phi^{-1}(k\alpha/2)$ for all $k=1,\ldots,\alpha^{-1}$. We define the rejection region $R^{\dag}\in\mathbb{R}^3$ corresponding to $A$ to be the following union of open hyperrectangles in the nonnegative orthant of $\mathbb{R}^3$:
\[\bigcup_{i=1}^{\alpha^{-1}}\bigcup_{j=1}^{\alpha^{-1}}(c_{i-1},c_i)\times(c_{j-1},c_j)\times(c_{A_{i,j}-1},c_{A_{i,j}}).\]
The test corresponding to $R^{\dag}$ rejects if $\lvert Z_*\rvert\in R^{\dag}$.
\end{comment}

%We show that this produces a similar test of $H_0^3$. 
Let $\lvert Z_*\rvert:= (\lvert Z_*^1\rvert, \lvert Z_*^2\rvert, \lvert Z_*^3\rvert)^{\top}$. When $\delta^*_3=0$, for any $\delta^*_1, \delta^*_2\in\mathbb{R}$, 
\begin{align*}
 &\mathrm{Pr}_{\delta^*}\{\lvert Z_*\rvert\in R^{\dag}\}\\
 &=E_{\delta^*_1,\delta^*_2}\left[\mathrm{Pr}_{\delta^*_3}\left\{\lvert Z_*\rvert\in R^{\dag}\mid Z^1_*,Z^2_*\right\}\right]\\
 &= E_{\delta^*_1,\delta^*_2}\left[\sum_{i=1}^{\alpha^{-1}}\sum_{j=1}^{\alpha^{-1}}I\left\{(\lvert Z_*^1\rvert,\vert Z_*^2\rvert)\in (c_{i-1},c_i)\times(c_{j-1},c_j)\right\}\mathrm{Pr}_{\delta^*_3}\left\{\lvert Z_*^3\rvert\in(c_{A_{i,j}-1},c_{A_{i,j}})\mid Z^1_*, Z^2_*\right\}\right]\\
 &= \alpha E_{\delta^*_1,\delta^*_2}\left[\sum_{i=1}^{\alpha^{-1}}\sum_{j=1}^{\alpha^{-1}}I\left\{(\lvert Z_*^1\rvert,\vert Z_*^2\rvert)\in (c_{i-1},c_i)\times(c_{j-1},c_j)\right\}\right]\\
 &= \alpha.
\end{align*}

%The Latin square property dictates that for a given row and a given symbol of $A$, there is a unique column in $A$ containing that symbol, and likewise that for a given row and a given symbol of $A$, there is a unique row in $A$ containing that symbol. 
We may alternatively represent $A$ in its orthogonal array representation, which is an array with rows corresponding to each entry of $A$ consisting of the triple (row, column, symbol). Any permutation of the coordinates of these triples yields an orthogonal array corresponding to another Latin square known as a conjugate. Let $A^{132}$ be the conjugate of $A$ generated by exchanging the second and third columns of the orthogonal array representation of $A$. The region $R^{\dag}$ can alternatively be expressed as
\[\bigcup_{i=1}^{\alpha^{-1}}\bigcup_{k=1}^{\alpha^{-1}}(c_{i-1},c_i)\times(c_{A^{132}_{i,k}-1},c_{A^{132}_{i,k}})\times(c_{k-1},c_k).\]

When $\delta_2^*=0$, for any $\delta_1^*, \delta_3^*\in\mathbb{R}$, we have the analogous result:
\begin{align*}
 &\mathrm{Pr}_{\delta^*}\{\lvert Z_*\rvert\in R^{\dag}\}\\
 &=E_{\delta^*_1,\delta^*_3}\left[\mathrm{Pr}_{\delta^*_2}\left\{\lvert Z_*\rvert\in R^{\dag}\mid Z^1_*,Z^3_*\right\}\right]\\
 &= E_{\delta^*_1,\delta^*_3}\left[\sum_{i=1}^{2/\alpha}\sum_{k=1}^{2/\alpha}I\left\{(\lvert Z_*^1\rvert,\vert Z_*^3\rvert)\in (c_{i-1},c_i)\times(c_{k-1},c_k)\right\}\mathrm{Pr}_{\delta^*_2}\left\{\lvert Z_*^2\rvert\in(c_{A^{132}_{i,k}-1},c_{A^{132}_{i,k}})\mid Z^1_*, Z^3_*\right\}\right]\\
 &= \alpha E_{\delta^*_1,\delta^*_3}\left[\sum_{i=1}^{2/\alpha}\sum_{k=1}^{2/\alpha}I\left\{(\lvert Z_*^1\rvert,\vert Z_*^3\rvert)\in (c_{i-1},c_i)\times(c_{k-1},c_k)\right\}\right]\\
 &= \alpha.
\end{align*}

Obviously, the same holds for any $\delta^*_2,\delta^*_3\in\mathbb{R}$ when $\delta^*_1=0$ by symmetry. Thus, $\mathrm{Pr}_{\delta^*}\{\lvert Z_*\rvert\in R^{\dag}\}=\alpha$ for all $(\delta^*_1,\delta^*_2,\delta^*_3)\in\{(\delta^*_1,\delta^*_2,\delta^*_3):\delta^*_1\delta^*_2\delta^*_3=0\}$, which is equal to the boundary of the null hypothesis space, so $R^{\dag}$ generates a similar test. 

Clearly, the test is symmetric with respect to negations by construction. A Latin square is totally symmetric if it is equal to all of its conjugates. Thus, if $A$ is totally symmetric, then $R^{\dag}=R^{\sigma}$ for all permutations $\sigma$ of (1,2,3), and the test corresponding to a totally symmetric Latin square is also symmetric with respect to permutations.
\end{proof}

\end{document}